\documentclass[a4paper, reqno]{amsart}
\usepackage{graphicx}
\usepackage[bookmarksnumbered,colorlinks,plainpages]{hyperref}

\usepackage{float}
\usepackage{amsthm}
\usepackage[all,2cell]{xy}
\usepackage{amsfonts}
\usepackage{amsmath}
\usepackage{amssymb}
\usepackage{xypic,leqno,amscd,amssymb,pstricks,latexsym,amsbsy,xypic,mathrsfs,verbatim}
\usepackage{multirow}
\usepackage{booktabs}
\usepackage{makecell}
\usepackage{cases}
\usepackage{appendix}
\usepackage{graphicx,colortbl}

\usepackage{etex}
\usepackage{mathtools}
\newcommand{\Nn}{\mathcal{N}}
\newcommand{\Nak}[2]{\Nn_{#1}(#2)}

\usepackage{tikz-cd}
\usepackage{tabularray}
\UseTblrLibrary{diagbox}
\usepackage{quiver}

\usepackage{tikz}
\usetikzlibrary{positioning,arrows}
\usetikzlibrary{decorations.pathreplacing}

\UseAllTwocells \SilentMatrices
\newtheorem{thm}{Theorem}[section]
\newtheorem{set}[thm]{Setting}
\newtheorem{conven}[thm]{Convention}

\newtheorem{cor}[thm]{Corollary}
\newtheorem{lem}[thm]{Lemma}

\newtheorem{prop}[thm]{Proposition}
\theoremstyle{definition}
\newtheorem{defn}[thm]{Definition}
\newtheorem{exm}[thm]{Example}
\theoremstyle{remark}
\newtheorem{rem}[thm]{\bf Remark}
\numberwithin{equation}{section}

\newtheorem{innercustomthm}{{\bf Theorem}}

\newtheorem{innercustomques}{{\bf Question}}

\begin{document}
	\input xy
	\xyoption{all}
	
	\title[Geometric-Combinatorics in Tilting Theory on WPL (2,2,n)]
	{Geometric-Combinatorial Approaches to Tilting Theory for Weighted Projective Lines}

		\author[Jianmin Chen]{Jianmin Chen}
	\address{School of Mathematical Sciences\\
		Xiamen University\\
		361005, Xiamen, P.R. China}
	\email{chenjianmin@xmu.edu.cn}
	\author[Jinfeng Zhang]{Jinfeng Zhang$^*$}
	\address{School of Mathematical Sciences\\
		Xiamen University\\
		361005, Xiamen, P.R. China}
	\email{zhangjinfeng@stu.xmu.edu.cn}
	
	\makeatletter \@namedef{subjclassname@2020}{\textup{2020} Mathematics Subject Classification} \makeatother
	
	\thanks{$^*$ the corresponding author}
	\subjclass[2020]{14F06, 18E10, 05E10, 16S99, 57M50}
	\date{\today}
	\keywords{weighted projective line; geometric model; equivariantization}
	\maketitle
	
	\begin{abstract}
		We provide a geometric-combinatorial model for the category of coherent sheaves on the
weighted projective line of type $(2,2,n)$ via a cylindrical surface with $n$ marked points on each of its upper and lower boundaries, equipped with an order $2$ self-homeomorphism. A bijection is established between indecomposable sheaves on the weighted projective line and skew-curves on the surface. Moreover, by defining a skew-arc as a self-compatible skew-curve and a pseudo-triangulation as a maximal set of distinct pairwise compatible skew-arcs, we show that pseudo-triangulations correspond bijectively to tilting sheaves. Under this bijection, the flip of a skew-arc within a pseudo-triangulation  coincides with the tilting mutation. As an application, we prove the connectivity of the tilting graph for the category of coherent sheaves.
	\end{abstract}
	
\section{Introduction}
Tilting theory is an important topic in the representation theory of algebras, with its origins traced back to reflection functors \cite{MR530043,MR393065}. Its development has since expanded to various contexts, including the category ${\rm coh}\mbox{-}\mathbb{X}$ of coherent sheaves on a weighted projective lines $\mathbb{X}$, as introduced by Geigle and Lenzing \cite{Geigle1987WeightedCurves}. In the literature, the connectedness of titling graph of ${\rm coh}\mbox{-}\mathbb{X}$
has been extensively investigated from a categorical perspective \cite{MR2608408, MR2146246, Geng2020MutationTilting, MR4232543}. 

On the other hand, representation theory has increasingly embraced geometric models as tools for understanding categorical structures, motivated by advancements in cluster algebras and symplectic geometry. For instance, 
 Caldero, Chapoton and Schiffler \cite{MR2187656} described the cluster category of type $A_n$ by using a regular $(n+3)$-gon. Later, Schiffler \cite{MR2366159} extended this work to type $D_n$ with a punctured regular $n$-gon. Further studies on cluster categories related to triangulated marked surfaces include Br\"ustle-Zhang \cite{MR2870100} (unpunctured case), Qiu-Zhou \cite{MR3705277} (punctured case), and Amiot-Plamondon \cite{MR4289034} (punctured case via group actions). Geometric models have also been developed for abelian categories, such as tube categories \cite{MR2921637}, module categories over gentle algebras \cite{MR4294120}, and skew-gentle algebras \cite{MR4581165}.

Notice that most of the above works establish a one-to-one correspondence between indecomposable objects and curves on marked surfaces, where $\operatorname{Ext}^1$ dimensions correspond to intersection numbers, and triangulations naturally align with tilting objects. Motivated by these work, we concern about constructing geometric  models for the categories of coherent sheaves over weighted projective lines. Recently, \cite{Chen2023GeometricModel} introduced a geometric-combinatorial model for ${\rm coh}\mbox{-}\mathbb{X}(n,n)$ of type $(n,n)$ based on a cylindrical surface $(\mathcal{S}, M)$ with $n$ marked points on both its upper and lower boundaries (corresponding to the case $p = q = n$ in a general framework). And further,  \cite{chen2024geometricmodelvectorbundles} characterized the intrinsic properties of the category ${\rm vect}\mbox{-}\mathbb{X}(2,2,n)$ of vector bundles on $\mathbb{X}(2,2,n)$ via an infinite marked strip. 
In this paper, we try to develop geometric-combinatorial approaches that provide a novel perspective on tilting theory for ${\rm coh}\mbox{-}\mathbb{X}(2,2,n)$.

Let $\operatorname{Aut}(\mathbb{X}(n,n))$ be the automorphism group of $\mathbb{X}(n,n)$ and $\operatorname{Aut}({\rm coh}\mbox{-}\mathbb{X}(n,n))$ be the group of isoclasses of self-equivalences of ${\rm coh}\mbox{-}\mathbb{X}(n,n)$. By \cite{MR3644095}, there exists an isomorphism $\mathbb{X}(n,n)/G \cong \mathbb{X}(2,2,n)$ as weighted Riemann surfaces, where $G$ is a subgroup of $\operatorname{Aut}(\mathbb{X}(n,n))$. Indeed, the group $G$ is generated by the automorphism $\sigma_{1,2}$, which exchanges the two weighted points of $\mathbb{X}(n,n)$. Naturally, the group $G$ can be embedded into $\operatorname{Aut}({\rm coh}\mbox{-}\mathbb{X}(n,n))$ as a subgroup fixing the structure sheaf (compare \cite{Lenzing2000AutomorphismGroup}). Moreover, ${\rm coh}\mbox{-}\mathbb{X}(2,2,n)$ is equivalent to the category $({\rm coh}\mbox{-}\mathbb{X}(n,n))^{G}$ of $G$-equivariant objects in ${\rm coh}\mbox{-}\mathbb{X}(n,n)$, see  Appendix \ref{equivariant}. On the other hand,  there exists an order $2$ self-homeomorphism $\sigma$ on $(\mathcal{S}, M)$ that is compatible with $\sigma_{1,2}$ in a specific sense, see \S \ref{thecyl}. Based on these facts, we propose a geometric-combinatorial model for ${\rm coh}\mbox{-}\mathbb{X}(2,2,n)$ via the marked surface $(\mathcal{S},M)$ equipped with $\sigma$. This model reveals the equivariant relationship between ${\rm coh}\mbox{-}\mathbb{X}(2,2,n)$ and ${\rm coh}\mbox{-}\mathbb{X}(n,n)$.  

According to \cite{Chen2023GeometricModel}, there is a bijection between isoclasses of indecomposable sheaves on $\mathbb{X}(n,n)$ and homotopy classes of curves in a set $\mathbf{C}$ on $(\mathcal{S}, M)$.  Using $\sigma$ and  $\mathbf{C}$, we construct the skew-curves on $(\mathcal{S}, M, \sigma)$, see \S \ref{hybrid}. Let $\widehat{\mathbf{C}}$ be the set of all skew-curves on $(\mathcal{S}, M, \sigma)$. Denote by $\operatorname{ind}\left({\rm coh}\mbox{-}\mathbb{X}(2,2,n)\right)$  the set of isoclasses of indecomposable sheaves in ${\rm coh}\mbox{-}\mathbb{X}(2,2,n)$. We show that:

\begin{thm}(Proposition \ref{correspondence})\label{thmA}
The isoclasses in $\operatorname{ind}\left({\rm coh}\mbox{-}\mathbb{X}(2,2,n)\right)$ can be parameterized by the skew-curves in $\widehat{\mathbf{C}}$, through a bijection \[\widehat{\phi}:\widehat{\mathbf{C}}\rightarrow \operatorname{ind}\left({\rm coh}\mbox{-}\mathbb{X}(2,2,n)\right). \]
The explicit correspondence of $\widehat{\phi}$ is given in Table \ref{table:skew-curves-mapping}.
\end{thm}

Moreover, by introducing the notations of compatibility among skew-curves and skew-arc (see Definition \ref{curvecompatible}) and defining a pseudo-triangulation on $(\mathcal{S}, M, \sigma)$ as a maximal set of distinct, pairwise compatible skew-arcs, we show that each pseudo-triangulation $\Lambda$ on $(\mathcal{S}, M, \sigma)$ corresponds to a tilting object $T=\oplus_{\widehat{\gamma}\in \Lambda}\widehat{\phi}(\widehat{\gamma})$ in ${\rm coh}\mbox{-}\mathbb{X}(2,2,n)$. This extends the classification of all tilting bundles on $\mathbb{X}(2,2,n)$ given in \cite{MR3313495}.

Furthermore, we define the flip of a skew-arc $\widehat{\gamma}$ in a pseudo-triangulation $\Lambda$.  The flip replaces $\widehat{\gamma}$ by a uniquely defined, distinct skew-arc  $\widehat{\mu}_{\Lambda}(\widehat{\gamma})$, such that the pseudo-triangulation $\Lambda$ is transformed into another pseudo-triangulation, see Appendix \ref{foh}. On the other hand, one can apply the mutation $\mu_T(\widehat{\phi}(\widehat{\gamma}))$  locally changing the summand $\widehat{\phi}(\widehat{\gamma})$ of $T$ to yield another tilting sheaf. We show that these two operations are compatible: 
\begin{thm} (Proposition \ref{tilting} and Corollary \ref{mutation and flip})\label{thmB}
Each pseudo-triangulation $\Lambda$ on $(\mathcal{S}, M, \sigma)$ corresponds to a tilting object $T$ in ${\rm coh}\mbox{-}\mathbb{X}(2,2,n)$. Specifically, the correspondence is given by $\Lambda \mapsto \oplus_{\widehat{\gamma}\in \Lambda}\widehat{\phi}(\widehat{\gamma})$. Moreover, if $\widehat{\gamma}\in \Lambda$, then 
\[
\widehat{\phi}(\widehat{\mu}_{\Lambda}(\widehat{\gamma})) = \mu_T(\widehat{\phi}(\widehat{\gamma})).
\]
\end{thm}

We then investigate the connectivity of the tilting graph $\mathcal{G}(\mathcal{T}_{\mathbb{X}(2,2,n)})$ of the category ${\rm coh}\mbox{-}\mathbb{X}(2,2,n)$. The vertices of this graph are the isoclasses of basic tilting sheaves of ${\rm coh}\mbox{-}\mathbb{X}(2,2,n)$, while two vertices $T$ and $T^\prime$ are connected by an edge if and only if they differ by precisely one indecomposable direct summand. We establish that any two pseudo-triangulation  on $(\mathcal{S}, M, \sigma)$ can be transformed into each other by a sequence of flips of skew-arcs.  Consequently, by Theorem \ref{thmB},  we have the following result on the  connectivity of $\mathcal{G}(\mathcal{T}_{\mathbb{X}(2,2,n)})$: 

\begin{thm}(Theorem \ref{connected})\label{thmC}
	The tilting graph $\mathcal{G}(\mathcal{T}_{\mathbb{X}(2,2,n)})$ is connected.
\end{thm}

\subsection{Context}The paper is organized as follows. 
In \S \ref{Preliminaries},  we collect the fundamental properties of the coherent sheaf categories on weighted projective lines.
In \S \ref{geometricmodel}, we recall the marked cylindrical surface  model $(\mathcal{S},M)$ for ${\rm coh}\mbox{-}\mathbb{X}(n,n)$ given in \cite{Chen2023GeometricModel}. We introduce an order $2$ self-homeomorphism $\sigma$ on $(\mathcal{S}, M)$ and define skew-curves on the extended structure $(\mathcal{S}, M, \sigma)$. Using this structure, we construct a geometric model for ${\rm coh}\mbox{-}\mathbb{X}(2, 2, n)$ and provide an illustrative example.
In \S \ref{GCTTT}, we establish a correspondence between tilting objects and pseudo-triangulations on $(\mathcal{S},M,\sigma)$ and show that the flip of a skew-arc is compatible with the tilting mutation. Building on this, \S \ref{connectedness} focuses on proving the connectivity of the tilting graph via geometric-combinatorial methods.  In Appendix \ref{equivariant} we illustrate the equivariant relationship between ${\rm coh}\mbox{-}\mathbb{X}(2,2,n)$ and ${\rm coh}\mbox{-}\mathbb{X}(n,n)$ and in Appendix \ref{foh} we discuss the flip of a skew-arc with respect to a  pseudo-triangulation on $(\mathcal{S},M,\sigma)$.

\subsection{Conventions}
Throughout the paper, let $\mathbf{k}$ be an algebraically closed field with characteristic not equal to $2$. The cardinality of a set $I$ is denoted by $|I|$. For a category $\mathcal{C}$, $\operatorname{ind}\mathcal{C}$ denotes the set of isomorphism classes of indecomposable objects in $\mathcal{C}$. Additionally, for the symbols $+$ and $-$, we define $\neg + = -$ and $\neg - = +$.

\section{Preliminaries}\label{Preliminaries}
In this section, we collect basic facts on weighted
projective lines from \cite{Geigle1987WeightedCurves, Lenzing2011WeightedProjective}.

Let $t\geq 0$ be an integer and ${\bf p}=(p_1, p_2, \cdots, p_t)$ be a sequence	of integers with each $p_i\geq 2$. Assume ${\boldsymbol\lambda}=(\lambda_1, \lambda_2, \cdots, \lambda_t)$ is a sequence of
pairwise distinct points on the projective line $\mathbb{P}_{\mathbf k}^1$. A \emph{weighted projective line} $\mathbb{X}:=\mathbb{P}_{\mathbf k}^1(\mathbf{p}; {\boldsymbol\lambda})$ of weight type $\mathbf{p}$ and parameter sequence ${\boldsymbol\lambda}$ is obtained from $\mathbb{P}_{\mathbf k}^1$ by attaching the weight $p_i$ to each point $\lambda_i$ for $1\leq i\leq t$. By normalizing ${\boldsymbol\lambda}$, we can always assume that $\lambda_1=\infty$, $\lambda_2=0$ and $\lambda_3=1$. 

	The \emph{string group} $\mathbb{L}({\bf p})$ is a rank one additive abelian group on generators $\vec{x}_1$, $\vec{x}_2$, $\cdots$, $\vec{x}_t$, subject to the relations $p_1\vec{x}_1=p_2\vec{x}_2=\cdots = p_t\vec{x}_t=: \vec{c}$, where $\vec{c}$ is called the \emph{canonical element} of $\mathbb{L}(\mathbf{p})$. Each element $\vec{x}$ in $\mathbb{L}(\mathbf{p})$ can be uniquely written in its \emph{normal form}
	\begin{align}\label{equ:nor}
		\vec{x}=\sum_{i=1}^t l_i\vec{x}_i+l\vec{c},
	\end{align}
	where $0\leq l_i\leq p_i-1$ for $1\leq i\leq t$ and $l\in \mathbb{Z}$. 
	
	The \emph{homogeneous coordinate algebra} $S(\mathbf{p}; {\boldsymbol\lambda})$ of the weighted projective line $\mathbb{X}$ is given by $\mathbf{k}[X_1, X_2, \cdots, X_t]/I$, where the ideal $I$ is generated by $X_i^{p_i}-(X_2^{p_2}-\lambda_iX_1^{p_1})$ for $3\leq i\leq t$. We write $x_i=X_i+I$ in $S(\mathbf{p}; {\boldsymbol\lambda})$.
The algebra $S(\mathbf{p}; {\boldsymbol\lambda})$ is $\mathbb{L}(\mathbf{p})$-graded by means of $\deg (x_i)=\vec{x}_i$.

Denote by ${\rm gr}\mbox{-}S(\mathbf{p}; {\boldsymbol\lambda})$ the abelian category of finitely generated $\mathbb{L}(\mathbf{p})$-graded $S(\mathbf{p}; {\boldsymbol\lambda})$-modules, and by ${\rm gr}_0\mbox{-}S(\mathbf{p}; {\boldsymbol\lambda})$ its Serre subcategory formed by finite dimensional modules.  By \cite[Theorem 1.8]{Geigle1987WeightedCurves}, the sheafification functor yields an equivalence
$$
 {\rm qgr}\mbox{-}S(\mathbf{p}; {\boldsymbol\lambda})\stackrel{\sim}\longrightarrow {\rm coh}\mbox{-}\mathbb{X}, 
$$
where 
 ${\rm qgr}\mbox{-}S(\mathbf{p}; {\boldsymbol\lambda})$ is the quotient abelian category ${\rm gr}\mbox{-}S(\mathbf{p}; {\boldsymbol\lambda})/{{\rm gr}_0\mbox{-}S(\mathbf{p}; {\boldsymbol\lambda})}$
and ${\rm coh}\mbox{-}\mathbb{X}$ is the category  of coherent sheaves on $\mathbb{X}$. From now on, we will identify these two categories.  Denote by $\mathcal{O}_{\mathbb{X}}$ the image of $S(\mathbf{p}; {\boldsymbol\lambda})$ under the sheafification functor, which we call the \emph{structure sheaf} of $\mathbb{X}$.  

In what follows, we abbreviate ${\rm Hom}_{{\rm coh}\mbox{-}\mathbb{X}}(-,-)$ as ${\rm Hom}_{\mathbb{X}}(-,-)$. Similarly, we abbreviate ${\rm Ext}^{1}_{{\rm coh}\mbox{-}\mathbb{X}}(-,-)$ as ${\rm Ext}^{1}_{\mathbb{X}}(-,-)$.

\begin{prop}(\cite{Geigle1987WeightedCurves, Lenzing2011WeightedProjective})\label{the properties of coherent sheaves}
The category ${\rm coh}\mbox{-}\mathbb{X}$ is connected, ${\rm Hom}$-finite and $\mathbf{k}$-linear with the following properties:
\begin{itemize}
    \item[(a)] ${\rm coh}\mbox{-}\mathbb{X}$ satisfies Serre duality in the form ${\rm Ext}_{\mathbb{X}}^{1}(X,Y)\cong D{\rm Hom}_{\mathbb{X}}(Y, \tau(X)),$ where the $\mathbf{k}$-equivalence $\tau:{\rm coh}\mbox{-}\mathbb{X}\longrightarrow {\rm coh}\mbox{-}\mathbb{X}$ is the shift $X\mapsto X(\vec{\omega})$ with the dualizing element $\vec{\omega}=(t-2)\vec{c}-\Sigma_{i=1}^t \vec{x}_i \in \mathbb{L}(\mathbf{p}).$

\item[(b)] ${\rm coh}\mbox{-}\mathbb{X}={\rm vect}\mbox{-}\mathbb{X}\vee{\rm coh}_{0}\mbox{-}\mathbb{X},$ where ${\rm vect}\mbox{-}\mathbb{X}$ denotes the full subcategory of ${\rm coh}\mbox{-}\mathbb{X}$ consisting of vector bundles on $\mathbb{X}$, and ${\rm coh}_{0}\mbox{-}\mathbb{X}$ denotes the full subcategory of ${\rm coh}\mbox{-}\mathbb{X}$ consisting of all objects of finite length, $\vee$ means that each indecomposable object of ${\rm coh}\mbox{-}\mathbb{X}$ is either in ${\rm vect}\mbox{-}\mathbb{X}$ or in ${\rm coh}_{0}\mbox{-}\mathbb{X}$, and there are no non-zero morphism from ${\rm coh}_{0}\mbox{-}\mathbb{X}$ to ${\rm vect}\mbox{-}\mathbb{X}.$

\item[(c)] For any $\vec{x}, \vec{y}\in\mathbb{L}(\mathbf{p})$, there has
$${\rm Hom}_{\mathbb{X}}(\mathcal{O}_{\mathbb{X}}(\vec{x}), \mathcal{O}_{\mathbb{X}}(\vec{y}))\cong S(\mathbf{p}; {\boldsymbol\lambda})_{\vec{y}-\vec{x}}.$$
In particular, ${\rm Hom}_{\mathbb{X}}(\mathcal{O}_{\mathbb{X}}(\vec{x}), \mathcal{O}_{\mathbb{X}}(\vec{y}))\neq 0$ if and only if $\vec{x}\leq \vec{y}.$

\item[(d)] Denote by $K_{0}({\rm coh}\mbox{-}\mathbb{X})$ the \emph{Grothendieck group} of the category ${\rm coh}\mbox{-}\mathbb{X}$. Then the classes $[\mathcal{O}_{\mathbb{X}}(\vec{x})] \; (0\leq \vec{x}\leq \vec{c})$ form a $\mathbb{Z}$-basis of $K_{0}({\rm coh}\mbox{-}\mathbb{X})$. Moreover, the rank of $K_{0}({\rm coh}\mbox{-}\mathbb{X})$ equals $\Sigma_{i=1}^t p_i+2-t.$
\end{itemize}
\end{prop}
By \cite[Proposition 1.1]{Lenzing2006HereditaryCategories}, the torsion subcategory ${\rm coh}_0\mbox{-}\mathbb{X}$ decomposes into a coproduct
	$\coprod_{\lambda\in \mathbb{X}}\mathcal{U}_{\lambda}$, where $\mathcal{U}_{\lambda}$ for each $\lambda\in \mathbb{X}$ is a connected uniserial length category whose associated Auslander-Reiten quiver is a stable tube $\mathbb{ZA}_{\infty}/(\tau^r)$  for some $r\in \mathbb{Z}_{\geq 1}$. Here, the integral $r$, which is called the $\tau$-period of the stable tube $\mathbb{ZA}_{\infty}/(\tau^r)$, depends on $\lambda$. Precisely, 
 $$r = \begin{cases} 
p_i & \text{if } \lambda = \lambda_i; \\
1 & \text{if } \lambda \in \mathbb{X} \setminus \{p_1, p_2, \cdots, p_t\}.
\end{cases}$$
 A stable tube with $\tau$-period $1$ is called a \emph{homogeneous stable tube}. Objects that lie at the bottom of the stable tubes are all simple objects of ${\rm coh}\mbox{-}\mathbb{X}$. Each $\lambda \in \mathbb{X} \setminus \{\lambda_1, \lambda_2, \cdots, \lambda_t\}$ is associated with a unique simple sheaf $S_{\lambda}$, called \emph{ordinary simple}; while $\lambda=\lambda_i\; ( 1\leq i\leq t)$ is associated with $p_i$ simple objects respectively, denoted by $S_{\lambda_i,\overline{k}}$ where $k\in \mathbb{Z}$ and  $\overline{k}$ denotes the reduction of $k$ modulo $p_i$, called \emph{exceptional simples}.
	
	According to \cite[Proposition 2.5]{Geigle1987WeightedCurves}, there are some important short exact sequences in ${\rm coh}\mbox{-}\mathbb{X}$.  For each ordinary simple sheaf $S_{\lambda}$, there is an
	exact sequence
	$$0\longrightarrow \mathcal{O}_{\mathbb{X}}\stackrel{u_{\lambda}}{\longrightarrow}\mathcal{O}_{\mathbb{X}}(\vec{c})\longrightarrow S_{\lambda}\longrightarrow 0,$$
	where the homomorphism $u_{\lambda}$ is given by
	multiplication with $x_2^{p_2}-\lambda x_1^{p_1}$. For each exceptional simple $S_{\lambda_i,\overline{k}}$, there is an exact sequences
	$$0\longrightarrow \mathcal{O}_{\mathbb{X}}(k\vec{x}_i)\stackrel{u_{\lambda_i}}{\longrightarrow} \mathcal{O}_{\mathbb{X}}((k+1)\vec{x}_i)\longrightarrow S_{\lambda_i,\overline{k}}\longrightarrow 0,$$	
	where $u_{\lambda_i}$ is given by multiplication with $x_i$.
	As is easily checked, for each $\vec{x}=\sum_{i=1}^t l_i\vec{x}_i \in \mathbb{L}(\mathbf{p})$, we have 
	$$S_{\lambda}(\vec{x})=S_{\lambda}\;\;{\rm for}\;\, \lambda \in \mathbb{X}\setminus \{\lambda_1,\lambda_2,\cdots,\lambda_t\}$$
	and $$S_{\lambda_i,\overline{k}}(\vec{x})=S_{\lambda_i, \overline{k+l_{i}}}\;\; {\rm for}\;\, 1 \leq i \leq t.$$

In the following, we denote the weighted projective lines of type $(2,2,n)$ and $(n,n)$ by $\mathbb{X}$ and $\mathbb{Y}$, respectively. Our focus is on the category ${\rm coh}\mbox{-}\mathbb{X}$ and the equivariant relationship between it and ${\rm coh}\mbox{-}\mathbb{Y}$. For more details on these categories and the equivariant relationship, see Appendix \ref{equivariant}. 
	
\section{Geometric model for the category ${\rm coh}\mbox{-}\mathbb{X}$}\label{geometricmodel}
In this section, we aim to construct a geometric model for ${\rm coh}\mbox{-}\mathbb{X}$ using a cylindrical surface with $n$ marked points on both its upper and lower boundaries, equipped with an order $2$ self-homeomorphism. Without the self-homeomorphism, such a cylindrical surface has previously been used in a geometric model for ${\rm coh}\mbox{-}\mathbb{Y}$; see \cite{Chen2023GeometricModel}. Here, we will recall a slightly modified version of this construction. 

Building on the equivariant relationship between ${\rm coh}\mbox{-}\mathbb{X}$ and ${\rm coh}\mbox{-}\mathbb{Y}$ (see Appendix \ref{equivariant}), we introduce an order $2$ self-homeomorphism on the marked cylindrical surface. We then define skew-curves on this extended structure and establish a correspondence between isoclasses of  indecomposable coherent sheaves on $\mathbb{X}$ and these skew-curves. As a first application of this correspondence, we investigate the behavior of the $\mathbb{L}(2,2,n)$-action on ${\rm vect}\mbox{-}\mathbb{X}$ within this geometric framework.

\subsection{A marked cylindrical surface with an order $2$ self-homeomorphism}\label{thecyl}

We adopt the notation and terminology for marked surfaces as described in \cite{MR2448067}. A marked surface $(\mathcal{S}, M)$ consists of:
\begin{itemize}
    \item $\mathcal{S}$: a connected, oriented, 2-dimensional Riemann surface with non-empty boundary $\partial\mathcal{S}$;
    \item $M$: a finite set of marked points on the closure of $\mathcal{S}$, with those in the interior of $\mathcal{S}$ called \emph{punctures}.
\end{itemize} 

A \emph{curve} on $(\mathcal{S}, M)$ is a continuous map $\gamma: [0,1] \to \mathcal{S}$ such that 
\begin{itemize}
    \item $\gamma(t)\in \mathcal{S}\setminus(M\cup\partial\mathcal{S})$ for any $t\in (0,1)$ and $\gamma(0), \gamma(1) \in M$;
    \item $\gamma$ is not contractible into $M$ or onto the boundary $\partial\mathcal{S}$.
\end{itemize}

The \emph{inverse} of a curve $\gamma$ is defined as $\gamma^{-1}:=\gamma(1-t)$ for $t \in[0, 1]$. In particular, if $\gamma_{|(0,1)}$ is injective, the curve is called an \emph{arc}. Unless otherwise specified, curves are considered up to inverse and isotopy, with endpoints fixed.

Two arcs in $(\mathcal{S}, M)$ are \emph{compatible} if they have representatives in their isotopy classes that do not intersect except at their endpoints. A maximal set of distinct, pairwise compatible arcs forms a \emph{triangulation}. Such a triangulation has
\begin{equation}\label{formula}
   \mathcal{N} = 6g + 3b + 3p + c - 6
\end{equation}
arcs, where $g$ is the genus of $\mathcal{S}$, $b$ the number of boundary components, $p$ the number of punctures, and $c$ the number of boundary-marked points.

In the following, consider the marked surface $(\mathcal{S}, M)$ with the following setup:
\begin{set}
Let $\mathcal{S} = S^1 \times [0, 1]$ be a cylindrical surface, with the set of marked points 
\[
M = \left\{ \frac{2\pi i}{n} \mid 0 \leq i \leq n-1, i \in \mathbb{Z} \right\} \times \{0, 1\}.
\]
To orient $\mathcal{S}$, we embed it in the $(\theta, h)$-plane via the map $\eta: (\theta, h) \mapsto \left( \frac{n \theta}{2 \pi}, h \right)$, adopting the orientation induced by the standard orientation of the plane. 
\end{set}

Consider the infinite strip $\widetilde{\mathcal{S}} := \mathbb{R} \times [0, 1]$, with the orientation inherited from its embedding in the $(\theta, h)$-plane. The universal cover $(\widetilde{\mathcal{S}}, \pi)$ of $\mathcal{S}$ is defined by the covering map
\[
\pi: (x, y) \mapsto \eta^{-1}_{|\operatorname{Im}(\eta)}\left((x \; \text{mod}\;n, y)\right),
\]
where $\eta^{-1}_{|\operatorname{Im}(\eta)}$ is the inverse of $\eta$, mapping from $\operatorname{Im}(\eta)$ to $\mathcal{S}$.
\begin{figure}[H]

\tikzset{every picture/.style={line width=0.75pt}}   

\begin{tikzpicture}[x=0.75pt,y=0.75pt,yscale=-1,xscale=1]

\draw [line width=1.5]    (37.26,143.25) -- (281.68,143.25) ;
  
\draw [line width=1.5]    (37.58,93.25) -- (282,93.25) ;
  
\draw   (458.43,60.07) -- (458.43,156.62) .. controls (458.43,162.82) and (441.68,167.84) .. (421.01,167.84) .. controls (400.35,167.84) and (383.6,162.82) .. (383.6,156.62) -- (383.6,60.07) .. controls (383.6,53.87) and (400.35,48.84) .. (421.01,48.84) .. controls (441.68,48.84) and (458.43,53.87) .. (458.43,60.07) .. controls (458.43,66.27) and (441.68,71.29) .. (421.01,71.29) .. controls (400.35,71.29) and (383.6,66.27) .. (383.6,60.07) ;
  
\draw    (383.6,60.07) ;
\draw [shift={(383.6,60.07)}, rotate = 0] [color={rgb, 255:red, 0; green, 0; blue, 0 }  ][fill={rgb, 255:red, 0; green, 0; blue, 0 }  ][line width=0.75]      (0, 0) circle [x radius= 1.34, y radius= 1.34]   ;
  
\draw    (407.33,70.67) ;
\draw [shift={(407.33,70.67)}, rotate = 0] [color={rgb, 255:red, 0; green, 0; blue, 0 }  ][fill={rgb, 255:red, 0; green, 0; blue, 0 }  ][line width=0.75]      (0, 0) circle [x radius= 1.34, y radius= 1.34]   ;
  
\draw    (394.33,52.33) ;
\draw [shift={(394.33,52.33)}, rotate = 0] [color={rgb, 255:red, 0; green, 0; blue, 0 }  ][fill={rgb, 255:red, 0; green, 0; blue, 0 }  ][line width=0.75]      (0, 0) circle [x radius= 1.34, y radius= 1.34]   ;
  
\draw    (394,68) ;
\draw [shift={(394,68)}, rotate = 0] [color={rgb, 255:red, 0; green, 0; blue, 0 }  ][fill={rgb, 255:red, 0; green, 0; blue, 0 }  ][line width=0.75]      (0, 0) circle [x radius= 1.34, y radius= 1.34]   ;
  
\draw    (393.5,164) ;
\draw [shift={(393.5,164)}, rotate = 0] [color={rgb, 255:red, 0; green, 0; blue, 0 }  ][fill={rgb, 255:red, 0; green, 0; blue, 0 }  ][line width=0.75]      (0, 0) circle [x radius= 1.34, y radius= 1.34]   ;
  
\draw    (383.6,156.62) ;
\draw [shift={(383.6,156.62)}, rotate = 0] [color={rgb, 255:red, 0; green, 0; blue, 0 }  ][fill={rgb, 255:red, 0; green, 0; blue, 0 }  ][line width=0.75]      (0, 0) circle [x radius= 1.34, y radius= 1.34]   ;
  
\draw    (406.67,167) ;
\draw [shift={(406.67,167)}, rotate = 0] [color={rgb, 255:red, 0; green, 0; blue, 0 }  ][fill={rgb, 255:red, 0; green, 0; blue, 0 }  ][line width=0.75]      (0, 0) circle [x radius= 1.34, y radius= 1.34]   ;
  
\draw    (394,147.33) ;
\draw [shift={(394,147.33)}, rotate = 0] [color={rgb, 255:red, 0; green, 0; blue, 0 }  ][fill={rgb, 255:red, 0; green, 0; blue, 0 }  ][line width=0.75]      (0, 0) circle [x radius= 1.34, y radius= 1.34]   ;
  
\draw  [dash pattern={on 0.84pt off 2.51pt}]  (245.34,93.25) -- (245.02,143.25) ;
  
\draw  [dash pattern={on 0.84pt off 2.51pt}]  (101.65,148.32) -- (132.57,148.63) ;
  
\draw  [dash pattern={on 0.84pt off 2.51pt}]  (101.42,88.52) -- (132.34,88.85) ;
  
\draw    (37.58,93.25) -- (74.25,93.25) ;
  
\draw    (86.47,93.25) -- (123.13,93.25) ;
  
\draw    (123.13,93.25) -- (159.79,93.25) ;
  
\draw    (172.01,93.25) -- (208.68,93.25) ;
  
\draw    (208.68,93.25) -- (245.34,93.25) ;
  
\draw    (245.34,93.25) -- (282,93.25) ;
\draw [shift={(245.34,93.25)}, rotate = 0] [color={rgb, 255:red, 0; green, 0; blue, 0 }  ][fill={rgb, 255:red, 0; green, 0; blue, 0 }  ][line width=0.75]      (0, 0) circle [x radius= 1.34, y radius= 1.34]   ;
  
\draw    (37.26,143.25) -- (73.92,143.25) ;
  
\draw    (73.92,143.25) -- (86.15,143.25) ;
\draw [shift={(86.15,143.25)}, rotate = 0] [color={rgb, 255:red, 0; green, 0; blue, 0 }  ][fill={rgb, 255:red, 0; green, 0; blue, 0 }  ][line width=0.75]      (0, 0) circle [x radius= 1.34, y radius= 1.34]   ;
\draw [shift={(73.92,143.25)}, rotate = 0] [color={rgb, 255:red, 0; green, 0; blue, 0 }  ][fill={rgb, 255:red, 0; green, 0; blue, 0 }  ][line width=0.75]      (0, 0) circle [x radius= 1.34, y radius= 1.34]   ;
  
\draw    (74.25,93.25) -- (86.47,93.25) ;
\draw [shift={(86.47,93.25)}, rotate = 0] [color={rgb, 255:red, 0; green, 0; blue, 0 }  ][fill={rgb, 255:red, 0; green, 0; blue, 0 }  ][line width=0.75]      (0, 0) circle [x radius= 1.34, y radius= 1.34]   ;
\draw [shift={(74.25,93.25)}, rotate = 0] [color={rgb, 255:red, 0; green, 0; blue, 0 }  ][fill={rgb, 255:red, 0; green, 0; blue, 0 }  ][line width=0.75]      (0, 0) circle [x radius= 1.34, y radius= 1.34]   ;
  
\draw    (86.15,143.25) -- (122.81,143.25) ;
  
\draw    (122.81,143.25) -- (159.47,143.25) ;
  
\draw    (171.69,143.25) -- (208.35,143.25) ;
  
\draw    (208.35,143.25) -- (245.02,143.25) ;
  
\draw    (245.02,143.25) -- (281.68,143.25) ;
\draw [shift={(245.02,143.25)}, rotate = 0] [color={rgb, 255:red, 0; green, 0; blue, 0 }  ][fill={rgb, 255:red, 0; green, 0; blue, 0 }  ][line width=0.75]      (0, 0) circle [x radius= 1.34, y radius= 1.34]   ;
  
\draw  [dash pattern={on 0.84pt off 2.51pt}]  (74.25,93.25) -- (73.92,143.25) ;
  
\draw  [dash pattern={on 0.84pt off 2.51pt}]  (159.79,93.25) -- (159.47,143.25) ;
  
\draw  [dash pattern={on 0.84pt off 2.51pt}]  (188.97,87.77) -- (219.89,88.1) ;
  
\draw  [dash pattern={on 0.84pt off 2.51pt}]  (192.36,148.02) -- (223.28,148.35) ;
  
\draw  [dash pattern={on 0.84pt off 2.51pt}]  (30.98,117.27) -- (61.9,117.6) ;
  
\draw  [dash pattern={on 0.84pt off 2.51pt}]  (262.98,117.27) -- (293.9,117.6) ;
  
\draw    (159.79,93.25) -- (172.01,93.25) ;
\draw [shift={(172.01,93.25)}, rotate = 0] [color={rgb, 255:red, 0; green, 0; blue, 0 }  ][fill={rgb, 255:red, 0; green, 0; blue, 0 }  ][line width=0.75]      (0, 0) circle [x radius= 1.34, y radius= 1.34]   ;
\draw [shift={(159.79,93.25)}, rotate = 0] [color={rgb, 255:red, 0; green, 0; blue, 0 }  ][fill={rgb, 255:red, 0; green, 0; blue, 0 }  ][line width=0.75]      (0, 0) circle [x radius= 1.34, y radius= 1.34]   ;
  
\draw    (159.47,143.25) -- (171.69,143.25) ;
\draw [shift={(171.69,143.25)}, rotate = 0] [color={rgb, 255:red, 0; green, 0; blue, 0 }  ][fill={rgb, 255:red, 0; green, 0; blue, 0 }  ][line width=0.75]      (0, 0) circle [x radius= 1.34, y radius= 1.34]   ;
\draw [shift={(159.47,143.25)}, rotate = 0] [color={rgb, 255:red, 0; green, 0; blue, 0 }  ][fill={rgb, 255:red, 0; green, 0; blue, 0 }  ][line width=0.75]      (0, 0) circle [x radius= 1.34, y radius= 1.34]   ;
  
\draw    (306.25,116.57) -- (359.25,116.57) ;
\draw [shift={(361.25,116.57)}, rotate = 180] [color={rgb, 255:red, 0; green, 0; blue, 0 }  ][line width=0.75]    (6.56,-2.94) .. controls (4.17,-1.38) and (1.99,-0.4) .. (0,0) .. controls (1.99,0.4) and (4.17,1.38) .. (6.56,2.94)   ;
  
\draw    (406.8,49.6) ;
\draw [shift={(406.8,49.6)}, rotate = 0] [color={rgb, 255:red, 0; green, 0; blue, 0 }  ][fill={rgb, 255:red, 0; green, 0; blue, 0 }  ][line width=0.75]      (0, 0) circle [x radius= 1.34, y radius= 1.34]   ;
  
\draw    (405.6,144.4) ;
\draw [shift={(405.6,144.4)}, rotate = 0] [color={rgb, 255:red, 0; green, 0; blue, 0 }  ][fill={rgb, 255:red, 0; green, 0; blue, 0 }  ][line width=0.75]      (0, 0) circle [x radius= 1.34, y radius= 1.34]   ;
  
\draw  [line width=1.5]  (383.6,60.07) .. controls (383.6,53.87) and (400.35,48.84) .. (421.01,48.84) .. controls (441.68,48.84) and (458.43,53.87) .. (458.43,60.07) .. controls (458.43,66.27) and (441.68,71.29) .. (421.01,71.29) .. controls (400.35,71.29) and (383.6,66.27) .. (383.6,60.07) -- cycle ;
  
\draw [line width=1.5]    (383.6,156.62) .. controls (382.13,170.05) and (456.05,172.55) .. (458.43,156.62) ;
  
\draw  [dash pattern={on 0.84pt off 2.51pt}]  (422.92,148.22) .. controls (453.44,147.51) and (467.44,162.71) .. (421.74,164.24) ;
  
\draw  [dash pattern={on 0.84pt off 2.51pt}]  (423.72,52.22) .. controls (454.24,51.51) and (468.24,66.71) .. (422.54,68.24) ;
  
\draw  [dash pattern={on 0.84pt off 2.51pt}]  (383.26,156.62) .. controls (391.47,136.73) and (456.8,141.06) .. (458.43,156.62) ;

\draw (69.47,79.4) node [anchor=north west][inner sep=0.75pt]  [font=\tiny]  {$0_{\partial }$};
  
\draw (68.03,152.15) node [anchor=north west][inner sep=0.75pt]  [font=\tiny]  {$0_{\partial ^{\prime }}$};
  
\draw (83.82,79.65) node [anchor=north west][inner sep=0.75pt]  [font=\tiny]  {$1_{\partial }$};
  
\draw (83.75,152.15) node [anchor=north west][inner sep=0.75pt]  [font=\tiny]  {$1_{\partial ^{\prime }}$};
  
\draw (154.94,79.15) node [anchor=north west][inner sep=0.75pt]  [font=\tiny]  {$n_{\partial }$};
  
\draw (154.29,152.65) node [anchor=north west][inner sep=0.75pt]  [font=\tiny]  {$n_{\partial ^{\prime }}$};
  
\draw (238.73,79.4) node [anchor=north west][inner sep=0.75pt]  [font=\tiny]  {$2n_{\partial }$};
  
\draw (239.56,152.15) node [anchor=north west][inner sep=0.75pt]  [font=\tiny]  {$2n_{\partial ^{\prime }}$};
  
\draw (326,93.4) node [anchor=north west][inner sep=0.75pt]    {$\pi $};
  
\draw (20.67,87.73) node [anchor=north west][inner sep=0.75pt]  [font=\tiny]  {$\partial $};
  
\draw (20,137.07) node [anchor=north west][inner sep=0.75pt]  [font=\tiny]  {$\partial ^{\prime }$};

\end{tikzpicture}
\caption{ The cylinder $(\mathcal{S},M)$ and its universal cover $(\tilde{\mathcal{S}},\pi)$.}\label{marked surface}
\end{figure}
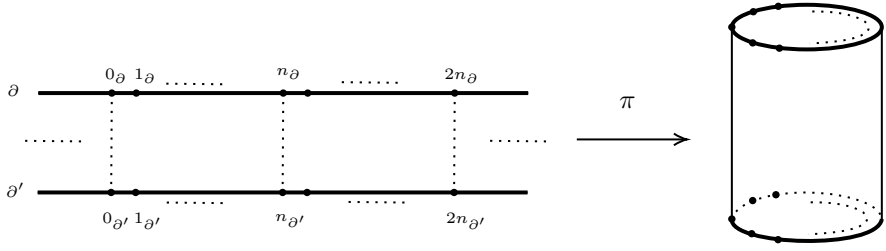

 The marked points on $\widetilde{\mathcal{S}}$ are given by $\widetilde{M} := \pi^{-1}(M) = \{(i, 0), (j, 1) \mid i, j \in \mathbb{Z}\}$. A marked point $(i, 1)$ on the boundary $\partial := \mathbb{R} \times \{1\}$ is denoted by $i_{\partial}$, and a marked point $(j, 0)$ on $\partial^\prime := \mathbb{R} \times \{0\}$ is denoted by $j_{\partial^\prime}$, where $i, j \in \mathbb{Z}$. Note that $(\widetilde{\mathcal{S}}, \widetilde{M})$ does not satisfy the definition of a marked surface because $|\widetilde{M}|$ is infinite. Nevertheless, we can still define curves on $(\widetilde{\mathcal{S}}, \widetilde{M})$ in the same way as those on $(\mathcal{S}, M)$. 

\begin{rem}
    It is evident that any curve on $\widetilde{\mathcal{S}}$ is uniquely determined by its two endpoints. For the marked points $(i,0)$, $(j,0)$, $(i,1)$, and $(j,1)$ of $\widetilde{\mathcal{S}}$, we denote:
\begin{itemize}
    \item $D_i^j$ as the curve with endpoints $(i,0)$ and $(j,1)$,
    \item $D_{i,j}$ as the curve with endpoints $(i,0)$ and $(j,0)$,
    \item $D^{i,j}$ as the curve with endpoints $(i,1)$ and $(j,1)$.
\end{itemize}

For simplicity, denote
\[[D^j_i]:=\pi(D^j_i),\;\;[D^{i,j}]:=\pi(D^{i,j})\;\;{\rm and}\;\;[D_{i,j}]:=\pi(D_{i,j}).\]
 It follows that for any $k\in\mathbb{Z}$,
$$[D^{i+kn}_{j+kn}]=[D^i_j],\;\;[D^{i+kn,j+kn}]=[D^{i, j}]\;\;{\rm and}\;\;[D_{i+kn,j+kn}]=[D_{i, j}].$$
Later, for visualization, we can represent arbitrary curves on $\mathcal{S}$ by drawing one of their lifts on $\widetilde{\mathcal{S}}$.
\end{rem}

Next, we introduce an order $2$ self-homeomorphism on $(\mathcal{S},M)$. Consider the homeomorphism $\widetilde{\sigma}: \widetilde{\mathcal{S}} \rightarrow \widetilde{\mathcal{S}}, \quad (x, y) \mapsto (-x, 1-y),$ which reflects all points on $\widetilde{\mathcal{S}}$ with respect to the point $\left(0,\frac{1}{2}\right)$. It can be verified that $\widetilde{\sigma}$ induces a homeomorphism $\sigma: \mathcal{S} \rightarrow \mathcal{S}$  such that the following diagram is commutative
\begin{figure}[H]
    \centering
       	\begin{tikzcd}
		\widetilde{\mathcal{S}} \arrow[d, "\widetilde{\sigma}"'] \arrow[r, "\pi"] & \mathcal{S} \arrow[d, "\sigma", dashed] \\
		\widetilde{\mathcal{S}} \arrow[r, "\pi"'] & \mathcal{S} 
	\end{tikzcd}
\end{figure}
\noindent In particular, $\sigma$ is a diffeomorphism of order $2$ that preserves $M$ setwise. 
\begin{rem}
The diffeomorphism $\sigma$ has exactly two fixed points, denoted $\times_1$ and $\times_2$. These fixed points are the images under $\pi$ of the points $(0, \frac{1}{2})$ and $(\frac{n}{2}, \frac{1}{2})$ in $\widetilde{\mathcal{S}}$, respectively. Define $\mathcal{X}=\{\times_1,\times_2\}$. Observe that $\sigma$ reflects all point with respect to the line through $\times_1$ and $\times_2$. 
\end{rem}
We will use the data $(\mathcal{S},M,\sigma)$ to give a geometric model for ${\rm coh}\mbox{-}\mathbb{X}$.  To do it, we first define the skew-curves on $(\mathcal{S},M,\sigma)$.
\subsection{Skew-curves on $(\mathcal{S},M,\sigma)$}\label{hybrid} 
  Let  $\mathbf{C}$ be the union of $\mathbf{C}_{\mathtt{b}}$, $\mathbf{C}_{\mathtt{p}}$, and $\mathbf{C}_{\mathbf{k}^*}$, where:
	\begin{itemize} 
	\item $\mathbf{C}_{\mathtt{b}}:=\{[D^j_i]\mid i,j\in \mathbb{Z}\}$; 
	\item $\mathbf{C}_{\mathtt{p}}:=\{[D_{i,j}],[D^{i,j}]\mid i,j\in \mathbb{Z} \text{ with } j - i \geq 2\}$; 
	\item $\mathbf{C}_{\mathbf{k}^*}:=\{(\lambda, L^{j})\mid \lambda \in \mathbf{k}^*, j\in \mathbb{Z}_+ \text{ and } L = S^1 \times \left\{\frac{1}{2}\right\} \text{ is a loop on }(\mathcal{S}, M) \}$. \end{itemize}

\begin{rem}\label{sigma-fixed}
Let $\gamma$ be a curve in $\mathbf{C}_{\mathtt{b}}\cup \mathbf{C}_{\mathtt{p}}$. It can be verified that 
\begin{itemize}
    \item $\gamma\in \mathbf{C}_{\mathtt{p}}$ are not $\sigma$-fixed; 
    \item $\gamma$ is $\sigma$-fixed if and only if  $\gamma = [D^j_i]$ for some $i, j \in \mathbb{Z}$ with $i + j \equiv 0 \pmod{n}$. In particular, $\gamma$ is called \emph{$\sigma$-fixed on} $\times_1$ (resp. $\times_2$) if $\frac{i+j}{n}$ is even (resp. odd). 
    \item Each  not-$\sigma$-fixed $\gamma\in \mathbf{C}_{\mathtt{b}}$ can be expressed uniquely as $[D^{k-i}_{i}]$ or $[D^{-i}_{i-k}]$ (namely, $\sigma([D^{k-i}_i])$) for some $i,k \in \mathbb{Z}$ and $1 \leq k \leq n-1$. 
\end{itemize}    
\end{rem}
 
With the self-homeomorphism $\sigma$ of $\mathcal{S}$, we describe the process of constructing skew-curves on $(\mathcal{S}, M, \sigma)$ from the elements of $\mathbf{C} = \mathbf{C}_{\mathtt{b}} \cup \mathbf{C}_{\mathtt{p}} \cup \mathbf{C}_{\mathbf{k}^*}$: 
\begin{itemize}
    \item [(i)] For each curve $\gamma$ that $\sigma$-fixed on $\times\in \mathcal{X}$, we divide it into a top half $\gamma^+$ and a bottom half $\gamma^-$ with respect to $\times$, as shown in Figure \ref{divide}:
    \begin{figure}[H]
    \centering
\tikzset{every picture/.style={line width=0.75pt}}  

\begin{tikzpicture}[x=0.75pt,y=0.75pt,yscale=-1,xscale=1]

\draw [line width=1.5]    (168.33,70.2) -- (309.97,70.28) ;
 
\draw [line width=1.5]    (168.73,20.6) -- (309.37,20.68) ;
 
\draw    (219.97,70.28) -- (230.37,70.28) ;
\draw [shift={(230.37,70.28)}, rotate = 0] [color={rgb, 255:red, 0; green, 0; blue, 0 }  ][fill={rgb, 255:red, 0; green, 0; blue, 0 }  ][line width=0.75]      (0, 0) circle [x radius= 1.34, y radius= 1.34]   ;
\draw [shift={(219.97,70.28)}, rotate = 0] [color={rgb, 255:red, 0; green, 0; blue, 0 }  ][fill={rgb, 255:red, 0; green, 0; blue, 0 }  ][line width=0.75]      (0, 0) circle [x radius= 1.34, y radius= 1.34]   ;
 
\draw    (239.81,45.26) ;
\draw [shift={(239.81,45.26)}, rotate = 0] [color={rgb, 255:red, 0; green, 0; blue, 0 }  ][fill={rgb, 255:red, 0; green, 0; blue, 0 }  ][line width=0.75]      (0, 0) circle [x radius= 1.34, y radius= 1.34]   ;
\draw [shift={(239.81,45.26)}, rotate = 0] [color={rgb, 255:red, 0; green, 0; blue, 0 }  ][fill={rgb, 255:red, 0; green, 0; blue, 0 }  ][line width=0.75]      (0, 0) circle [x radius= 1.34, y radius= 1.34]   ;
 
\draw    (259.95,20.54) -- (269.95,20.43) ;
\draw [shift={(269.95,20.43)}, rotate = 359.39] [color={rgb, 255:red, 0; green, 0; blue, 0 }  ][fill={rgb, 255:red, 0; green, 0; blue, 0 }  ][line width=0.75]      (0, 0) circle [x radius= 1.34, y radius= 1.34]   ;
\draw [shift={(259.95,20.54)}, rotate = 359.39] [color={rgb, 255:red, 0; green, 0; blue, 0 }  ][fill={rgb, 255:red, 0; green, 0; blue, 0 }  ][line width=0.75]      (0, 0) circle [x radius= 1.34, y radius= 1.34]   ;
 
\draw  [dash pattern={on 0.84pt off 2.51pt}]  (195.82,67.98) -- (176.48,67.82) ;
 
\draw  [dash pattern={on 0.84pt off 2.51pt}]  (262.32,67.48) -- (242.98,67.32) ;
 
\draw  [dash pattern={on 0.84pt off 2.51pt}]  (234.32,24.48) -- (214.98,24.32) ;
 
\draw  [dash pattern={on 0.84pt off 2.51pt}]  (302.32,24.48) -- (282.98,24.32) ;
 
\draw [color={rgb, 255:red, 0; green, 0; blue, 0 }  ,draw opacity=1 ]   (219.67,69.98) -- (259.95,20.54) ;
\draw [shift={(259.95,20.54)}, rotate = 309.17] [color={rgb, 255:red, 0; green, 0; blue, 0 }  ,draw opacity=1 ][fill={rgb, 255:red, 0; green, 0; blue, 0 }  ,fill opacity=1 ][line width=0.75]      (0, 0) circle [x radius= 1.34, y radius= 1.34]   ;
\draw [shift={(239.81,45.26)}, rotate = 309.17] [color={rgb, 255:red, 0; green, 0; blue, 0 }  ,draw opacity=1 ][fill={rgb, 255:red, 0; green, 0; blue, 0 }  ,fill opacity=1 ][line width=0.75]      (0, 0) circle [x radius= 1.34, y radius= 1.34]   ;
\draw [shift={(219.67,69.98)}, rotate = 309.17] [color={rgb, 255:red, 0; green, 0; blue, 0 }  ,draw opacity=1 ][fill={rgb, 255:red, 0; green, 0; blue, 0 }  ,fill opacity=1 ][line width=0.75]      (0, 0) circle [x radius= 1.34, y radius= 1.34]   ;
 
\draw   (244.13,46.2) .. controls (247.46,48.95) and (250.5,48.65) .. (253.25,45.32) -- (253.25,45.32) .. controls (257.17,40.56) and (260.8,39.55) .. (264.13,42.3) .. controls (260.8,39.55) and (261.09,35.8) .. (265.01,31.05)(263.25,33.19) -- (265.01,31.05) .. controls (267.76,27.72) and (267.46,24.68) .. (264.13,21.93) ;
 
\draw   (235,42.73) .. controls (231.81,40.1) and (228.9,40.37) .. (226.27,43.56) -- (226.27,43.56) .. controls (222.5,48.11) and (219.03,49.06) .. (215.85,46.43) .. controls (219.03,49.06) and (218.74,52.65) .. (214.98,57.2)(216.67,55.16) -- (214.98,57.2) .. controls (212.34,60.39) and (212.61,63.3) .. (215.8,65.93) ;
 
\draw    (209.67,70.09) -- (219.67,69.98) ;
\draw [shift={(219.67,69.98)}, rotate = 359.39] [color={rgb, 255:red, 0; green, 0; blue, 0 }  ][fill={rgb, 255:red, 0; green, 0; blue, 0 }  ][line width=0.75]      (0, 0) circle [x radius= 1.34, y radius= 1.34]   ;
\draw [shift={(209.67,70.09)}, rotate = 359.39] [color={rgb, 255:red, 0; green, 0; blue, 0 }  ][fill={rgb, 255:red, 0; green, 0; blue, 0 }  ][line width=0.75]      (0, 0) circle [x radius= 1.34, y radius= 1.34]   ;
 
\draw    (249.95,20.65) -- (259.95,20.54) ;
\draw [shift={(259.95,20.54)}, rotate = 359.39] [color={rgb, 255:red, 0; green, 0; blue, 0 }  ][fill={rgb, 255:red, 0; green, 0; blue, 0 }  ][line width=0.75]      (0, 0) circle [x radius= 1.34, y radius= 1.34]   ;
\draw [shift={(249.95,20.65)}, rotate = 359.39] [color={rgb, 255:red, 0; green, 0; blue, 0 }  ][fill={rgb, 255:red, 0; green, 0; blue, 0 }  ][line width=0.75]      (0, 0) circle [x radius= 1.34, y radius= 1.34]   ;

\draw (153.97,63.7) node [anchor=north west][inner sep=0.75pt]  [font=\tiny]  {$\partial ^{\prime }$};
 
\draw (154.97,15.7) node [anchor=north west][inner sep=0.75pt]  [font=\tiny]  {$\partial $};
 
\draw (266.27,40.33) node [anchor=north west][inner sep=0.75pt]  [font=\tiny]  {$\gamma ^{+ }$};
 
\draw (206.93,38.2) node [anchor=north west][inner sep=0.75pt]  [font=\tiny]  {$\gamma ^{- }$};
 
\draw (238.81,46.66) node [anchor=north west][inner sep=0.75pt]  [font=\tiny]  {$\times $};
\end{tikzpicture}
\caption{The top half $\gamma^+$ and bottom  half $\gamma^-$ of  $\gamma$}\label{divide}
\end{figure}
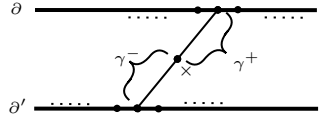

\item[(ii)] For each non-$\sigma$-fixed curve $\gamma$ in $\mathbf{C}_{\mathtt{b}} \cup \mathbf{C}_{\mathtt{p}}$, we pair $\gamma$ with its image under $\sigma$, forming $\{\gamma, \sigma(\gamma)\}$.
 
\item[(iii)]  For each parameterized loop $(\lambda, L^j)$ in $\mathbf{C}_{\mathbf{k}^*}$, where $\lambda\neq \pm 1,j\in \mathbb{Z}_+$, we pair it with $(\lambda^{-1}, L^j)$, forming $\{(\lambda, L^j), (\lambda^{-1}, L^j)\}$. 

\item[(iv)]  For each parameterized loop $(\lambda, L^j)$ in $\mathbf{C}_{\mathbf{k}^*}$, where $\lambda= 1$ or $-1$ and $j\in \mathbb{Z}_+$, we assign new parameters to generate two elements: $(\lambda, L^j)^+$ and $(\lambda, L^j)^-$.
\end{itemize}

We refer to the elements constructed in (i)-(iv) as \emph{skew-curves} on $(\mathcal{S}, M, \sigma)$. A skew-curve is denoted by $\widehat{\gamma}$. Let $\widehat{\mathbf{C}}$ denote the set of all skew-curves on $(\mathcal{S}, M, \sigma)$. For convenience in future discussions, we partition $\widehat{\mathbf{C}}$ as follows:
\[
\widehat{\mathbf{C}} = \mathbf{C}_{\mathtt{b},\times_1}^+\cup \mathbf{C}_{\mathtt{b},\times_1}^-\cup \mathbf{C}_{\mathtt{b},\times_2}^+\cup \mathbf{C}_{\mathtt{b},\times_2}^- \cup \mathbf{C}_{\mathtt{b}}^\sigma \cup \mathbf{C}_{\mathtt{p}}^\sigma \cup \mathbf{C}_{\mathbf{k}^*}^{\text{pw}} \cup \mathbf{C}_{\mathbf{k}^*}^{\text{sp}},
\]
where, in slightly more precise terms, the components are defined as follows:
\begin{itemize}
    \item  $\mathbf{C}_{\mathtt{b},\times_1}^+ = \{[D_{-i}^{i}]^+\mid i\in \mathbb{Z}\}$ and $\mathbf{C}_{\mathtt{b},\times_1}^- = \{[D_{-i}^{i}]^-\mid i\in \mathbb{Z}\}$;
    \item   $\mathbf{C}_{\mathtt{b},\times_2}^+ = \{[D_{-i}^{n+i}]^+\mid i\in \mathbb{Z}\}$ and $\mathbf{C}_{\mathtt{b},\times_2}^- = \{[D_{-i}^{n+i}]^-\mid i\in \mathbb{Z}\}$;
    \item  $\mathbf{C}_{\mathtt{b}}^\sigma = \{\widetilde{[D_i^{k-i}]}(= \widetilde{[D_{i-k}^{-i}]}):= \{[D^{k-i}_{i}],[D^{-i}_{i-k}]\} \mid i,j\in \mathbb{Z} \text{ and } 1\leq k\leq n-1\}$;
\item  $\mathbf{C}_{\mathtt{p}}^\sigma = \{ \widetilde{[D^{i,j}]}(= \widetilde{[D_{-j,-i}]}):=\{[D^{i,j}], [D_{-j,-i}]\} \mid i,j\in \mathbb{Z} \text{ and } j-i\geq 2\};$
    \item  $\mathbf{C}_{\mathbf{k}^*}^{\text{pw}} = \{ \{(\lambda, L^j), (\lambda^{-1}, L^j)\} \mid \lambda \neq \lambda^{-1}, j \in \mathbb{Z}_+ \}$;
    \item  $\mathbf{C}_{\mathbf{k}^*}^{\text{sp}} = \{ (\lambda, L^j)^+, (\lambda, L^j)^- \mid \lambda = 1 \text{ or } -1, j \in \mathbb{Z}_+ \}.$
\end{itemize}
We also use the notation $\mathbf{C}_{\mathtt{b},\times_i} = \mathbf{C}_{\mathtt{b},\times_i}^+ \cup \mathbf{C}_{\mathtt{b},\times_i}^-$, and $\mathbf{C}_{\mathtt{b},\mathcal{X}} = \mathbf{C}_{\mathtt{b},\times_1} \cup \mathbf{C}_{\mathtt{b},\times_2}$. 

Recall from Appendix \ref{equivariant} that each indecomposable bundle in ${\rm vect}\mbox{-}\mathbb{X}$ is either a line bundle or an extension bundle. Note that any line bundle in ${\rm vect}\mbox{-}\mathbb{X}$ can be expressed as a degree shift of the structure sheaf $\mathcal{O}_\mathbb{X}$, namely $\mathcal{O}_\mathbb{X}(\vec{x})$ for some $\vec{x} \in \mathbb{L}(2,2,n)$. Under the relations $2\vec{x}_1 = 2\vec{x}_2 = n\vec{x}_3 = \vec{c}$, $\vec{x}$ can be uniquely written in one of the forms $i\vec{x}_3$, $\vec{x}_1 - \vec{x}_2 + i\vec{x}_3$,  $\vec{x}_1 + i\vec{x}_3$, and $\vec{x}_2 + i\vec{x}_3$, where $i \in \mathbb{Z}$. Moreover, by \cite[Proposition 2.3 (i)]{Dong2025Two}, each extension bundle has a unique expression as $E_{\mathcal{O}_\mathbb{X}(-(i+1)\vec{x}_3)}\langle(k-1)\vec{x}_3\rangle$ with $i,k\in \mathbb{Z}$ and $1\leq k\leq n-1$. 

Denote by $G$  the cyclic group of order $2$, generated by $\sigma_{1,2}$, which exchanges the two weighted points of $\mathbb{Y}$. Let $\lambda \notin \{\pm 1, 0, \infty\} \subset \mathbb{Y}$, and  $\varphi: \mathbb{Y}/G \to \mathbb{X}$ be the isomorphism introduced in Remark \ref{isomor}.

\begin{prop}\label{correspondence}
The isoclasses of $\operatorname{ind}({\rm coh}\mbox{-}\mathbb{X})$ can be  parameterized by the skew-curves in $\widehat{\mathbf{C}}$ through the bijection \[\widehat{\phi}:\widehat{\mathbf{C}}\rightarrow \operatorname{ind}\left({\rm coh}\mbox{-}\mathbb{X}(2,2,n)\right). \] given as follows:
\begin{table}[H]
\centering
\caption{Mapping of Skew-Curves to Isoclasses in ${\rm ind}({\rm coh}\mbox{-}\mathbb{X})$}
\label{table:skew-curves-mapping}
\[\begin{tabular}{|c|c|}
\hline
\text{Skew-Curve Class} & \text{Mapping $\widehat{\phi}$} \\
\hline
\multirow{2}{*}{$\mathbf{C}_{\mathtt{b},\times_1}$} & $[D^i_{-i}]^+ \mapsto \mathcal{O}_\mathbb{X}(\vec{x}_1 - \vec{x}_2 + i \vec{x}_3)$ \\
 & $[D^i_{-i}]^- \mapsto \mathcal{O}_\mathbb{X}(i \vec{x}_3)$ \\
\hline
\multirow{2}{*}{$\mathbf{C}_{\mathtt{b},\times_2}$} & $[D^{n+i}_{-i}]^+ \mapsto \mathcal{O}_\mathbb{X}(\vec{x}_1 + i \vec{x}_3)$ \\
 & $[D^{n+i}_{-i}]^- \mapsto \mathcal{O}_\mathbb{X}(\vec{x}_2 + i \vec{x}_3)$ \\
\hline
$\mathbf{C}_{\mathtt{b}}^\sigma$ & $\widetilde{[D^{k-i}_i]} \mapsto E_{\mathcal{O}_\mathbb{X}(-(i+1)\vec{x}_3)}\langle(k-1)\vec{x}_3\rangle$ \\
\hline
$\mathbf{C}_{\mathtt{p}}^\sigma$ & $\widetilde{[D^{i-j-1,i}]} \mapsto S_{1,\overline{i}}^{(j)}$ \\
\hline
$\mathbf{C}_{\mathbf{k}^*}^{\text{pw}}$ & $\{(\lambda, L^j),(\lambda^{-1}, L^j)\} \mapsto S_{\varphi(G\lambda)}^{(j)}$ \\
\hline
\multirow{4}{*}{$\mathbf{C}_{\mathtt{b}}^{\text{sp}}$} & $(1, L^j)^+ \mapsto S_{\infty,\overline{j+1}}^{(j)}$ \\
 & $(1, L^j)^- \mapsto S_{\infty,\overline{j}}^{(j)}$\\
& $(-1, L^j)^+ \mapsto S_{0,\overline{j+1}}^{(j)}$ \\
 & $(-1, L^j)^- \mapsto S_{0,\overline{j}}^{(j)}$ \\
\hline
\end{tabular}\]
\end{table}
\end{prop}
\begin{proof} 
According to Remark \ref{sigma-fixed}, $\widehat{\phi}$ is well-defined. It remains to show that $\widehat{\phi}$ is a bijection.

First, recall that the objects of $\operatorname{ind}({\rm coh}\mbox{-}\mathbb{Y})$ are line bundles $\mathcal{O}_\mathbb{Y}(\vec{y})$ with $\vec{y}\in \mathbb{L}(n,n)$, and objects of finite length $S_{\infty,\overline{i}}^{(j)}$,  $S_{0,\overline{i}}^{(j)}$ ($i\in\mathbb{Z}$) and $S_{\lambda}^{(j)}\,(\lambda \in \mathbf k^{*})$ for $j \in \mathbb{Z}_+$. By \cite[Theorem 3.4]{Chen2023GeometricModel}, there exists a bijection 
	\begin{align*}
	\phi: \mathbf{C}&\rightarrow {\rm ind}({\rm coh}\mbox{-}\mathbb{Y})\\
		[D^j_i] & \mapsto \mathcal{O}_\mathbb{Y}(j\vec{y}_{1}-i\vec{y}_{2}), \\
		[D^{i-j-1,i}] & \mapsto S_{\infty,\overline{i}}^{(j)}, \\
		[D_{-i,j-i+1}]  & \mapsto S_{0,\overline{i}}^{(j)}, \\
		(\lambda,L^{j}) & \mapsto S_{\lambda}^{(j)}.
	\end{align*}
Next, note that the diffeomorphism $\sigma$ can naturally be regarded as an element of  the mapping class group of $(\mathcal{S},M)$ and it is implicit in the proof of \cite[Theorem 1.1]{Chen2023GeometricModel} that $\sigma$ is compatible with $\sigma_{1,2}$ in the following sense: 
\[ \phi(\sigma(\gamma)) = \sigma_{1,2}(\phi(\gamma)), \text{ for any curve }  \gamma \in \mathbf{C}_{\mathtt{b}}\cup \mathbf{C}_{\mathtt{p}},\] 
where $\sigma_{1,2}$ is the automorphism of ${\rm coh}\mbox{-}\mathbb{Y}$ mentioned in Appendix \ref{equivariant}. 

Therefore, based on the bijections $\phi$ and  Table \ref{H2}, it is straightforward to verify that $\widehat{\phi}$ is indeed a bijection.
\end{proof}

To illustrate the bijection $\widehat{\phi}$, we present an example.

\begin{exm}\label{example1}
Let $n=4$ and consider the following curves on $(\mathcal{S},M)$, which are depicted in distinct colors:
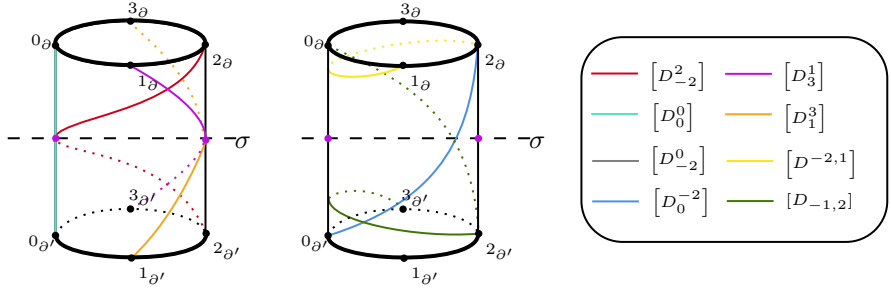
\begin{figure}[H]
    \centering

\tikzset{every picture/.style={line width=0.75pt}}   

\begin{tikzpicture}[x=0.75pt,y=0.75pt,yscale=-1,xscale=1]

\draw  [dash pattern={on 4.5pt off 4.5pt}]  (43.75,112.85) -- (163.42,112.85) ;
  
\draw  [dash pattern={on 4.5pt off 4.5pt}]  (190.42,112.85) -- (310.08,112.85) ;
  
\draw   (142.43,65.17) -- (142.43,161.72) .. controls (142.43,167.92) and (125.68,172.94) .. (105.01,172.94) .. controls (84.35,172.94) and (67.6,167.92) .. (67.6,161.72) -- (67.6,65.17) .. controls (67.6,58.97) and (84.35,53.94) .. (105.01,53.94) .. controls (125.68,53.94) and (142.43,58.97) .. (142.43,65.17) .. controls (142.43,71.37) and (125.68,76.39) .. (105.01,76.39) .. controls (84.35,76.39) and (67.6,71.37) .. (67.6,65.17) ;
  
\draw [color={rgb, 255:red, 208; green, 2; blue, 27 }  ,draw opacity=1 ] [dash pattern={on 0.84pt off 2.51pt}]  (142.76,160.72) .. controls (125.3,117.58) and (67.3,121.58) .. (67.75,112.85) ;
  
\draw [color={rgb, 255:red, 208; green, 2; blue, 27 }  ,draw opacity=1 ]   (142.33,65.77) .. controls (131.74,101.7) and (64.4,101.7) .. (66.99,114.45) ;
  
\draw [color={rgb, 255:red, 245; green, 166; blue, 35 }  ,draw opacity=1 ]   (105.5,173.1) .. controls (129.99,150.26) and (143.5,112.77) .. (142.58,113.68) ;
  
\draw [color={rgb, 255:red, 245; green, 166; blue, 35 }  ,draw opacity=1 ] [dash pattern={on 0.84pt off 2.51pt}]  (103.92,54.1) .. controls (128.4,74.93) and (140.2,92.41) .. (141.49,113.84) ;
  
\draw [color={rgb, 255:red, 189; green, 16; blue, 224 }  ,draw opacity=1 ] [dash pattern={on 0.84pt off 2.51pt}]  (105,148.43) .. controls (121.67,138.77) and (140.55,123.33) .. (141.5,115.1) ;
  
\draw [color={rgb, 255:red, 189; green, 16; blue, 224 }  ,draw opacity=1 ]   (105.01,76.39) .. controls (126.08,88.1) and (143.09,101.83) .. (142.58,113.68) ;
  
\draw [color={rgb, 255:red, 208; green, 2; blue, 27 }  ,draw opacity=1 ]   (334.87,80.28) -- (358.2,80.28) ;
  
\draw [color={rgb, 255:red, 80; green, 227; blue, 194 }  ,draw opacity=1 ]   (334.87,100.62) -- (358.2,100.62) ;
  
\draw [color={rgb, 255:red, 189; green, 16; blue, 224 }  ,draw opacity=1 ]   (401.7,80.45) -- (425.03,80.45) ;
  
\draw [color={rgb, 255:red, 245; green, 166; blue, 35 }  ,draw opacity=1 ]   (401.7,100.78) -- (425.03,100.78) ;
  
\draw [color={rgb, 255:red, 128; green, 128; blue, 128 }  ,draw opacity=1 ]   (335.37,124.28) -- (358.7,124.28) ;
  
\draw [color={rgb, 255:red, 74; green, 144; blue, 226 }  ,draw opacity=1 ]   (335.37,144.62) -- (358.7,144.62) ;
  
\draw [color={rgb, 255:red, 80; green, 227; blue, 194 }  ,draw opacity=1 ]   (67.6,65.17) -- (67.6,161.72) ;
  
\draw [color={rgb, 255:red, 248; green, 231; blue, 28 }  ,draw opacity=1 ]   (402.37,124.28) -- (425.7,124.28) ;
  
\draw [color={rgb, 255:red, 65; green, 117; blue, 5 }  ,draw opacity=1 ]   (402.53,144.45) -- (425.87,144.45) ;
  
\draw [color={rgb, 255:red, 74; green, 144; blue, 226 }  ,draw opacity=1 ]   (203.6,161.82) .. controls (263.83,142.87) and (275.07,106.8) .. (278.33,65.87) ;
  
\draw [color={rgb, 255:red, 65; green, 117; blue, 5 }  ,draw opacity=1 ] [dash pattern={on 0.84pt off 2.51pt}]  (203.6,66.27) .. controls (239.5,82.53) and (275.17,99.53) .. (278.76,160.82) ;
  
\draw [color={rgb, 255:red, 65; green, 117; blue, 5 }  ,draw opacity=1 ] [dash pattern={on 0.84pt off 2.51pt}]  (203.97,143.12) .. controls (207.3,138.78) and (225.17,135.78) .. (240.92,147.62) ;
  
\draw [color={rgb, 255:red, 65; green, 117; blue, 5 }  ,draw opacity=1 ]   (203.97,143.12) .. controls (203.3,154.78) and (248.63,163.45) .. (278.76,160.82) ;
  
\draw [color={rgb, 255:red, 248; green, 231; blue, 28 }  ,draw opacity=1 ] [dash pattern={on 0.84pt off 2.51pt}]  (203.87,78.55) .. controls (196.54,72.88) and (253.12,58.22) .. (278.33,65.87) ;
  
\draw [color={rgb, 255:red, 248; green, 231; blue, 28 }  ,draw opacity=1 ]   (203.87,78.55) .. controls (204.54,83.22) and (227.22,83.53) .. (241.01,76.49) ;
  
\draw    (67.6,161.72) ;
\draw [shift={(67.6,161.72)}, rotate = 0] [color={rgb, 255:red, 0; green, 0; blue, 0 }  ][fill={rgb, 255:red, 0; green, 0; blue, 0 }  ][line width=0.75]      (0, 0) circle [x radius= 1.34, y radius= 1.34]   ;
  
\draw    (67.6,66.17) ;
\draw [shift={(67.6,66.17)}, rotate = 0] [color={rgb, 255:red, 0; green, 0; blue, 0 }  ][fill={rgb, 255:red, 0; green, 0; blue, 0 }  ][line width=0.75]      (0, 0) circle [x radius= 1.34, y radius= 1.34]   ;
  
\draw    (105.5,173.1) ;
\draw [shift={(105.5,173.1)}, rotate = 0] [color={rgb, 255:red, 0; green, 0; blue, 0 }  ][fill={rgb, 255:red, 0; green, 0; blue, 0 }  ][line width=0.75]      (0, 0) circle [x radius= 1.34, y radius= 1.34]   ;
  
\draw    (105,148.43) ;
\draw [shift={(105,148.43)}, rotate = 0] [color={rgb, 255:red, 0; green, 0; blue, 0 }  ][fill={rgb, 255:red, 0; green, 0; blue, 0 }  ][line width=0.75]      (0, 0) circle [x radius= 1.34, y radius= 1.34]   ;
  
\draw    (142.76,160.72) ;
\draw [shift={(142.76,160.72)}, rotate = 0] [color={rgb, 255:red, 0; green, 0; blue, 0 }  ][fill={rgb, 255:red, 0; green, 0; blue, 0 }  ][line width=0.75]      (0, 0) circle [x radius= 1.34, y radius= 1.34]   ;
  
\draw    (105,76.1) ;
\draw [shift={(105,76.1)}, rotate = 0] [color={rgb, 255:red, 0; green, 0; blue, 0 }  ][fill={rgb, 255:red, 0; green, 0; blue, 0 }  ][line width=0.75]      (0, 0) circle [x radius= 1.34, y radius= 1.34]   ;
  
\draw    (142.33,65.77) ;
\draw [shift={(142.33,65.77)}, rotate = 0] [color={rgb, 255:red, 0; green, 0; blue, 0 }  ][fill={rgb, 255:red, 0; green, 0; blue, 0 }  ][line width=0.75]      (0, 0) circle [x radius= 1.34, y radius= 1.34]   ;
  
\draw    (105.33,53.43) -- (105.01,53.94) ;
\draw [shift={(105.01,53.94)}, rotate = 122.24] [color={rgb, 255:red, 0; green, 0; blue, 0 }  ][fill={rgb, 255:red, 0; green, 0; blue, 0 }  ][line width=0.75]      (0, 0) circle [x radius= 1.34, y radius= 1.34]   ;
  
\draw   (278.43,65.27) -- (278.43,161.82) .. controls (278.43,168.02) and (261.68,173.04) .. (241.01,173.04) .. controls (220.35,173.04) and (203.6,168.02) .. (203.6,161.82) -- (203.6,65.27) .. controls (203.6,59.07) and (220.35,54.04) .. (241.01,54.04) .. controls (261.68,54.04) and (278.43,59.07) .. (278.43,65.27) .. controls (278.43,71.47) and (261.68,76.49) .. (241.01,76.49) .. controls (220.35,76.49) and (203.6,71.47) .. (203.6,65.27) ;
  
\draw    (203.6,161.82) ;
\draw [shift={(203.6,161.82)}, rotate = 0] [color={rgb, 255:red, 0; green, 0; blue, 0 }  ][fill={rgb, 255:red, 0; green, 0; blue, 0 }  ][line width=0.75]      (0, 0) circle [x radius= 1.34, y radius= 1.34]   ;
  
\draw    (203.6,66.27) ;
\draw [shift={(203.6,66.27)}, rotate = 0] [color={rgb, 255:red, 0; green, 0; blue, 0 }  ][fill={rgb, 255:red, 0; green, 0; blue, 0 }  ][line width=0.75]      (0, 0) circle [x radius= 1.34, y radius= 1.34]   ;
  
\draw    (241.5,173.2) ;
\draw [shift={(241.5,173.2)}, rotate = 0] [color={rgb, 255:red, 0; green, 0; blue, 0 }  ][fill={rgb, 255:red, 0; green, 0; blue, 0 }  ][line width=0.75]      (0, 0) circle [x radius= 1.34, y radius= 1.34]   ;
  
\draw    (241,148.53) ;
\draw [shift={(241,148.53)}, rotate = 0] [color={rgb, 255:red, 0; green, 0; blue, 0 }  ][fill={rgb, 255:red, 0; green, 0; blue, 0 }  ][line width=0.75]      (0, 0) circle [x radius= 1.34, y radius= 1.34]   ;
  
\draw    (278.76,160.82) ;
\draw [shift={(278.76,160.82)}, rotate = 0] [color={rgb, 255:red, 0; green, 0; blue, 0 }  ][fill={rgb, 255:red, 0; green, 0; blue, 0 }  ][line width=0.75]      (0, 0) circle [x radius= 1.34, y radius= 1.34]   ;
  
\draw    (241,76.2) ;
\draw [shift={(241,76.2)}, rotate = 0] [color={rgb, 255:red, 0; green, 0; blue, 0 }  ][fill={rgb, 255:red, 0; green, 0; blue, 0 }  ][line width=0.75]      (0, 0) circle [x radius= 1.34, y radius= 1.34]   ;
  
\draw    (278.33,65.87) ;
\draw [shift={(278.33,65.87)}, rotate = 0] [color={rgb, 255:red, 0; green, 0; blue, 0 }  ][fill={rgb, 255:red, 0; green, 0; blue, 0 }  ][line width=0.75]      (0, 0) circle [x radius= 1.34, y radius= 1.34]   ;
  
\draw    (241.33,53.53) -- (241.01,54.04) ;
\draw [shift={(241.01,54.04)}, rotate = 122.24] [color={rgb, 255:red, 0; green, 0; blue, 0 }  ][fill={rgb, 255:red, 0; green, 0; blue, 0 }  ][line width=0.75]      (0, 0) circle [x radius= 1.34, y radius= 1.34]   ;
  
\draw  [line width=1.5]  (67.5,64.77) .. controls (67.5,58.57) and (84.25,53.54) .. (104.92,53.54) .. controls (125.58,53.54) and (142.33,58.57) .. (142.33,64.77) .. controls (142.33,70.97) and (125.58,75.99) .. (104.92,75.99) .. controls (84.25,75.99) and (67.5,70.97) .. (67.5,64.77) -- cycle ;
  
\draw  [line width=1.5]  (203.6,65.27) .. controls (203.6,59.07) and (220.35,54.04) .. (241.01,54.04) .. controls (261.68,54.04) and (278.43,59.07) .. (278.43,65.27) .. controls (278.43,71.47) and (261.68,76.49) .. (241.01,76.49) .. controls (220.35,76.49) and (203.6,71.47) .. (203.6,65.27) -- cycle ;
  
\draw   (330,82.6) .. controls (330,71.22) and (339.22,62) .. (350.6,62) -- (462,62) .. controls (473.38,62) and (482.6,71.22) .. (482.6,82.6) -- (482.6,144.4) .. controls (482.6,155.78) and (473.38,165) .. (462,165) -- (350.6,165) .. controls (339.22,165) and (330,155.78) .. (330,144.4) -- cycle ;
  
\draw [color={rgb, 255:red, 189; green, 16; blue, 224 }  ,draw opacity=1 ]   (67.75,112.85) ;
\draw [shift={(67.75,112.85)}, rotate = 0] [color={rgb, 255:red, 189; green, 16; blue, 224 }  ,draw opacity=1 ][fill={rgb, 255:red, 189; green, 16; blue, 224 }  ,fill opacity=1 ][line width=0.75]      (0, 0) circle [x radius= 1.34, y radius= 1.34]   ;
  
\draw [color={rgb, 255:red, 189; green, 16; blue, 224 }  ,draw opacity=1 ]   (278.6,112.94) ;
\draw [shift={(278.6,112.94)}, rotate = 0] [color={rgb, 255:red, 189; green, 16; blue, 224 }  ,draw opacity=1 ][fill={rgb, 255:red, 189; green, 16; blue, 224 }  ,fill opacity=1 ][line width=0.75]      (0, 0) circle [x radius= 1.34, y radius= 1.34]   ;
  
\draw [color={rgb, 255:red, 189; green, 16; blue, 224 }  ,draw opacity=1 ]   (203.6,112.94) ;
\draw [shift={(203.6,112.94)}, rotate = 0] [color={rgb, 255:red, 189; green, 16; blue, 224 }  ,draw opacity=1 ][fill={rgb, 255:red, 189; green, 16; blue, 224 }  ,fill opacity=1 ][line width=0.75]      (0, 0) circle [x radius= 1.34, y radius= 1.34]   ;
  
\draw [color={rgb, 255:red, 189; green, 16; blue, 224 }  ,draw opacity=1 ]   (142.58,113.68) ;
\draw [shift={(142.58,113.68)}, rotate = 0] [color={rgb, 255:red, 189; green, 16; blue, 224 }  ,draw opacity=1 ][fill={rgb, 255:red, 189; green, 16; blue, 224 }  ,fill opacity=1 ][line width=0.75]      (0, 0) circle [x radius= 1.34, y radius= 1.34]   ;
  
\draw [line width=1.5]    (67.93,161.72) .. controls (66.46,175.15) and (140.38,177.65) .. (142.76,161.72) ;
  
\draw [line width=1.5]    (203.6,161.82) .. controls (202.13,175.25) and (276.05,177.75) .. (278.43,161.82) ;
  
\draw  [dash pattern={on 0.84pt off 2.51pt}]  (67.6,161.72) .. controls (75.8,141.83) and (141.13,146.16) .. (142.76,161.72) ;
  
\draw  [dash pattern={on 0.84pt off 2.51pt}]  (203.6,161.82) .. controls (211.8,141.93) and (277.13,146.26) .. (278.76,161.82) ;

\draw (362.03,73.18) node [anchor=north west][inner sep=0.75pt]  [font=\tiny]  {$\left[ D_{-2}^{2}\right]$};
  
\draw (363.03,93.68) node [anchor=north west][inner sep=0.75pt]  [font=\tiny]  {$\left[ D_{0}^{0}\right]$};
  
\draw (428.87,94.35) node [anchor=north west][inner sep=0.75pt]  [font=\tiny]  {$\left[ D_{1}^{3}\right]$};
  
\draw (428.53,73.18) node [anchor=north west][inner sep=0.75pt]  [font=\tiny]  {$\left[ D_{3}^{1}\right]$};
  
\draw (362.53,115.35) node [anchor=north west][inner sep=0.75pt]  [font=\tiny]  {$\left[ D_{-2}^{0}\right]$};
  
\draw (363.53,136.52) node [anchor=north west][inner sep=0.75pt]  [font=\tiny]  {$\left[ D_{0}^{-2}\right]$};
  
\draw (429.53,117.35) node [anchor=north west][inner sep=0.75pt]  [font=\tiny]  {$\left[ D^{-2,1}\right]$};
  
\draw (430.53,139.02) node [anchor=north west][inner sep=0.75pt]  [font=\tiny]  {$[ D_{-1,2}]$};
  
\draw (52.67,160.17) node [anchor=north west][inner sep=0.75pt]  [font=\tiny]  {$0_{\partial ^{\prime }}$};
  
\draw (107.5,176.5) node [anchor=north west][inner sep=0.75pt]  [font=\tiny]  {$1_{\partial ^{\prime }}$};
  
\draw (144.76,164.12) node [anchor=north west][inner sep=0.75pt]  [font=\tiny]  {$2_{\partial ^{\prime }}$};
  
\draw (54,59.5) node [anchor=north west][inner sep=0.75pt]  [font=\tiny]  {$0_{\partial }$};
  
\draw (107.01,79.79) node [anchor=north west][inner sep=0.75pt]  [font=\tiny]  {$1_{\partial }$};
  
\draw (144.33,69.17) node [anchor=north west][inner sep=0.75pt]  [font=\tiny]  {$2_{\partial }$};
  
\draw (102.17,43.4) node [anchor=north west][inner sep=0.75pt]  [font=\tiny]  {$3_{\partial }$};
  
\draw (102.76,137.12) node [anchor=north west][inner sep=0.75pt]  [font=\tiny]  {$3_{\partial ^{\prime }}$};
  
\draw (188.67,160.27) node [anchor=north west][inner sep=0.75pt]  [font=\tiny]  {$0_{\partial ^{\prime }}$};
  
\draw (243.5,176.6) node [anchor=north west][inner sep=0.75pt]  [font=\tiny]  {$1_{\partial ^{\prime }}$};
  
\draw (280.76,164.22) node [anchor=north west][inner sep=0.75pt]  [font=\tiny]  {$2_{\partial ^{\prime }}$};
  
\draw (190,59.6) node [anchor=north west][inner sep=0.75pt]  [font=\tiny]  {$0_{\partial }$};
  
\draw (243.01,79.89) node [anchor=north west][inner sep=0.75pt]  [font=\tiny]  {$1_{\partial }$};
  
\draw (280.33,69.27) node [anchor=north west][inner sep=0.75pt]  [font=\tiny]  {$2_{\partial }$};
  
\draw (238.17,43.5) node [anchor=north west][inner sep=0.75pt]  [font=\tiny]  {$3_{\partial }$};
  
\draw (238.76,137.22) node [anchor=north west][inner sep=0.75pt]  [font=\tiny]  {$3_{\partial ^{\prime }}$};
  
\draw (155.5,110.4) node [anchor=north west][inner sep=0.75pt]  [font=\normalsize]  {$\sigma $};
  
\draw (302.5,110.4) node [anchor=north west][inner sep=0.75pt]  [font=\normalsize]  {$\sigma $};

\end{tikzpicture}

\end{figure}
\noindent By Proposition \ref{correspondence}, we have
		\begin{align*}
			\widehat{\phi}([D^2_{-2}]^+)        & = \mathcal{O}_{\mathbb{X}}(\vec{x}_1-\vec{x}_2+2\vec{x}_3),                     &
			\widehat{\phi}([D^2_{-2}]^-)        & = \mathcal{O}_{\mathbb{X}}(2\vec{x}_3),   \\
			\widehat{\phi}([D^0_0]^+)        & = \mathcal{O}_{\mathbb{X}}(\vec{x}_1-\vec{x}_2),                     &
			\widehat{\phi}([D^0_0]^-)        & = \mathcal{O}_{\mathbb{X}},   \\
			\widehat{\phi}([D^1_3]^+)        & = \mathcal{O}_{\mathbb{X}}(\vec{x}_1-3\vec{x}_3),                                 &
			\widehat{\phi}([D^1_3]^-)        & = \mathcal{O}_{\mathbb{X}}(\vec{x}_2-3\vec{x}_3),              \\
			\widehat{\phi}([D^3_1]^+)        & = \mathcal{O}_{\mathbb{X}}(\vec{x}_1-\vec{x}_3),                                 &
			\widehat{\phi}([D^3_1]^-)        & = \mathcal{O}_{\mathbb{X}}(\vec{x}_2-\vec{x}_3),              \\
			\widehat{\phi}(\widetilde{[D^2_0]})        & = E_{\mathcal{O}_{\mathbb{X}}(-\vec{x}_3)}\langle\vec{x}_3\rangle=\tau(E_{\mathcal{O}_{\mathbb{X}}}\langle\vec{x}_3\rangle),           &
			\widehat{\phi}(\widetilde{[D^{-2,1}]})        & = S_{1,\overline{1}}^{(2)}.
  		\end{align*}
  		\end{exm}

\subsection{The $\mathbb{L}(2,2,n)$-action on ${\rm vect}\mbox{-}\mathbb{X}$}
Recall that each $\vec{x} \in \mathbb{L}(2,2,n)$ admits a degree-shift automorphism $(\vec{x})$ on ${\rm coh}\mbox{-}\mathbb{X}$. This automorphism $(\vec{x})$ can be restricted to ${\rm vect}\mbox{-}\mathbb{X}$, as it preserves the finiteness of object lengths. Thus, there is a natural $\mathbb{L}(2,2,n)$-action on ${\rm vect}\mbox{-}\mathbb{X}$  such that  
\begin{equation}\label{LLL}
    X(\vec{x}+\vec{y})=(X(\vec{x}))(\vec{y})=(X(\vec{y}))(\vec{x})
\end{equation} 
for any $X\in {\rm vect}\mbox{-}\mathbb{X}(2,2,n)$ and $\vec{x}$, $\vec{y}\in \mathbb{L}(2,2,n)$. 

Now, let $X$ be an indecomposable bundle in ${\rm vect}\mbox{-}\mathbb{X}$, and suppose  $\vec{x}=\sum_{i=1}^3 l_i\vec{x}_i+l\vec{c} \in \mathbb{L}(2,2,n)$ is in normal form. We are interested in understanding how the transformation occurs from the skew-curve $\widehat{\phi}^{-1}(X)$ to the skew-curve $\widehat{\phi}^{-1}(X(\vec{x}))$. From  Proposition \ref{correspondence}, we know that $X$ is a line bundle if and only if $X$ can be expressed as $\widehat{\phi}([D^j_i]^{\epsilon})$ for some $i,j \in \mathbb{Z}$ satisfying $i+j \equiv 0 \mod n$, and $\epsilon \in \{+, -\}$. Moreover, $X$ is an extension bundle if and only if $X$ can be written as $\widehat{\phi}(\widetilde{[D^{k-i}_i]})$ (or equivalently $\widehat{\phi}(\widetilde{[D^{-i}_{i-k}]})$) for some $i \in \mathbb{Z}$ and $1 \leq k \leq n-1$. 

We end this section by giving the following observation:
\begin{prop}\label{L-action}
  Keep the notations as above. Then the following statements hold.
  \begin{itemize}
      \item [(a)] If $X$ is a line bundle, then
            \[
X(\vec{x}) = 
\begin{cases} 
\widehat{\phi}([D^{j+l_3+(l_1+l_2+l)n}_{i-l_3-l n}]^\epsilon) & \text{if } l_1 = 0, \\
\widehat{\phi}([D^{j+l_3+(l_1+l_2+l)n}_{i-l_3-l n}]^{\neg \epsilon}) & \text{if } l_1 = 1.
\end{cases}
\]

\item[(b)] If $X$ is an extension bundle, then \[X(\vec{x})=\widehat{\phi}(\widetilde{[D^{k-i+l_3+(l_1+l_2+l)n}_{i-l_3-l n}]}).\]
  \end{itemize}

\end{prop}
\begin{proof}
We first consider the case where $\vec{x} = \vec{x}_1, \vec{x}_2,$ or $\vec{x}_3$. There are two cases to examine:

    (1)  $X$ is a line bundle. In this case, the possible forms of $X$ and their corresponding $\widehat{\phi}^{-1}(X)$, $X(\vec{x})$, and $\widehat{\phi}^{-1}(X(\vec{x}))$ are presented as follows:
    
    \begin{tikzpicture}[scale=0.9, transform shape]
     
    \node (1) at (0, 0) {$\begin{array}{|c|c|}
\hline
X & \widehat{\phi}^{-1}(X) \\ 
\hline
\mathcal{O}_\mathbb{X}(-i \vec{x}_3) & [D^{-i}_i]^- \\ 
\mathcal{O}_\mathbb{X}(\vec{x}_1 - \vec{x}_2 - i \vec{x}_3) & [D^{-i}_i]^+ \\ 
\mathcal{O}_\mathbb{X}(\vec{x}_2 - i \vec{x}_3) & [D^{n-i}_i]^- \\ 
\mathcal{O}_\mathbb{X}(\vec{x}_1 - i \vec{x}_3) & [D^{n-i}_i]^+ \\ 
\hline
\end{array}$};
    \node (2) at (7.5, 2.7) {$\begin{array}{|c|c|}
\hline
X(\vec{x}_1) & \widehat{\phi}^{-1}(X(\vec{x}_1)) \\ 
\hline
\mathcal{O}_\mathbb{X}(\vec{x}_1-i \vec{x}_3) & [D^{n-i}_i]^+ \\ 
\mathcal{O}_\mathbb{X}(\vec{x}_2 - i \vec{x}_3) & [D^{n-i}_i]^- \\ 
\mathcal{O}_\mathbb{X}(\vec{x}_1-\vec{x}_2 + (n-i) \vec{x}_3) & [D^{2n-i}_i]^+ \\ 
\mathcal{O}_\mathbb{X}( (n-i) \vec{x}_3) & [D^{2n-i}_i]^- \\ 
\hline
\end{array}$};
    \node (3) at (7.5, 0) {$\begin{array}{|c|c|}
\hline
X(\vec{x}_2) & \widehat{\phi}^{-1}(X(\vec{x}_2)) \\ 
\hline
\mathcal{O}_\mathbb{X}(\vec{x}_2-i \vec{x}_3) & [D^{n-i}_i]^- \\ 
\mathcal{O}_\mathbb{X}(\vec{x}_1 - i \vec{x}_3) & [D^{n-i}_i]^+ \\ 
\mathcal{O}_\mathbb{X}( (n-i) \vec{x}_3) & [D^{2n-i}_i]^- \\ 
\mathcal{O}_\mathbb{X}( \vec{x}_1-\vec{x}_2 +(n-i) \vec{x}_3) & [D^{2n-i}_i]^+ \\ 
\hline
\end{array}$};
    \node (4) at (7.5, -2.7) {$\begin{array}{|c|c|}
\hline
X(\vec{x}_3) & \widehat{\phi}^{-1}(X(\vec{x}_3)) \\ 
\hline
\mathcal{O}_\mathbb{X}((1-i) \vec{x}_3) & [D^{1-i}_{i-1}]^- \\ 
\mathcal{O}_\mathbb{X}(\vec{x}_1 - \vec{x}_2 - (i-1) \vec{x}_3) & [D^{1-i}_{i-1}]^+ \\ 
\mathcal{O}_\mathbb{X}(\vec{x}_2 - (i-1) \vec{x}_3) & [D^{n+1-i}_{i-1}]^- \\ 
\mathcal{O}_\mathbb{X}(\vec{x}_1 - (i-1) \vec{x}_3) & [D^{n+1-i}_{i-1}]^+ \\ 
\hline
\end{array}$};

    \draw[->] (1) -- (2) node[midway, above] {$(\vec{x}_1)$};
    \draw[->] (1) -- (3) node[midway, below] {$(\vec{x}_2)$};
    \draw[->] (1) -- (4) node[midway, below] {$(\vec{x}_3)$};
\end{tikzpicture}

\noindent where $i\in \mathbb{Z}$. Then one can see that the statement (a) holds for $\vec{x}=\vec{x}_1,\vec{x}_2$ and $\vec{x}_3$.     

(2) $X$ is an extension bundle. Then, $X$ takes the form $E_{\mathcal{O}_\mathbb{X}(-(i+1)\vec{x}_3)}\langle(k-1)\vec{x}_3\rangle$  where $i \in \mathbb{Z}$, $1 \leq k \leq n-1$, and the skew-curve corresponding to $X$ is $\widetilde{[D^{k-i}_i]}$. By \cite[Proposition 2.3 (i)]{Dong2025Two}, $X(\vec{x}_1),X(\vec{x}_2)$ can also be written in the form $E_{\mathcal{O}_\mathbb{X}((k-i-1)\vec{x}_3)}\langle(n-k-1)\vec{x}_3\rangle$. Then, by Proposition \ref{correspondence}, it follows that \[\widehat{\phi}^{-1}(X(\vec{x}_1))=\widehat{\phi}^{-1}(X(\vec{x}_2))=\widetilde{[D^{n-i}_{i-k}]}=\widetilde{[D^{k-i+n}_i]}.\]
    Moverover, note that  $X(\vec{x}_3)=E_{\mathcal{O}_\mathbb{X}(-i\vec{x}_3)}\langle(k-1)\vec{x}_3\rangle=\widetilde{[D^{k-i+1}_{i-1}]}$. Therefore, the statement (b) holds for $\vec{x}=\vec{x}_1,\vec{x}_2$ and $\vec{x}_3$.

 Since for a general $\vec{x}$, by \eqref{LLL}, the skew-curve $\widehat{\phi}^{-1}(X(\vec{x}))$ can be obtained  by a combination of the previously described operations, the proof is complete.
\end{proof}

\section{The geometric interpretation of tilting sheaf}\label{GCTTT}
 In this section, we investigate the correspondence between tilting sheaves in the category
${\rm coh}\mbox{-}\mathbb{X}$ and pseudo-triangulations on $(\mathcal{S},M, \sigma)$, and then study the compatibility
between the flip of a skew-arc within a pseudo-triangulation and the mutation of indecomposable direct
summand of the corresponding tilting sheaf.

\subsection{Tilting sheaves}
Recall from \cite{Geigle1987WeightedCurves,Geng2020MutationTilting} that a sheaf $T$ in ${\rm coh}\mbox{-}\mathbb{X}$ is called a \emph{tilting sheaf} if $T$ is rigid, namely, ${\rm Ext}^{1}_\mathbb{X}(T, T)=0$ and for $X\in {\rm coh}\mbox{-}\mathbb{X}$ with  ${\rm Hom}_\mathbb{X}(T,X)=0={\rm Ext}^{1}_\mathbb{X}(T,X)$, we have $X=0$. The following proposition provides an equivalent definition for tilting sheaves.

\begin{prop}(\cite{Lenzing2006HereditaryCategories})\label{tilting2}
	Assume that $\mathcal{H}$ is a hereditary abelian category with a tilting object. Then $T$ is a tilting object in $\mathcal{H}$ if and only if ${\rm Ext}_{\mathbb{X}}^{1}(T,T)=0$ and the number of non-isomorphic indecomposable direct summands of $T$ equals the rank of the Grothendieck group $K_{0}(\mathcal {H})$.
\end{prop}

Thus, by Proposition \ref{the properties of coherent sheaves}, we know that each tilting sheaf in ${\rm coh}\mbox{-}\mathbb{X}$  has $n+3$ indecomposable direct  summands.

Consider an indecomposable sheaf $X$ in ${\rm coh}\mbox{-}\mathbb{X}$. Note that if $X$ lies in a homogeneous tube or in a tube of $\tau$-period  $2$ with length greater than or equal to 2, then ${\rm Ext}^{1}(X,X) \neq 0$. Therefore, for a tilting sheaf $T$ in ${\rm coh}\mbox{-}\mathbb{X}$, its indecomposable direct summands can only come from simple sheaves in the tubes of $\tau$-period  $2$, vector bundles, or sheaves in the tube of $\tau$-period  $n$.

\begin{rem}
For the four simple objects in the  tubes of $\tau$-period  $2$, under the bijection $\widehat{\phi}$, we have the following correspondence:
\[
(1, L)^+ \leftrightarrow S_{\infty, \overline{0}}, \quad
(1, L)^- \leftrightarrow S_{\infty, \overline{1}}, \quad
(-1, L)^+ \leftrightarrow S_{0, \overline{0}}, \quad
(-1, L)^- \leftrightarrow S_{0, \overline{1}}.
\]

To explicitly represent these four parameterized curves $(1, L)^+$, $(1, L)^-$, $(-1, L)^+$, $(-1, L)^-$ on the cylindrical surface, we define two semicircles:
 \[L_0=\{ (\theta, \frac{1}{2}) \mid \theta \in [0, \pi] \},\quad L_1=\{ (\theta, \frac{1}{2}) \mid \theta \in [\pi, 2\pi] \}.\]
These semicircles form the pair $\{L_0, L_1\}$, which will be used for visualization. Specifically, when we depict the skew-curves on the cylindrical surface, we draw $L_0$ and $L_1$, and label them with the parameter pair $(\epsilon_1, \epsilon_2)$, where $\epsilon_1, \epsilon_2 \in \{+, -\}$. The labeling corresponds to the following association:
\[
(1, L)^+ \leftrightarrow (+,+), \quad
(1, L)^- \leftrightarrow (-, -), \quad
(-1, L)^+ \leftrightarrow (+, -), \quad
(-1, L)^- \leftrightarrow (-, +).
\]
For simplicity in textual descriptions, we will refer to these curves only by their parameter pairs $(\epsilon_1, \epsilon_2)$. We denote the set of all parameter pairs as $\mathbf{C}_{*}$:
\[\mathbf{C}_{*}=\{(\epsilon_1, \epsilon_2)\mid  \epsilon_1, \epsilon_2\in \{+,-\}\}\subseteq\mathbf{C}_{\mathbf{k}^*}^{\text{pw}}.\]
\end{rem}

\subsection{Compatibility of skew-curves}
Note that skew-curves are not curves on $(\mathcal{S}, M)$, so the original definition of curve compatibility does not apply. We give the definition of compatibility for skew-curves on $(\mathcal{S}, M,\sigma)$: 
\begin{defn}\label{curvecompatible}

Two skew-curves $\widehat{\gamma}_1$ and $\widehat{\gamma}_2$ in $\widehat{\mathbf{C}}$ are called \emph{compatible} if
   $$
   \operatorname{dim}_{\mathbf{k}} \operatorname{Ext}_{\mathbb{X}}^1(\widehat{\phi}(\widehat{\gamma}_1), \widehat{\phi}(\widehat{\gamma}_2)) = \operatorname{dim}_{\mathbf{k}} \operatorname{Ext}_{\mathbb{X}}^1(\widehat{\phi}(\widehat{\gamma}_2), \widehat{\phi}(\widehat{\gamma}_1)) = 0.
   $$
A self-compatible skew-curve $\widehat{\gamma}$ is called a \emph{skew-arc}.
\end{defn}
 
\begin{exm}
 Based on the Table \ref{tab:dim_ext1} and the bijection $\widehat{\phi}$, the parameter pair $(\epsilon_1, \epsilon_2)$ represents the skew-curve in $\mathbf{C}_*$ that is compatible with all skew-curves in $\bigcup_{i=1}^2\mathbf{C}_{\mathtt{b}, \times_i}^{\epsilon_i}$. 
 \begin{longtblr}[
  caption = {The value of $\operatorname{dim}_{\mathbf{k}} {\rm Ext}^1_\mathbb{X}(X, Y)$},
label = {tab:dim_ext1}]{
  width = \linewidth,
  colspec = {Q[388]Q[175]Q[165]Q[137]Q[137]},
  cells = {c},
  vline{2} = {-}{},
  hline{1,6} = {-}{0.08em},
  hline{2} = {-}{},
}
\diagbox[innerwidth=\linewidth]{$Y $}{$X $} & $S_{\infty,\overline{0}}$ & $S_{\infty,\overline{1}} $ & $S_{0,\overline{0}} $ & $S_{0,\overline{1}} $ \\
$\mathcal{O}_\mathbb{X}(i \vec{x}_3)$                               & $1$            & $0 $            & $1 $       & $0 $       \\
$\mathcal{O}_\mathbb{X}(\vec{x}_1 - \vec{x}_2 + i \vec{x}_3)$       & $0 $           & $1 $            & $0 $       & $1 $       \\
$\mathcal{O}_\mathbb{X}(\vec{x}_1+ i \vec{x}_3)$                    & $0 $           & $1 $            & $1 $       & $0 $       \\
$\mathcal{O}_\mathbb{X}( \vec{x}_2 + i \vec{x}_3)$                  & $1 $           & $0 $            & $0 $       & $1 $       
\end{longtblr}
\end{exm}

 We will show that the compatibility between skew-curves is closely related to whether their associated curves intersect. Before that, we need some preparations.  
 
\begin{defn}
Let $A$ and $B$ be sets of curves on $(\mathcal{S},M)$. Define the total number of intersection points between curves in $A$ and $B$ as:
\[
I(A, B) = \sum_{\gamma \in A} \sum_{\delta \in B} |\gamma \cap \delta|.
\]
\end{defn}

 \begin{lem}(\cite{MR3313495})\label{extension-free}
	For any line bundle $L$ and $\vec{x}\in\mathbb{L}(2,2,n)$, the direct sum $L\oplus L(\vec{x})$ is extension-free if and only if $-\vec{c}\leq \vec{x}\leq \vec{c}$. 
\end{lem}
\begin{rem}\label{relative}
The above lemma states that under the bijection $\widehat{\phi}$, the relative positions of $\widehat{\phi}^{-1}(L)$ and $\widehat{\phi}^{-1}(L(\vec{x}))$  $(\vec{x}\neq 0)$ can be represented as follows:
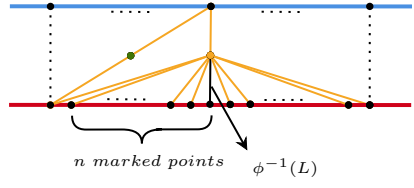
\begin{figure}[H]
	\centering

\tikzset{every picture/.style={line width=0.75pt}}  

\begin{tikzpicture}[x=0.75pt,y=0.75pt,yscale=-1,xscale=1]

\draw [color={rgb, 255:red, 245; green, 166; blue, 35 }  ,draw opacity=1 ]   (289.93,90.28) -- (210.47,65.26) ;
\draw [shift={(210.47,65.26)}, rotate = 197.48] [color={rgb, 255:red, 245; green, 166; blue, 35 }  ,draw opacity=1 ][fill={rgb, 255:red, 245; green, 166; blue, 35 }  ,fill opacity=1 ][line width=0.75]      (0, 0) circle [x radius= 1.34, y radius= 1.34]   ;
\draw [shift={(289.93,90.28)}, rotate = 197.48] [color={rgb, 255:red, 245; green, 166; blue, 35 }  ,draw opacity=1 ][fill={rgb, 255:red, 245; green, 166; blue, 35 }  ,fill opacity=1 ][line width=0.75]      (0, 0) circle [x radius= 1.34, y radius= 1.34]   ;
 
\draw [color={rgb, 255:red, 74; green, 144; blue, 226 }  ,draw opacity=1 ][line width=1.5]    (110.4,40.6) -- (311.03,40.68) ;
 
\draw [color={rgb, 255:red, 245; green, 166; blue, 35 }  ,draw opacity=1 ]   (210.47,65.26) -- (230.34,89.87) ;
\draw [shift={(230.34,89.87)}, rotate = 51.09] [color={rgb, 255:red, 245; green, 166; blue, 35 }  ,draw opacity=1 ][fill={rgb, 255:red, 245; green, 166; blue, 35 }  ,fill opacity=1 ][line width=0.75]      (0, 0) circle [x radius= 1.34, y radius= 1.34]   ;
\draw [shift={(210.47,65.26)}, rotate = 51.09] [color={rgb, 255:red, 245; green, 166; blue, 35 }  ,draw opacity=1 ][fill={rgb, 255:red, 245; green, 166; blue, 35 }  ,fill opacity=1 ][line width=0.75]      (0, 0) circle [x radius= 1.34, y radius= 1.34]   ;
 
\draw [color={rgb, 255:red, 208; green, 2; blue, 27 }  ,draw opacity=1 ][line width=1.5]    (110,90.2) -- (257.03,90.26) -- (310.63,90.28) ;
 
\draw [color={rgb, 255:red, 245; green, 166; blue, 35 }  ,draw opacity=1 ]   (130.63,90.28) -- (210.47,65.26) ;
\draw [shift={(210.47,65.26)}, rotate = 342.6] [color={rgb, 255:red, 245; green, 166; blue, 35 }  ,draw opacity=1 ][fill={rgb, 255:red, 245; green, 166; blue, 35 }  ,fill opacity=1 ][line width=0.75]      (0, 0) circle [x radius= 1.34, y radius= 1.34]   ;
\draw [shift={(130.63,90.28)}, rotate = 342.6] [color={rgb, 255:red, 245; green, 166; blue, 35 }  ,draw opacity=1 ][fill={rgb, 255:red, 245; green, 166; blue, 35 }  ,fill opacity=1 ][line width=0.75]      (0, 0) circle [x radius= 1.34, y radius= 1.34]   ;
 
\draw   (141.43,95.63) .. controls (141.44,100.3) and (143.78,102.63) .. (148.45,102.62) -- (166.25,102.58) .. controls (172.92,102.56) and (176.26,104.88) .. (176.27,109.55) .. controls (176.26,104.88) and (179.58,102.54) .. (186.25,102.53)(183.25,102.54) -- (204.05,102.49) .. controls (208.72,102.48) and (211.04,100.15) .. (211.03,95.48) ;
 
\draw [color={rgb, 255:red, 245; green, 166; blue, 35 }  ,draw opacity=1 ]   (210.72,40.64) -- (130.63,90.28) ;
\draw [shift={(130.63,90.28)}, rotate = 148.21] [color={rgb, 255:red, 245; green, 166; blue, 35 }  ,draw opacity=1 ][fill={rgb, 255:red, 245; green, 166; blue, 35 }  ,fill opacity=1 ][line width=0.75]      (0, 0) circle [x radius= 1.34, y radius= 1.34]   ;
\draw [shift={(170.67,65.46)}, rotate = 148.21] [color={rgb, 255:red, 245; green, 166; blue, 35 }  ,draw opacity=1 ][fill={rgb, 255:red, 245; green, 166; blue, 35 }  ,fill opacity=1 ][line width=0.75]      (0, 0) circle [x radius= 1.34, y radius= 1.34]   ;
\draw [shift={(210.72,40.64)}, rotate = 148.21] [color={rgb, 255:red, 245; green, 166; blue, 35 }  ,draw opacity=1 ][fill={rgb, 255:red, 245; green, 166; blue, 35 }  ,fill opacity=1 ][line width=0.75]      (0, 0) circle [x radius= 1.34, y radius= 1.34]   ;
 
\draw [color={rgb, 255:red, 245; green, 166; blue, 35 }  ,draw opacity=1 ]   (210.47,65.26) -- (200.63,90.28) ;
\draw [shift={(200.63,90.28)}, rotate = 111.47] [color={rgb, 255:red, 245; green, 166; blue, 35 }  ,draw opacity=1 ][fill={rgb, 255:red, 245; green, 166; blue, 35 }  ,fill opacity=1 ][line width=0.75]      (0, 0) circle [x radius= 1.34, y radius= 1.34]   ;
\draw [shift={(210.47,65.26)}, rotate = 111.47] [color={rgb, 255:red, 245; green, 166; blue, 35 }  ,draw opacity=1 ][fill={rgb, 255:red, 245; green, 166; blue, 35 }  ,fill opacity=1 ][line width=0.75]      (0, 0) circle [x radius= 1.34, y radius= 1.34]   ;
 
\draw [color={rgb, 255:red, 245; green, 166; blue, 35 }  ,draw opacity=1 ]   (210.47,65.26) -- (141.03,90.28) ;
\draw [shift={(141.03,90.28)}, rotate = 160.19] [color={rgb, 255:red, 245; green, 166; blue, 35 }  ,draw opacity=1 ][fill={rgb, 255:red, 245; green, 166; blue, 35 }  ,fill opacity=1 ][line width=0.75]      (0, 0) circle [x radius= 1.34, y radius= 1.34]   ;
\draw [shift={(210.47,65.26)}, rotate = 160.19] [color={rgb, 255:red, 245; green, 166; blue, 35 }  ,draw opacity=1 ][fill={rgb, 255:red, 245; green, 166; blue, 35 }  ,fill opacity=1 ][line width=0.75]      (0, 0) circle [x radius= 1.34, y radius= 1.34]   ;
 
\draw [color={rgb, 255:red, 245; green, 166; blue, 35 }  ,draw opacity=1 ]   (210.47,65.26) -- (190.63,90.28) ;
\draw [shift={(190.63,90.28)}, rotate = 128.42] [color={rgb, 255:red, 245; green, 166; blue, 35 }  ,draw opacity=1 ][fill={rgb, 255:red, 245; green, 166; blue, 35 }  ,fill opacity=1 ][line width=0.75]      (0, 0) circle [x radius= 1.34, y radius= 1.34]   ;
\draw [shift={(210.47,65.26)}, rotate = 128.42] [color={rgb, 255:red, 245; green, 166; blue, 35 }  ,draw opacity=1 ][fill={rgb, 255:red, 245; green, 166; blue, 35 }  ,fill opacity=1 ][line width=0.75]      (0, 0) circle [x radius= 1.34, y radius= 1.34]   ;
 
\draw  [dash pattern={on 0.84pt off 2.51pt}]  (177.98,87.48) -- (158.65,87.32) ;
 
\draw [color={rgb, 255:red, 245; green, 166; blue, 35 }  ,draw opacity=1 ]   (210.72,40.64) -- (210.47,65.26) ;
 
\draw    (210.47,65.26) -- (210.23,89.88) ;
\draw [shift={(210.23,89.88)}, rotate = 90.56] [color={rgb, 255:red, 0; green, 0; blue, 0 }  ][fill={rgb, 255:red, 0; green, 0; blue, 0 }  ][line width=0.75]      (0, 0) circle [x radius= 1.34, y radius= 1.34]   ;
\draw [shift={(210.47,65.26)}, rotate = 90.56] [color={rgb, 255:red, 0; green, 0; blue, 0 }  ][fill={rgb, 255:red, 0; green, 0; blue, 0 }  ][line width=0.75]      (0, 0) circle [x radius= 1.34, y radius= 1.34]   ;
 
\draw    (141.03,90.28) ;
\draw [shift={(141.03,90.28)}, rotate = 0] [color={rgb, 255:red, 0; green, 0; blue, 0 }  ][fill={rgb, 255:red, 0; green, 0; blue, 0 }  ][line width=0.75]      (0, 0) circle [x radius= 1.34, y radius= 1.34]   ;
\draw [shift={(141.03,90.28)}, rotate = 0] [color={rgb, 255:red, 0; green, 0; blue, 0 }  ][fill={rgb, 255:red, 0; green, 0; blue, 0 }  ][line width=0.75]      (0, 0) circle [x radius= 1.34, y radius= 1.34]   ;
 
\draw    (210.72,40.64) ;
\draw [shift={(210.72,40.64)}, rotate = 0] [color={rgb, 255:red, 0; green, 0; blue, 0 }  ][fill={rgb, 255:red, 0; green, 0; blue, 0 }  ][line width=0.75]      (0, 0) circle [x radius= 1.34, y radius= 1.34]   ;
\draw [shift={(210.72,40.64)}, rotate = 0] [color={rgb, 255:red, 0; green, 0; blue, 0 }  ][fill={rgb, 255:red, 0; green, 0; blue, 0 }  ][line width=0.75]      (0, 0) circle [x radius= 1.34, y radius= 1.34]   ;
 
\draw    (170.67,65.46) ;
\draw [shift={(170.67,65.46)}, rotate = 0] [color={rgb, 255:red, 0; green, 0; blue, 0 }  ][fill={rgb, 255:red, 0; green, 0; blue, 0 }  ][line width=0.75]      (0, 0) circle [x radius= 1.34, y radius= 1.34]   ;
\draw [shift={(170.67,65.46)}, rotate = 0] [color={rgb, 255:red, 0; green, 0; blue, 0 }  ][fill={rgb, 255:red, 0; green, 0; blue, 0 }  ][line width=0.75]      (0, 0) circle [x radius= 1.34, y radius= 1.34]   ;
 
\draw [color={rgb, 255:red, 245; green, 166; blue, 35 }  ,draw opacity=1 ]   (210.47,65.26) -- (220.34,89.98) ;
\draw [shift={(220.34,89.98)}, rotate = 68.24] [color={rgb, 255:red, 245; green, 166; blue, 35 }  ,draw opacity=1 ][fill={rgb, 255:red, 245; green, 166; blue, 35 }  ,fill opacity=1 ][line width=0.75]      (0, 0) circle [x radius= 1.34, y radius= 1.34]   ;
\draw [shift={(210.47,65.26)}, rotate = 68.24] [color={rgb, 255:red, 245; green, 166; blue, 35 }  ,draw opacity=1 ][fill={rgb, 255:red, 245; green, 166; blue, 35 }  ,fill opacity=1 ][line width=0.75]      (0, 0) circle [x radius= 1.34, y radius= 1.34]   ;
 
\draw [color={rgb, 255:red, 245; green, 166; blue, 35 }  ,draw opacity=1 ]   (210.47,65.26) -- (279.37,90.25) ;
\draw [shift={(279.37,90.25)}, rotate = 19.94] [color={rgb, 255:red, 245; green, 166; blue, 35 }  ,draw opacity=1 ][fill={rgb, 255:red, 245; green, 166; blue, 35 }  ,fill opacity=1 ][line width=0.75]      (0, 0) circle [x radius= 1.34, y radius= 1.34]   ;
\draw [shift={(210.47,65.26)}, rotate = 19.94] [color={rgb, 255:red, 245; green, 166; blue, 35 }  ,draw opacity=1 ][fill={rgb, 255:red, 245; green, 166; blue, 35 }  ,fill opacity=1 ][line width=0.75]      (0, 0) circle [x radius= 1.34, y radius= 1.34]   ;
 
\draw  [dash pattern={on 0.84pt off 2.51pt}]  (237.42,87.74) -- (256.76,87.69) ;
 
\draw    (210.47,65.26) ;
\draw [shift={(210.47,65.26)}, rotate = 0] [color={rgb, 255:red, 0; green, 0; blue, 0 }  ][fill={rgb, 255:red, 0; green, 0; blue, 0 }  ][line width=0.75]      (0, 0) circle [x radius= 1.34, y radius= 1.34]   ;
\draw [shift={(210.47,65.26)}, rotate = 0] [color={rgb, 255:red, 0; green, 0; blue, 0 }  ][fill={rgb, 255:red, 0; green, 0; blue, 0 }  ][line width=0.75]      (0, 0) circle [x radius= 1.34, y radius= 1.34]   ;
 
\draw    (200.63,90.28) ;
\draw [shift={(200.63,90.28)}, rotate = 0] [color={rgb, 255:red, 0; green, 0; blue, 0 }  ][fill={rgb, 255:red, 0; green, 0; blue, 0 }  ][line width=0.75]      (0, 0) circle [x radius= 1.34, y radius= 1.34]   ;
\draw [shift={(200.63,90.28)}, rotate = 0] [color={rgb, 255:red, 0; green, 0; blue, 0 }  ][fill={rgb, 255:red, 0; green, 0; blue, 0 }  ][line width=0.75]      (0, 0) circle [x radius= 1.34, y radius= 1.34]   ;
 
\draw    (279.37,90.25) ;
\draw [shift={(279.37,90.25)}, rotate = 0] [color={rgb, 255:red, 0; green, 0; blue, 0 }  ][fill={rgb, 255:red, 0; green, 0; blue, 0 }  ][line width=0.75]      (0, 0) circle [x radius= 1.34, y radius= 1.34]   ;
\draw [shift={(279.37,90.25)}, rotate = 0] [color={rgb, 255:red, 0; green, 0; blue, 0 }  ][fill={rgb, 255:red, 0; green, 0; blue, 0 }  ][line width=0.75]      (0, 0) circle [x radius= 1.34, y radius= 1.34]   ;
 
\draw    (289.93,90.28) ;
\draw [shift={(289.93,90.28)}, rotate = 0] [color={rgb, 255:red, 0; green, 0; blue, 0 }  ][fill={rgb, 255:red, 0; green, 0; blue, 0 }  ][line width=0.75]      (0, 0) circle [x radius= 1.34, y radius= 1.34]   ;
\draw [shift={(289.93,90.28)}, rotate = 0] [color={rgb, 255:red, 0; green, 0; blue, 0 }  ][fill={rgb, 255:red, 0; green, 0; blue, 0 }  ][line width=0.75]      (0, 0) circle [x radius= 1.34, y radius= 1.34]   ;
 
\draw    (211,81) -- (227.14,109.99) ;
\draw [shift={(228.6,112.61)}, rotate = 240.89] [fill={rgb, 255:red, 0; green, 0; blue, 0 }  ][line width=0.08]  [draw opacity=0] (5.36,-2.57) -- (0,0) -- (5.36,2.57) -- (3.56,0) -- cycle    ;
 
\draw  [dash pattern={on 0.84pt off 2.51pt}]  (130.67,40.67) -- (130.63,90.28) ;
\draw [shift={(130.63,90.28)}, rotate = 90.04] [color={rgb, 255:red, 0; green, 0; blue, 0 }  ][fill={rgb, 255:red, 0; green, 0; blue, 0 }  ][line width=0.75]      (0, 0) circle [x radius= 1.34, y radius= 1.34]   ;
\draw [shift={(130.67,40.67)}, rotate = 90.04] [color={rgb, 255:red, 0; green, 0; blue, 0 }  ][fill={rgb, 255:red, 0; green, 0; blue, 0 }  ][line width=0.75]      (0, 0) circle [x radius= 1.34, y radius= 1.34]   ;
 
\draw  [dash pattern={on 0.84pt off 2.51pt}]  (289.97,40.66) -- (289.93,90.28) ;
\draw [shift={(289.93,90.28)}, rotate = 90.04] [color={rgb, 255:red, 0; green, 0; blue, 0 }  ][fill={rgb, 255:red, 0; green, 0; blue, 0 }  ][line width=0.75]      (0, 0) circle [x radius= 1.34, y radius= 1.34]   ;
\draw [shift={(289.97,40.66)}, rotate = 90.04] [color={rgb, 255:red, 0; green, 0; blue, 0 }  ][fill={rgb, 255:red, 0; green, 0; blue, 0 }  ][line width=0.75]      (0, 0) circle [x radius= 1.34, y radius= 1.34]   ;
 
\draw [color={rgb, 255:red, 245; green, 166; blue, 35 }  ,draw opacity=1 ]   (210.47,65.26) ;
\draw [shift={(210.47,65.26)}, rotate = 0] [color={rgb, 255:red, 245; green, 166; blue, 35 }  ,draw opacity=1 ][fill={rgb, 255:red, 245; green, 166; blue, 35 }  ,fill opacity=1 ][line width=0.75]      (0, 0) circle [x radius= 1.34, y radius= 1.34]   ;
 
\draw [color={rgb, 255:red, 65; green, 117; blue, 5 }  ,draw opacity=1 ]   (170.67,65.46) ;
\draw [shift={(170.67,65.46)}, rotate = 0] [color={rgb, 255:red, 65; green, 117; blue, 5 }  ,draw opacity=1 ][fill={rgb, 255:red, 65; green, 117; blue, 5 }  ,fill opacity=1 ][line width=0.75]      (0, 0) circle [x radius= 1.34, y radius= 1.34]   ;
 
\draw  [dash pattern={on 0.84pt off 2.51pt}]  (179.98,43.48) -- (160.65,43.32) ;
 
\draw  [dash pattern={on 0.84pt off 2.51pt}]  (254.98,43.48) -- (235.65,43.32) ;
 
\draw    (230.34,89.87) ;
\draw [shift={(230.34,89.87)}, rotate = 0] [color={rgb, 255:red, 0; green, 0; blue, 0 }  ][fill={rgb, 255:red, 0; green, 0; blue, 0 }  ][line width=0.75]      (0, 0) circle [x radius= 1.34, y radius= 1.34]   ;
\draw [shift={(230.34,89.87)}, rotate = 0] [color={rgb, 255:red, 0; green, 0; blue, 0 }  ][fill={rgb, 255:red, 0; green, 0; blue, 0 }  ][line width=0.75]      (0, 0) circle [x radius= 1.34, y radius= 1.34]   ;
 
\draw    (190.63,90.28) ;
\draw [shift={(190.63,90.28)}, rotate = 0] [color={rgb, 255:red, 0; green, 0; blue, 0 }  ][fill={rgb, 255:red, 0; green, 0; blue, 0 }  ][line width=0.75]      (0, 0) circle [x radius= 1.34, y radius= 1.34]   ;
\draw [shift={(190.63,90.28)}, rotate = 0] [color={rgb, 255:red, 0; green, 0; blue, 0 }  ][fill={rgb, 255:red, 0; green, 0; blue, 0 }  ][line width=0.75]      (0, 0) circle [x radius= 1.34, y radius= 1.34]   ;
 
\draw    (220.34,89.98) ;
\draw [shift={(220.34,89.98)}, rotate = 0] [color={rgb, 255:red, 0; green, 0; blue, 0 }  ][fill={rgb, 255:red, 0; green, 0; blue, 0 }  ][line width=0.75]      (0, 0) circle [x radius= 1.34, y radius= 1.34]   ;
\draw [shift={(220.34,89.98)}, rotate = 0] [color={rgb, 255:red, 0; green, 0; blue, 0 }  ][fill={rgb, 255:red, 0; green, 0; blue, 0 }  ][line width=0.75]      (0, 0) circle [x radius= 1.34, y radius= 1.34]   ;

\draw (140.6,113.8) node [anchor=north west][inner sep=0.75pt]  [font=\tiny]  {$n\ marked\ points\ $};
 
\draw (229.93,115.34) node [anchor=north west][inner sep=0.75pt]  [font=\tiny]  {$\phi ^{-1}( L)$};

c\end{tikzpicture}
\end{figure}
\noindent In this illustration, bold blue and red curves represent different boundaries of $(\mathcal{S}, M)$, while orange and green points denote various $\times_i$. The black skew-curve represents $\widehat{\phi}^{-1}(L)$, while the orange skew-curves depict all possible cases of $\widehat{\phi}^{-1}(L(\vec{x}))$.
\end{rem}

The following notation of the set of curves associated with a skew-curve can be regarded, in a certain sense, as the reverse process of constructing skew-curves: 
\begin{defn}
For a skew-curve $\widehat{\gamma} \in \widehat{\mathbf{C}}$, the  set  $\mathbf{C}(\widehat{\gamma})$  of curves associated with $\widehat{\gamma}$ is defined as:
\begin{itemize}
    \item  If $\widehat{\gamma}=\gamma^+$ or $\gamma^-$ for some $\sigma$-fixed $\gamma\in \mathbf{C}_{\mathtt{b}}$, then $\mathbf{C}(\widehat{\gamma})=\{\gamma\}$;
    \item  If $\widehat{\gamma}=\{\gamma, \sigma(\gamma )\}$ for some not $\sigma$-fixed $\gamma\in \mathbf{C}_{\mathtt{b}}\cup \mathbf{C}_{\mathtt{p}}$, then $\mathbf{C}(\widehat{\gamma})=\{\gamma, \sigma(\gamma )\}$;
    \item  If $\widehat{\gamma}\in \mathbf{C}_{\mathbf{k}^*}^{\text{pw}}\cup \mathbf{C}_{\mathbf{k}^*}^{\text{sp}}$, then $\mathbf{C}(\widehat{\gamma})=\{L_0, L_1\}$.
\end{itemize}
\end{defn}

\begin{prop}\label{prop:6.4}
    Two skew-curves $\widehat{\gamma}_1$, $\widehat{\gamma}_2\in\widehat{\mathbf{C}}\setminus \{\mathbf{C}_{\mathbf{k}^*}^{\text{pw}} \cup \mathbf{C}_{\mathbf{k}^*}^{\text{sp}}\}$  are  compatible if and only if one of the following conditions holds:
    \begin{itemize}
        \item[\text{(T1)}] $I(\mathbf{C}(\widehat{\gamma}_1),\mathbf{C}(\widehat{\gamma}_2)) = 0$.
        \item[\text{(T2)}]  $I(\mathbf{C}(\widehat{\gamma}_1),\mathbf{C}(\widehat{\gamma}_2)) = 1$ and $\widehat{\gamma}_1$ and $\widehat{\gamma}_2$ both belong to either $\mathbf{C}_{\mathtt{b},\times}^+$ or $\mathbf{C}_{\mathtt{b},\times}^-$ for some $\times \in \mathcal{X}$. 
    \end{itemize}
\end{prop}
\begin{proof}
In this proof, we set 
\[\operatorname{Ext}^1_{\mathbb{X}^{\mathbb{Z}(\vec{x}_1 - \vec{x}_2)}}(-,-):=\operatorname{Ext}^1_{({\rm coh}\mbox{-}\mathbb{X})^{\mathbb{Z}(\vec{x}_1 - \vec{x}_2)}}(-,-)\] 
and 
\[\operatorname{Ext}^1_{\mathbb{Y}^{G}}(-,-):=\operatorname{Ext}^1_{({\rm coh}\mbox{-}\mathbb{Y})^{G}}(-,-). \]
Let $H_1$ and $H_2$ be the equivalences in Proposition \ref{(2,2,n) and (n,n)} (a) and (b), respectively. Let  ${\rm Ind}_1: {\rm coh}\mbox{-}\mathbb{X} \rightarrow ({\rm coh}\mbox{-}\mathbb{X})^{\mathbb{Z}(\vec{x}_1-\vec{x}_2)}$ and ${\rm Ind}_2: {\rm coh}\mbox{-}\mathbb{Y} \rightarrow ({\rm coh}\mbox{-}\mathbb{Y})^{G}$ be the induction functors. Fix  the forgetful functor $F:  ({\rm coh}\mbox{-}\mathbb{X})^{\mathbb{Z}(\vec{x}_1-\vec{x}_2)} \rightarrow {\rm coh}\mbox{-}\mathbb{X}$. We have two cases: 

(1) $\widehat{\gamma}_1,\widehat{\gamma}_2$  are in $\mathbf{C}_{\mathtt{b} ,\mathcal{X}}$. By Remark \ref{relative}, the skew-curves $\widehat{\gamma}_1$, $\widehat{\gamma}_2$ are compatible if and only if $\mathbf{C}(\widehat{\gamma}_1)=\mathbf{C}(\widehat{\gamma}_2)$ (in which case $I(\mathbf{C}(\widehat{\gamma}_1), \mathbf{C}(\widehat{\gamma}_2)) = 0$) or if they satisfy condition \text{(T2)}. 
    
(2) At least one $\widehat{\gamma}_i \notin \mathbf{C}_{\mathtt{b} ,\mathcal{X}}$. Note that for any skew-curve $\widehat{\gamma}$ in $\mathbf{C}_{\mathtt{b} ,\mathcal{X}} \cup \mathbf{C}_{\mathtt{b}}^\sigma \cup \mathbf{C}_{\mathtt{p}}^\sigma$, $\widehat{\phi}(\widehat{\gamma})$ is always an indecomposable direct summand of $H_2 \cdot {\rm Ind}_2(\phi(\gamma))$, where $\gamma \in \mathbf{C}(\widehat{\gamma})$. Moreover, under the bijection $\phi$, the indecomposable direct summands of $F\cdot{\rm Ind}_2(\widehat{\phi}(\widehat{\gamma_i}))$ correspond to curves in $\mathbf{C}(\widehat{\gamma}_i)$ for each $i=1,2$. Therefore, by \cite[Lemma 3.4]{chen2014noteserredualityequivariantization}, we have
\begin{align*}
\operatorname{dim}_{\mathbf{k}} \operatorname{Ext}^1_{\mathbb{X}}(\widehat{\phi}(\widehat{\gamma}_1), \widehat{\phi}(\widehat{\gamma}_2)) & \leq \operatorname{dim}_{\mathbf{k}} \operatorname{Ext}^1_{\mathbb{Y}^{G}}({\rm Ind}_2(\phi(\gamma_1)), {\rm Ind}_2(\phi(\gamma_2))) \\
& \leq \operatorname{dim}_{\mathbf{k}} \operatorname{Ext}^1_{\mathbb{Y}}\left(\bigoplus_{\gamma \in \mathbf{C}(\widehat{\gamma_1})}\phi(\gamma), \bigoplus_{\delta \in \mathbf{C}(\widehat{\gamma_2})}\phi(\delta)\right),
\end{align*}

\noindent where $\gamma_1 \in \mathbf{C}(\widehat{\gamma_1})$ and $\gamma_2 \in \mathbf{C}(\widehat{\gamma_2})$. By \cite[Theorem 3.10]{Chen2023GeometricModel}, the condition \[I(\mathbf{C}(\widehat{\gamma}_1), \mathbf{C}(\widehat{\gamma}_2)) = 0\]
implies that the right-hand side is zero. It follows that $\operatorname{Ext}^1_{\mathbb{X}}(\phi(\widehat{\gamma}_1), \phi(\widehat{\gamma}_2)) = 0$.

Conversely, if $\operatorname{Ext}^1_{\mathbb{X}}(\phi(\widehat{\gamma}_1), \phi(\widehat{\gamma}_2)) = 0$, without loss of generality, we assume $\widehat{\gamma}_1 \notin \mathbf{C}_{\mathtt{b} ,\mathcal{X}}$. Then, by Proposition \ref{L-action}, we have $\widehat{\phi}(\widehat{\gamma}_1) = \widehat{\phi}(\widehat{\gamma}_1)(\vec{x}_1 - \vec{x}_2)$. Hence, $$\operatorname{Ext}^1_{\mathbb{X}}(\phi(\widehat{\gamma}_1), \phi(\widehat{\gamma}_2)(\vec{x}_1 - \vec{x}_2)) = 0.$$ 

\noindent Note that an indecomposable direct summand of $F \cdot {\rm Ind}_2(\phi(\widehat{\gamma}_i))$ is either $\widehat{\phi}(\widehat{\gamma}_i)$ or $\widehat{\phi}(\widehat{\gamma}_i)(\vec{x}_1 - \vec{x}_2)$ for $i=1,2$. By \cite[Lemma 3.4]{chen2014noteserredualityequivariantization}, we have
\begin{align*}
0 &= \operatorname{dim}_{\mathbf{k}} \operatorname{Ext}^1_{\mathbb{X}}(F \cdot {\rm Ind}_1(\widehat{\phi}(\widehat{\gamma}_1)), F \cdot {\rm Ind}_1(\widehat{\phi}(\widehat{\gamma}_2))) \\
& \geq \operatorname{dim}_{\mathbf{k}} \operatorname{Ext}^1_{\mathbb{X}^{\mathbb{Z}(\vec{x}_1 - \vec{x}_2)}}({\rm Ind}_2(\widehat{\phi}(\widehat{\gamma}_1)), {\rm Ind}_2(\widehat{\phi}(\widehat{\gamma}_2))) \\
& = \operatorname{dim}_{\mathbf{k}} \operatorname{Ext}^1_{\mathbb{Y}}(H_1 \cdot {\rm Ind}_2(\widehat{\phi}(\widehat{\gamma}_1)), H_1 \cdot {\rm Ind}_2(\widehat{\phi}(\widehat{\gamma}_2))). 
\end{align*}
\noindent For each $i=1,2$, under the bijection $\phi$, the indecomposable direct summands of $H_1 \cdot {\rm Ind}_2(\widehat{\phi}(\widehat{\gamma}_i))$ correspond to curves in $\mathbf{C}(\widehat{\gamma}_i)$. By \cite[Theorem 3.10]{Chen2023GeometricModel}, we conclude that $I(\mathbf{C}(\widehat{\gamma}_1), \mathbf{C}(\widehat{\gamma}_2)) = 0$.
\end{proof}
For the skew-curves in $\mathbf{C}_*$, we observe the following:
\begin{lem}\label{lem:6.4}
For any $(\epsilon_1, \epsilon_2)\in \mathbf{C}_*$,  the skew-curve  $(\epsilon_1, \epsilon_2)$ is not compatible with all $\widehat{\gamma}\in \mathbf{C}_{\mathtt{b}}^\sigma$. 
\end{lem}
\begin{proof}
    We only prove the case where $(\epsilon_1, \epsilon_2) = (+, +)$, as the others follow similarly. Since $\widehat{\phi}(\widehat{\gamma})$ is an extension bundle, by definition of extension bundle, $\widehat{\phi}(\widehat{\gamma})$ can be expressed as $E_L\langle \vec{x} \rangle$ for some line bundle $L$ and $0 \leq \vec{x} \leq \vec{\delta}$. By \cite[Proposition 2.3 (i)]{Dong2025Two}, $E_L\langle \vec{x} \rangle$ can also be written in the form $E_{L(\vec{x}_1 - \vec{x}_2)}\langle \vec{x} \rangle$. Note that either ${\rm Ext}^1_\mathbb{X}(S_{\infty, \overline{0}}, L) \neq 0$ or ${\rm Ext}^1_\mathbb{X}(S_{\infty, \overline{0}}, L(\vec{x}_1 - \vec{x}_2)) \neq 0$. Without loss of generality, we assume ${\rm Ext}^1_\mathbb{X}(S_{\infty, \overline{0}}, L(\vec{x}_1 - \vec{x}_2)) \neq 0$. Then we obtain the following pushout commutative diagram with exact rows and columns:
\[\begin{tikzcd}
	&& 0 & 0 \\
	&& {L(\vec{x}-\vec{x}_1)} & {L(\vec{x}-\vec{x}_1)} \\
	0 & {L(\vec{\omega})} & {E_L\langle\vec{x}\rangle} & {L(\vec{x})} & 0 \\
	0 & {L(\vec{\omega})} & {L(\vec{x}_2-\vec{x}_3)} & {S_{\infty, \overline{0}}} & 0 \\
	&& 0 & 0
	\arrow[from=1-3, to=2-3]
	\arrow[from=1-4, to=2-4]
	\arrow[Rightarrow, no head, from=2-3, to=2-4]
	\arrow[from=2-3, to=3-3]
	\arrow[from=2-4, to=3-4]
	\arrow[from=3-1, to=3-2]
	\arrow[from=3-2, to=3-3]
	\arrow[Rightarrow, no head, from=3-2, to=4-2]
	\arrow[from=3-3, to=3-4]
	\arrow[from=3-3, to=4-3]
	\arrow[from=3-4, to=3-5]
	\arrow[from=3-4, to=4-4]
	\arrow[from=4-1, to=4-2]
	\arrow[from=4-2, to=4-3]
	\arrow[from=4-3, to=4-4]
	\arrow[from=4-3, to=5-3]
	\arrow[from=4-4, to=4-5]
	\arrow[from=4-4, to=5-4]
\end{tikzcd}\]
\noindent  This implies that ${\rm Ext}^1_\mathbb{X}(S_{\infty, \overline{0}}, E_L\langle \vec{x} \rangle) \neq 0$ and finishes the proof.
\end{proof}

\subsection{Pseudo-triangulations and tilting bundles}
Analogous to the definition of a triangulation on $(\mathcal{S}, M)$, we define a \emph{pseudo-triangulation} on $(\mathcal{S}, M,\sigma)$ as a maximal collection of distinct, pairwise compatible skew-arcs. We denote by $\Lambda$ a pseudo-triangulation on $(\mathcal{S}, M,\sigma)$.  By abuse of notations, we denote
$\Lambda_{\mathtt{b}, \mathcal{X}} = \Lambda \cap \mathbf{C}_{\mathtt{b}, \mathcal{X}}$, 
$\Lambda_{\mathtt{b}}^\sigma = \Lambda \cap \mathbf{C}_{\mathtt{b}}^\sigma$, 
$\Lambda_{\mathtt{p}}^\sigma = \Lambda \cap \mathbf{C}_{\mathtt{p}}^\sigma$, 
$\Lambda_* = \Lambda \cap \mathbf{C}_*$, $\cdots$.

Now, assume that $\Lambda$ is a pseudo-triangulation  on $(\mathcal{S}, M,\sigma)$. We will outline some properties of $\Lambda$ and propose a classification framework for all possible $\Lambda$. Furthermore, we calculate the number of skew-arcs in a pseudo-triangulation and then establish a correspondence between  pseudo-triangulations on $(\mathcal{S}, M,\sigma)$ and tilting sheaves in ${\rm coh}\mbox{-}\mathbb{X}$.

\begin{lem}\label{lem:6.5}
    For a pseudo-triangulation $\Lambda$ on $(\mathcal{S}, M,\sigma)$, we have $|\Lambda_*| \leq 2$. Moreover,
    \begin{itemize}
        \item[(a)] If $\Lambda_* = \{(\epsilon_1, \epsilon_2)\}$, then $\Lambda\setminus\Lambda_* = \Lambda_{\mathtt{b}, \times_1}^{\epsilon_1} \cup \Lambda_{\mathtt{b}, \times_2}^{\epsilon_2} \cup \Lambda_\mathtt{p}^\sigma$.
        \item[(b)] If $\Lambda_* = \{(\epsilon_1, \epsilon_2), (\epsilon_1^\prime, \epsilon_2^\prime)\}$, then $\epsilon_1 = \epsilon_1^\prime$ or $\epsilon_2 = \epsilon_2^\prime$, and
        \begin{itemize}
            \item  $\Lambda\setminus\Lambda_* = \Lambda_{\mathtt{b}, \times_1}^{\epsilon_1} \cup \Lambda_{\mathtt{p}}^\sigma$, when $\epsilon_1 = \epsilon_1^\prime$;
            \item  $\Lambda\setminus\Lambda_*= \Lambda_{\mathtt{b}, \times_2}^{\epsilon_2} \cup \Lambda_\mathtt{p}^\sigma$,  when $\epsilon_2 = \epsilon_2^\prime$.
        \end{itemize}
    \end{itemize}
\end{lem}

\begin{proof}
    The first statement is clear since for any two pairwise non-isomorphic simple objects $S_1, S_2$ in the same $\tau$-period $2$ tube, we have ${\rm Ext}^{1}_\mathbb{X}(S_1,S_2)\neq 0$. As for the second statement, the skew-arcs in $\mathbf{C}_{\mathtt{p}}^\sigma$ correspond to indecomposable objects in the $\tau$-period $n$ tubes, and different tubes are orthogonal in $\operatorname{coh}\mbox{-}\mathbb{X}$.
\end{proof}

Based on the previous information, we can classify the cases for $\Lambda$ into three types:

\begin{itemize}
    \item  $\Lambda = \Lambda_{\mathtt{b},\mathcal{X}} \cup \Lambda_{\mathtt{b}}^\sigma \cup \Lambda_{\mathtt{p}}^\sigma$; 
    \item  $\Lambda = \Lambda_{\mathtt{b},\mathcal{X}} \cup \Lambda_{\mathtt{p}}^\sigma \cup \Lambda_*$, with $|\Lambda_*| = 1$; 
    \item  $\Lambda = \Lambda_{\mathtt{b},\mathcal{X}} \cup \Lambda_{\mathtt{p}}^\sigma \cup \Lambda_*$, with $|\Lambda_*| = 2$.
\end{itemize}

In order to  refine above classification, we firsty introduce some notations. 
\begin{conven}\label{conven2}
If $\Lambda$ satisfies $\Lambda_{\mathtt{b}}^\sigma \neq \emptyset$, then for each $\widehat{\gamma} \in \Lambda_{\mathtt{b}}^\sigma$, $\widehat{\gamma}$ and the boundaries of $(\mathcal{S}, M)$ form two quadrilaterals, one containing $\times_1$ and the other containing $\times_2$. A skew-arc $\widehat{\gamma}\in \Lambda_{\mathtt{b}}^\sigma$ is said to be \emph{closest} to $\times_i$ if no other skew-arc in $\Lambda_{\mathtt{b}}^\sigma$ lies within the quadrilateral containing $\times_i$. For $i = 1, 2$, let $\widehat{\gamma}_i = \{\gamma_i, \sigma(\gamma_i)\} \in \Lambda_{\mathtt{b}}^\sigma$ denote the skew-curve closest to $\times_i$, where $\widehat{\gamma}_1$ and $\widehat{\gamma}_2$ may coincide.
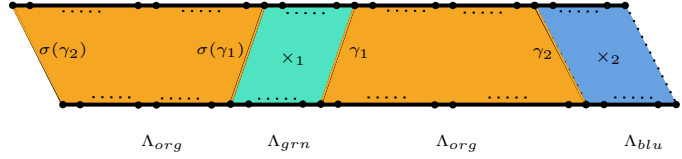
\begin{figure}[H]

\tikzset{every picture/.style={line width=0.75pt}}  

\begin{tikzpicture}[x=0.75pt,y=0.75pt,yscale=-1,xscale=1]

\draw  [fill={rgb, 255:red, 245; green, 166; blue, 35 }  ,fill opacity=1 ] (215.18,144.96) -- (198.31,194.54) -- (114.61,194.84) -- (89.57,144.89) -- cycle ;
 
\draw [color={rgb, 255:red, 245; green, 166; blue, 35 }  ,draw opacity=1 ]   (89.57,144.89) -- (115,194.77) ;
 
\draw  [fill={rgb, 255:red, 104; green, 161; blue, 226 }  ,fill opacity=1 ][dash pattern={on 0.84pt off 2.51pt}] (350.37,144.93) -- (395.76,144.93) -- (421.19,194.8) -- (376.21,194.59) -- cycle ;
 
\draw  [fill={rgb, 255:red, 245; green, 166; blue, 35 }  ,fill opacity=1 ] (349.98,144.93) -- (375.8,194.8) -- (243.54,194.59) -- (260.37,144.93) -- cycle ;
 
\draw    (152.37,144.93) -- (161.37,144.93) ;
\draw [shift={(161.37,144.93)}, rotate = 0] [color={rgb, 255:red, 0; green, 0; blue, 0 }  ][fill={rgb, 255:red, 0; green, 0; blue, 0 }  ][line width=0.75]      (0, 0) circle [x radius= 1.34, y radius= 1.34]   ;
\draw [shift={(152.37,144.93)}, rotate = 0] [color={rgb, 255:red, 0; green, 0; blue, 0 }  ][fill={rgb, 255:red, 0; green, 0; blue, 0 }  ][line width=0.75]      (0, 0) circle [x radius= 1.34, y radius= 1.34]   ;
 
\draw    (300.76,144.93) -- (309.37,144.93) ;
\draw [shift={(309.37,144.93)}, rotate = 0] [color={rgb, 255:red, 0; green, 0; blue, 0 }  ][fill={rgb, 255:red, 0; green, 0; blue, 0 }  ][line width=0.75]      (0, 0) circle [x radius= 1.34, y radius= 1.34]   ;
\draw [shift={(300.76,144.93)}, rotate = 0] [color={rgb, 255:red, 0; green, 0; blue, 0 }  ][fill={rgb, 255:red, 0; green, 0; blue, 0 }  ][line width=0.75]      (0, 0) circle [x radius= 1.34, y radius= 1.34]   ;
 
\draw    (341.37,144.93) -- (349.98,144.93) ;
\draw [shift={(349.98,144.93)}, rotate = 0] [color={rgb, 255:red, 0; green, 0; blue, 0 }  ][fill={rgb, 255:red, 0; green, 0; blue, 0 }  ][line width=0.75]      (0, 0) circle [x radius= 1.34, y radius= 1.34]   ;
\draw [shift={(341.37,144.93)}, rotate = 0] [color={rgb, 255:red, 0; green, 0; blue, 0 }  ][fill={rgb, 255:red, 0; green, 0; blue, 0 }  ][line width=0.75]      (0, 0) circle [x radius= 1.34, y radius= 1.34]   ;
 
\draw    (260.37,144.93) -- (269.37,144.93) ;
\draw [shift={(269.37,144.93)}, rotate = 0] [color={rgb, 255:red, 0; green, 0; blue, 0 }  ][fill={rgb, 255:red, 0; green, 0; blue, 0 }  ][line width=0.75]      (0, 0) circle [x radius= 1.34, y radius= 1.34]   ;
\draw [shift={(260.37,144.93)}, rotate = 0] [color={rgb, 255:red, 0; green, 0; blue, 0 }  ][fill={rgb, 255:red, 0; green, 0; blue, 0 }  ][line width=0.75]      (0, 0) circle [x radius= 1.34, y radius= 1.34]   ;
 
\draw    (152.5,195) -- (161.5,195) ;
\draw [shift={(161.5,195)}, rotate = 0] [color={rgb, 255:red, 0; green, 0; blue, 0 }  ][fill={rgb, 255:red, 0; green, 0; blue, 0 }  ][line width=0.75]      (0, 0) circle [x radius= 1.34, y radius= 1.34]   ;
\draw [shift={(152.5,195)}, rotate = 0] [color={rgb, 255:red, 0; green, 0; blue, 0 }  ][fill={rgb, 255:red, 0; green, 0; blue, 0 }  ][line width=0.75]      (0, 0) circle [x radius= 1.34, y radius= 1.34]   ;
 
\draw    (367.19,194.8) -- (375.8,194.8) ;
\draw [shift={(375.8,194.8)}, rotate = 0] [color={rgb, 255:red, 0; green, 0; blue, 0 }  ][fill={rgb, 255:red, 0; green, 0; blue, 0 }  ][line width=0.75]      (0, 0) circle [x radius= 1.34, y radius= 1.34]   ;
\draw [shift={(367.19,194.8)}, rotate = 0] [color={rgb, 255:red, 0; green, 0; blue, 0 }  ][fill={rgb, 255:red, 0; green, 0; blue, 0 }  ][line width=0.75]      (0, 0) circle [x radius= 1.34, y radius= 1.34]   ;
 
\draw  [dash pattern={on 0.84pt off 2.51pt}]  (266.79,191.2) -- (286.12,191.2) ;
 
\draw    (215.37,144.93) -- (206.37,144.93) ;
\draw [shift={(206.37,144.93)}, rotate = 180] [color={rgb, 255:red, 0; green, 0; blue, 0 }  ][fill={rgb, 255:red, 0; green, 0; blue, 0 }  ][line width=0.75]      (0, 0) circle [x radius= 1.34, y radius= 1.34]   ;
\draw [shift={(215.37,144.93)}, rotate = 180] [color={rgb, 255:red, 0; green, 0; blue, 0 }  ][fill={rgb, 255:red, 0; green, 0; blue, 0 }  ][line width=0.75]      (0, 0) circle [x radius= 1.34, y radius= 1.34]   ;
 
\draw    (198.5,194.5) -- (189,195) ;
\draw [shift={(189,195)}, rotate = 176.99] [color={rgb, 255:red, 0; green, 0; blue, 0 }  ][fill={rgb, 255:red, 0; green, 0; blue, 0 }  ][line width=0.75]      (0, 0) circle [x radius= 1.34, y radius= 1.34]   ;
\draw [shift={(198.5,194.5)}, rotate = 176.99] [color={rgb, 255:red, 0; green, 0; blue, 0 }  ][fill={rgb, 255:red, 0; green, 0; blue, 0 }  ][line width=0.75]      (0, 0) circle [x radius= 1.34, y radius= 1.34]   ;
 
\draw  [dash pattern={on 0.84pt off 2.51pt}]  (148.48,191.18) -- (129.15,191.02) ;
 
\draw    (98.37,144.93) -- (89.76,144.86) ;
\draw [shift={(89.76,144.86)}, rotate = 180.47] [color={rgb, 255:red, 0; green, 0; blue, 0 }  ][fill={rgb, 255:red, 0; green, 0; blue, 0 }  ][line width=0.75]      (0, 0) circle [x radius= 1.34, y radius= 1.34]   ;
\draw [shift={(98.37,144.93)}, rotate = 180.47] [color={rgb, 255:red, 0; green, 0; blue, 0 }  ][fill={rgb, 255:red, 0; green, 0; blue, 0 }  ][line width=0.75]      (0, 0) circle [x radius= 1.34, y radius= 1.34]   ;
 
\draw  [dash pattern={on 0.84pt off 2.51pt}]  (133.52,148.01) -- (114.18,147.85) ;
 
\draw  [dash pattern={on 0.84pt off 2.51pt}]  (295.52,147.67) -- (276.18,147.52) ;
 
\draw  [fill={rgb, 255:red, 80; green, 227; blue, 194 }  ,fill opacity=1 ] (215.41,145.02) -- (260.41,145.02) -- (243.54,194.59) -- (198.54,194.59) -- cycle ;
 
\draw  [dash pattern={on 0.84pt off 2.51pt}]  (193.22,147.81) -- (173.88,147.65) ;
 
\draw  [dash pattern={on 0.84pt off 2.51pt}]  (182.72,191.81) -- (163.38,191.65) ;
 
\draw  [dash pattern={on 0.84pt off 2.51pt}]  (314.76,148.53) -- (334.09,148.53) ;
 
\draw  [dash pattern={on 0.84pt off 2.51pt}]  (327.16,191.23) -- (346.49,191.23) ;
 
\draw [color={rgb, 255:red, 245; green, 166; blue, 35 }  ,draw opacity=1 ]   (350.37,144.93) -- (375.8,194.8) ;
 
\draw [color={rgb, 255:red, 245; green, 166; blue, 35 }  ,draw opacity=1 ]   (215.37,144.93) -- (198.5,194.5) ;
 
\draw [color={rgb, 255:red, 245; green, 166; blue, 35 }  ,draw opacity=1 ]   (260.37,144.93) -- (243.5,194.5) ;
 
\draw    (300.37,194.93) -- (309.37,194.93) ;
\draw [shift={(309.37,194.93)}, rotate = 0] [color={rgb, 255:red, 0; green, 0; blue, 0 }  ][fill={rgb, 255:red, 0; green, 0; blue, 0 }  ][line width=0.75]      (0, 0) circle [x radius= 1.34, y radius= 1.34]   ;
\draw [shift={(300.37,194.93)}, rotate = 0] [color={rgb, 255:red, 0; green, 0; blue, 0 }  ][fill={rgb, 255:red, 0; green, 0; blue, 0 }  ][line width=0.75]      (0, 0) circle [x radius= 1.34, y radius= 1.34]   ;
 
\draw    (350.37,144.93) -- (359.37,144.93) ;
\draw [shift={(359.37,144.93)}, rotate = 0] [color={rgb, 255:red, 0; green, 0; blue, 0 }  ][fill={rgb, 255:red, 0; green, 0; blue, 0 }  ][line width=0.75]      (0, 0) circle [x radius= 1.34, y radius= 1.34]   ;
\draw [shift={(350.37,144.93)}, rotate = 0] [color={rgb, 255:red, 0; green, 0; blue, 0 }  ][fill={rgb, 255:red, 0; green, 0; blue, 0 }  ][line width=0.75]      (0, 0) circle [x radius= 1.34, y radius= 1.34]   ;
 
\draw    (198.5,194.5) -- (207.5,194.5) ;
\draw [shift={(207.5,194.5)}, rotate = 0] [color={rgb, 255:red, 0; green, 0; blue, 0 }  ][fill={rgb, 255:red, 0; green, 0; blue, 0 }  ][line width=0.75]      (0, 0) circle [x radius= 1.34, y radius= 1.34]   ;
\draw [shift={(198.5,194.5)}, rotate = 0] [color={rgb, 255:red, 0; green, 0; blue, 0 }  ][fill={rgb, 255:red, 0; green, 0; blue, 0 }  ][line width=0.75]      (0, 0) circle [x radius= 1.34, y radius= 1.34]   ;
 
\draw    (234.5,194.5) -- (243.5,194.5) ;
\draw [shift={(243.5,194.5)}, rotate = 0] [color={rgb, 255:red, 0; green, 0; blue, 0 }  ][fill={rgb, 255:red, 0; green, 0; blue, 0 }  ][line width=0.75]      (0, 0) circle [x radius= 1.34, y radius= 1.34]   ;
\draw [shift={(234.5,194.5)}, rotate = 0] [color={rgb, 255:red, 0; green, 0; blue, 0 }  ][fill={rgb, 255:red, 0; green, 0; blue, 0 }  ][line width=0.75]      (0, 0) circle [x radius= 1.34, y radius= 1.34]   ;
 
\draw    (251.37,144.93) -- (260.37,144.93) ;
\draw [shift={(260.37,144.93)}, rotate = 0] [color={rgb, 255:red, 0; green, 0; blue, 0 }  ][fill={rgb, 255:red, 0; green, 0; blue, 0 }  ][line width=0.75]      (0, 0) circle [x radius= 1.34, y radius= 1.34]   ;
\draw [shift={(251.37,144.93)}, rotate = 0] [color={rgb, 255:red, 0; green, 0; blue, 0 }  ][fill={rgb, 255:red, 0; green, 0; blue, 0 }  ][line width=0.75]      (0, 0) circle [x radius= 1.34, y radius= 1.34]   ;
 
\draw    (215.37,144.93) -- (224.37,144.93) ;
\draw [shift={(224.37,144.93)}, rotate = 0] [color={rgb, 255:red, 0; green, 0; blue, 0 }  ][fill={rgb, 255:red, 0; green, 0; blue, 0 }  ][line width=0.75]      (0, 0) circle [x radius= 1.34, y radius= 1.34]   ;
\draw [shift={(215.37,144.93)}, rotate = 0] [color={rgb, 255:red, 0; green, 0; blue, 0 }  ][fill={rgb, 255:red, 0; green, 0; blue, 0 }  ][line width=0.75]      (0, 0) circle [x radius= 1.34, y radius= 1.34]   ;
 
\draw    (411.8,194.8) -- (420.8,194.8) ;
\draw [shift={(420.8,194.8)}, rotate = 0] [color={rgb, 255:red, 0; green, 0; blue, 0 }  ][fill={rgb, 255:red, 0; green, 0; blue, 0 }  ][line width=0.75]      (0, 0) circle [x radius= 1.34, y radius= 1.34]   ;
\draw [shift={(411.8,194.8)}, rotate = 0] [color={rgb, 255:red, 0; green, 0; blue, 0 }  ][fill={rgb, 255:red, 0; green, 0; blue, 0 }  ][line width=0.75]      (0, 0) circle [x radius= 1.34, y radius= 1.34]   ;
 
\draw    (375.8,194.8) -- (384.8,194.8) ;
\draw [shift={(384.8,194.8)}, rotate = 0] [color={rgb, 255:red, 0; green, 0; blue, 0 }  ][fill={rgb, 255:red, 0; green, 0; blue, 0 }  ][line width=0.75]      (0, 0) circle [x radius= 1.34, y radius= 1.34]   ;
\draw [shift={(375.8,194.8)}, rotate = 0] [color={rgb, 255:red, 0; green, 0; blue, 0 }  ][fill={rgb, 255:red, 0; green, 0; blue, 0 }  ][line width=0.75]      (0, 0) circle [x radius= 1.34, y radius= 1.34]   ;
 
\draw    (386.37,144.93) -- (395.37,144.93) ;
\draw [shift={(395.37,144.93)}, rotate = 0] [color={rgb, 255:red, 0; green, 0; blue, 0 }  ][fill={rgb, 255:red, 0; green, 0; blue, 0 }  ][line width=0.75]      (0, 0) circle [x radius= 1.34, y radius= 1.34]   ;
\draw [shift={(386.37,144.93)}, rotate = 0] [color={rgb, 255:red, 0; green, 0; blue, 0 }  ][fill={rgb, 255:red, 0; green, 0; blue, 0 }  ][line width=0.75]      (0, 0) circle [x radius= 1.34, y radius= 1.34]   ;
 
\draw    (123.41,194.87) -- (114.8,194.8) ;
\draw [shift={(114.8,194.8)}, rotate = 180.47] [color={rgb, 255:red, 0; green, 0; blue, 0 }  ][fill={rgb, 255:red, 0; green, 0; blue, 0 }  ][line width=0.75]      (0, 0) circle [x radius= 1.34, y radius= 1.34]   ;
\draw [shift={(123.41,194.87)}, rotate = 180.47] [color={rgb, 255:red, 0; green, 0; blue, 0 }  ][fill={rgb, 255:red, 0; green, 0; blue, 0 }  ][line width=0.75]      (0, 0) circle [x radius= 1.34, y radius= 1.34]   ;
 
\draw    (252.11,194.57) -- (243.5,194.5) ;
\draw [shift={(243.5,194.5)}, rotate = 180.47] [color={rgb, 255:red, 0; green, 0; blue, 0 }  ][fill={rgb, 255:red, 0; green, 0; blue, 0 }  ][line width=0.75]      (0, 0) circle [x radius= 1.34, y radius= 1.34]   ;
\draw [shift={(252.11,194.57)}, rotate = 180.47] [color={rgb, 255:red, 0; green, 0; blue, 0 }  ][fill={rgb, 255:red, 0; green, 0; blue, 0 }  ][line width=0.75]      (0, 0) circle [x radius= 1.34, y radius= 1.34]   ;
 
\draw [line width=1.5]    (114.8,194.8) -- (420.8,194.8) ;
 
\draw  [dash pattern={on 0.84pt off 2.51pt}]  (245.35,148.84) -- (226.02,148.68) ;
 
\draw  [dash pattern={on 0.84pt off 2.51pt}]  (212.16,191.83) -- (231.49,191.83) ;
 
\draw  [dash pattern={on 0.84pt off 2.51pt}]  (381.68,148.01) -- (362.35,147.85) ;
 
\draw  [dash pattern={on 0.84pt off 2.51pt}]  (408.85,191.34) -- (389.52,191.18) ;
 
\draw [line width=1.5]    (89.76,144.93) -- (395.76,144.93) ;

\draw (221.66,166.73) node [anchor=north west][inner sep=0.75pt]  [font=\tiny]  {$\times _{1}$};
 
\draw (256.17,164.57) node [anchor=north west][inner sep=0.75pt]  [font=\tiny]  {$\gamma _{1}$};
 
\draw (347.83,165.07) node [anchor=north west][inner sep=0.75pt]  [font=\tiny]  {$\gamma _{2}$};
 
\draw (152.33,208.73) node [anchor=north west][inner sep=0.75pt]  [font=\tiny]  {$\Lambda _{{org}}$};
 
\draw (378.99,166.23) node [anchor=north west][inner sep=0.75pt]  [font=\tiny]  {$\times _{2}$};
 
\draw (215.16,208.46) node [anchor=north west][inner sep=0.75pt]  [font=\tiny]  {$\Lambda _{grn}$};
 
\draw (393.16,208.46) node [anchor=north west][inner sep=0.75pt]  [font=\tiny]  {$\Lambda _{blu}$};
 
\draw (299.33,208.73) node [anchor=north west][inner sep=0.75pt]  [font=\tiny]  {$\Lambda _{{org}}$};
 
\draw (180,161.73) node [anchor=north west][inner sep=0.75pt]  [font=\tiny]  {$\sigma ( \gamma _{1})$};
 
\draw (101.33,162.07) node [anchor=north west][inner sep=0.75pt]  [font=\tiny]  {$\sigma ( \gamma _{2})$};

\end{tikzpicture}\end{figure}
\noindent Then $\widehat{\gamma}_1$ and $\widehat{\gamma}_2$ divide $\Lambda$ into three disjoint subsets:
\[
\Lambda = \Lambda_{org} \cup \Lambda_{grn} \cup \Lambda_{blu}
\]
as shown in above figure, where $\Lambda_{org}$ is the set of skew-arcs in $\Lambda$ that fall within the orange region containing $\widehat{\gamma}_1$ and $\widehat{\gamma}_2$; $\Lambda_{grn}$ consists of those in the green region, and $\Lambda_{blu}$ contains those in the blue region.
\end{conven}

For $i = 1, 2$ and $\epsilon \in \{+, -\}$, let 
$\mathbf{C}_{\times_i}^{\epsilon} = \mathbf{C}_{\mathtt{b}, \times_i}^{\epsilon} \cup \{(\epsilon, +), (\epsilon, -)\}$ and $\mathbf{C}_{\times_i}=\mathbf{C}_{\times_i}^{+}\cup\mathbf{C}_{\times_i}^{-}$. According to Proposition \ref{prop:6.4} and Lemma \ref{lem:6.5}, it can be easily verified that
 $\Lambda_{\times_i}^{+} \neq \emptyset$ and $\Lambda_{\times_i}^{-} \neq \emptyset$
if and only if one of the following holds:
\begin{itemize}
    \item[$(A0)$] $\Lambda_{\times_i}=\{\gamma^+,\gamma^-\}$ for some $\gamma$ $\sigma$-fixed on $\times_i$;
    \item[$(A1)$] If $i = 1$, and $(-, +), (+, +)\in \Lambda$ or  $(-, -),(+, -) \in \Lambda$;
    \item[$(A2)$] If $i = 2$, and $(-, +),(-, -)\in \Lambda$ or $(+, +),(+, -) \in \Lambda$.
\end{itemize}

Now, we can define a map $\zeta$ to distinguish the skew-arcs in $\Lambda_{\times_i}$ as follows: 
\[
\zeta: \{\text{All pseudo-triangulations in } (\mathcal{S}, M)\} \to \{\pm, +, -\} \times \{\pm, +, -\},
\]
such that
\[
\Lambda \mapsto (\zeta(\Lambda)_1, \zeta(\Lambda)_2),
\]
with
\[
	\zeta(\Lambda)_i =
\begin{cases}
\epsilon & \text{if } \Lambda_{\times_i} \subset \mathbf{C}_{\times_i}^{\epsilon}, \quad \epsilon \in \{+, -\}; \\
\pm & \text{if } \Lambda_{\times_i}^{+} \neq \emptyset \text{ and } \Lambda_{\times_i}^{-} \neq \emptyset.
\end{cases}
\]
To show $\zeta$ is well-defined, we need the following lemma:  
\begin{lem}\label{lem:01}\label{well-defined}
  For each $\times \in \mathcal{X}$,  if there exists $\epsilon\in \{+,-\}$ such that $\Lambda_{\times}^{\epsilon} = \emptyset$, then  $|\Lambda_{\times}^{\neg\epsilon}| \geq 2$. 
\end{lem}

\begin{proof}
We only consider the case $\times = \times_1$, as the argument for $\times = \times_2$ follows similarly. 

Assume that  $\Lambda_{\times_1}^{\epsilon}=\emptyset$ for some $\epsilon\in \{+,-\}$.  There are  four cases to consider:

(1) $\Lambda_{\mathtt{b}}^\sigma \neq \emptyset$. In this case, $\Lambda_*=\emptyset$. We claim that either $\Lambda_{\times_1}^+ \neq \emptyset$ or $\Lambda_{\times_1}^- \neq \emptyset$.  To show it, we adopt the notations from Convention \ref{conven2}. Consider the skew-arcs  $\widehat{\alpha}_1, \widehat{\alpha}_2, \widehat{\alpha}_3$, and $\widehat{\alpha}_4$, in $\mathbf{C}_{\mathtt{b},\times_1}$, as illustrated below.
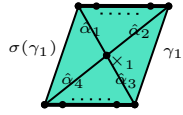
\begin{figure}[H]
    \centering

\tikzset{every picture/.style={line width=0.75pt}}  

\begin{tikzpicture}[x=0.75pt,y=0.75pt,yscale=-1,xscale=1]

\draw  [fill={rgb, 255:red, 80; green, 227; blue, 194 }  ,fill opacity=1 ] (235.41,165.02) -- (280.41,165.02) -- (263.54,214.59) -- (218.54,214.59) -- cycle ;
 
\draw    (218.5,214.5) -- (227.5,214.5) ;
\draw [shift={(227.5,214.5)}, rotate = 0] [color={rgb, 255:red, 0; green, 0; blue, 0 }  ][fill={rgb, 255:red, 0; green, 0; blue, 0 }  ][line width=0.75]      (0, 0) circle [x radius= 1.34, y radius= 1.34]   ;
\draw [shift={(218.5,214.5)}, rotate = 0] [color={rgb, 255:red, 0; green, 0; blue, 0 }  ][fill={rgb, 255:red, 0; green, 0; blue, 0 }  ][line width=0.75]      (0, 0) circle [x radius= 1.34, y radius= 1.34]   ;
 
\draw    (254.5,214.5) -- (263.5,214.5) ;
\draw [shift={(263.5,214.5)}, rotate = 0] [color={rgb, 255:red, 0; green, 0; blue, 0 }  ][fill={rgb, 255:red, 0; green, 0; blue, 0 }  ][line width=0.75]      (0, 0) circle [x radius= 1.34, y radius= 1.34]   ;
\draw [shift={(254.5,214.5)}, rotate = 0] [color={rgb, 255:red, 0; green, 0; blue, 0 }  ][fill={rgb, 255:red, 0; green, 0; blue, 0 }  ][line width=0.75]      (0, 0) circle [x radius= 1.34, y radius= 1.34]   ;
 
\draw    (271.37,164.93) -- (280.37,164.93) ;
\draw [shift={(280.37,164.93)}, rotate = 0] [color={rgb, 255:red, 0; green, 0; blue, 0 }  ][fill={rgb, 255:red, 0; green, 0; blue, 0 }  ][line width=0.75]      (0, 0) circle [x radius= 1.34, y radius= 1.34]   ;
\draw [shift={(271.37,164.93)}, rotate = 0] [color={rgb, 255:red, 0; green, 0; blue, 0 }  ][fill={rgb, 255:red, 0; green, 0; blue, 0 }  ][line width=0.75]      (0, 0) circle [x radius= 1.34, y radius= 1.34]   ;
 
\draw    (235.37,164.93) -- (244.37,164.93) ;
\draw [shift={(244.37,164.93)}, rotate = 0] [color={rgb, 255:red, 0; green, 0; blue, 0 }  ][fill={rgb, 255:red, 0; green, 0; blue, 0 }  ][line width=0.75]      (0, 0) circle [x radius= 1.34, y radius= 1.34]   ;
\draw [shift={(235.37,164.93)}, rotate = 0] [color={rgb, 255:red, 0; green, 0; blue, 0 }  ][fill={rgb, 255:red, 0; green, 0; blue, 0 }  ][line width=0.75]      (0, 0) circle [x radius= 1.34, y radius= 1.34]   ;
 
\draw  [dash pattern={on 0.84pt off 2.51pt}]  (265.35,168.84) -- (246.02,168.68) ;
 
\draw  [dash pattern={on 0.84pt off 2.51pt}]  (232.16,211.83) -- (251.49,211.83) ;
 
\draw    (218.54,214.59) -- (280.41,165.02) ;
\draw [shift={(280.41,165.02)}, rotate = 321.3] [color={rgb, 255:red, 0; green, 0; blue, 0 }  ][fill={rgb, 255:red, 0; green, 0; blue, 0 }  ][line width=0.75]      (0, 0) circle [x radius= 1.34, y radius= 1.34]   ;
\draw [shift={(249.47,189.81)}, rotate = 321.3] [color={rgb, 255:red, 0; green, 0; blue, 0 }  ][fill={rgb, 255:red, 0; green, 0; blue, 0 }  ][line width=0.75]      (0, 0) circle [x radius= 1.34, y radius= 1.34]   ;
\draw [shift={(218.54,214.59)}, rotate = 321.3] [color={rgb, 255:red, 0; green, 0; blue, 0 }  ][fill={rgb, 255:red, 0; green, 0; blue, 0 }  ][line width=0.75]      (0, 0) circle [x radius= 1.34, y radius= 1.34]   ;
 
\draw    (235.37,164.93) -- (263.5,214.5) ;
\draw [shift={(263.5,214.5)}, rotate = 60.43] [color={rgb, 255:red, 0; green, 0; blue, 0 }  ][fill={rgb, 255:red, 0; green, 0; blue, 0 }  ][line width=0.75]      (0, 0) circle [x radius= 1.34, y radius= 1.34]   ;
\draw [shift={(249.44,189.71)}, rotate = 60.43] [color={rgb, 255:red, 0; green, 0; blue, 0 }  ][fill={rgb, 255:red, 0; green, 0; blue, 0 }  ][line width=0.75]      (0, 0) circle [x radius= 1.34, y radius= 1.34]   ;
\draw [shift={(235.37,164.93)}, rotate = 60.43] [color={rgb, 255:red, 0; green, 0; blue, 0 }  ][fill={rgb, 255:red, 0; green, 0; blue, 0 }  ][line width=0.75]      (0, 0) circle [x radius= 1.34, y radius= 1.34]   ;
 
\draw [line width=1.5]    (235.34,164.83) -- (280.37,164.93) ;
 
\draw [line width=1.5]    (218.54,214.59) -- (263.57,214.68) ;

\draw (234.33,171.4) node [anchor=north west][inner sep=0.75pt]  [font=\tiny]  {$\hat{\alpha }_{1}$};
 
\draw (258.67,172.4) node [anchor=north west][inner sep=0.75pt]  [font=\tiny]  {$\hat{\alpha }_{2}$};
 
\draw (225.33,197.23) node [anchor=north west][inner sep=0.75pt]  [font=\tiny]  {$\hat{\alpha }_{4}$};
 
\draw (252,197.23) node [anchor=north west][inner sep=0.75pt]  [font=\tiny]  {$\hat{\alpha }_{3}$};
 
\draw (249.44,188.11) node [anchor=north west][inner sep=0.75pt]  [font=\tiny]  {$\times _{1}$};
 
\draw (276.17,184.57) node [anchor=north west][inner sep=0.75pt]  [font=\tiny]  {$\gamma _{1}$};
 
\draw (199,180.73) node [anchor=north west][inner sep=0.75pt]  [font=\tiny]  {$\sigma ( \gamma _{1})$};

\end{tikzpicture}
\end{figure}
\noindent  If $\Lambda_{\times_1}^+ =\Lambda_{\times_1}^- \neq \emptyset$, then each $\widehat{\alpha}_i$ is compatible with all skew-arcs in $\Lambda$. It follows that $\Lambda$ is not maximal, a contradiction. Furthermore, it is obvious that if  $\Lambda_{\times_1}^- = \emptyset$, then $\Lambda_{\times_1}^+ \neq \emptyset$ and $\widehat{\alpha}_1, \widehat{\alpha}_2 \in \Lambda_{\times_1}^+$. Conversely, if  $\Lambda_{\times_1}^+ = \emptyset$, then $\Lambda_{\times_1}^- \neq \emptyset$ and $\widehat{\alpha}_3, \widehat{\alpha}_4 \in \Lambda_{\times_1}^-$. Thus, the statement holds.

(2) $\Lambda_* \neq \emptyset$ and $|\Lambda_*|=2$. Since $\Lambda_* \subset \Lambda_{\times_1}^{\neg\epsilon}$, the statement holds immediately.

(3)
$\Lambda_* \neq \emptyset$ and  $|\Lambda_*| = 1$. In this case, $\Lambda_* = \{(\neg\epsilon, +)\}$ or $\{(\neg\epsilon, -)\}$. We claim that $\Lambda_{\mathtt{b},\times_1}^{\neg\epsilon} \neq \emptyset$. Assume to the contrary that $\Lambda_{\mathtt{b},\times_1}^{\neg\epsilon}= \emptyset$. Then $\Lambda_{\mathtt{p}}^\sigma$ does not consist of mutually compatible skew-arcs. Therefore, it follows that $|\Lambda_{\times_1}^{\neg\epsilon}| = |\Lambda_{\mathtt{b},\times_1}^{\neg\epsilon}| + 1 \geq 2$. 

(4)
$\Lambda_{\mathtt{b}}^\sigma=\Lambda_* = \emptyset$. We claim that both $\Lambda_{\mathtt{b},\times_2}^{+}$ and $\Lambda_{\mathtt{b},\times_2}^{-}$ must be non-empty. Otherwise, if there exists $\epsilon^\prime\in \{+,-\}$ such that $\Lambda_{\times_2}^{\epsilon^\prime}=\emptyset$ , then $(\epsilon,\epsilon^\prime)\in \Lambda_*$, a contradiction to $\Lambda_* = \emptyset$. Then there always exist two skew-arcs in $\mathbf{C}_{\times_1}^{\neg\epsilon}$ that are compatible with all skew-arcs in $\Lambda_{\times_1}^{\neg\epsilon} \cup \Lambda_{\mathtt{p}}^\sigma$. These skew-arcs are highlighted in orange in the figure below:
 
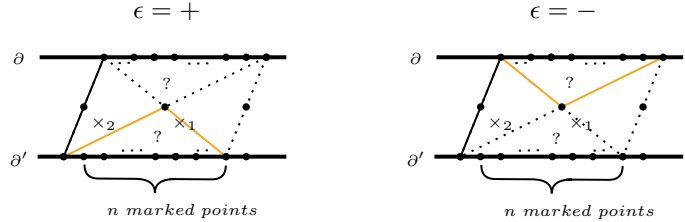
\begin{figure}[H]
     \centering

\tikzset{every picture/.style={line width=0.75pt}}  

\begin{tikzpicture}[x=0.75pt,y=0.75pt,yscale=-1,xscale=1]

\draw [color={rgb, 255:red, 245; green, 166; blue, 35 }  ,draw opacity=1 ]   (178.42,215.68) -- (208.87,240.68) ;
 
\draw [color={rgb, 255:red, 245; green, 166; blue, 35 }  ,draw opacity=1 ]   (178.42,215.68) -- (127.87,240.68) ;
 
\draw    (147.97,190.68) -- (127.87,240.68) ;
 
\draw [line width=1.5]    (114.87,240.68) -- (238.97,240.68) ;
 
\draw  [dash pattern={on 0.84pt off 2.51pt}]  (167,237.44) -- (154.98,237.72) ;
 
\draw  [dash pattern={on 0.84pt off 2.51pt}]  (194.26,237.64) -- (201,237.44) ;
 
\draw [line width=1.5]    (115.87,190.68) -- (238.97,190.68) ;
 
\draw  [dash pattern={on 0.84pt off 2.51pt}]  (155.26,193.64) -- (162,193.44) ;
 
\draw    (137.87,240.68) -- (147.97,240.68) ;
\draw [shift={(147.97,240.68)}, rotate = 0] [color={rgb, 255:red, 0; green, 0; blue, 0 }  ][fill={rgb, 255:red, 0; green, 0; blue, 0 }  ][line width=0.75]      (0, 0) circle [x radius= 1.34, y radius= 1.34]   ;
\draw [shift={(137.87,240.68)}, rotate = 0] [color={rgb, 255:red, 0; green, 0; blue, 0 }  ][fill={rgb, 255:red, 0; green, 0; blue, 0 }  ][line width=0.75]      (0, 0) circle [x radius= 1.34, y radius= 1.34]   ;
 
\draw    (208.87,190.68) -- (218.97,190.68) ;
\draw [shift={(218.97,190.68)}, rotate = 0] [color={rgb, 255:red, 0; green, 0; blue, 0 }  ][fill={rgb, 255:red, 0; green, 0; blue, 0 }  ][line width=0.75]      (0, 0) circle [x radius= 1.34, y radius= 1.34]   ;
\draw [shift={(208.87,190.68)}, rotate = 0] [color={rgb, 255:red, 0; green, 0; blue, 0 }  ][fill={rgb, 255:red, 0; green, 0; blue, 0 }  ][line width=0.75]      (0, 0) circle [x radius= 1.34, y radius= 1.34]   ;
 
\draw    (208.87,240.68) -- (218.97,240.68) ;
\draw [shift={(218.97,240.68)}, rotate = 0] [color={rgb, 255:red, 0; green, 0; blue, 0 }  ][fill={rgb, 255:red, 0; green, 0; blue, 0 }  ][line width=0.75]      (0, 0) circle [x radius= 1.34, y radius= 1.34]   ;
\draw [shift={(208.87,240.68)}, rotate = 0] [color={rgb, 255:red, 0; green, 0; blue, 0 }  ][fill={rgb, 255:red, 0; green, 0; blue, 0 }  ][line width=0.75]      (0, 0) circle [x radius= 1.34, y radius= 1.34]   ;
 
\draw    (162.97,190.68) -- (173.07,190.68) ;
\draw [shift={(173.07,190.68)}, rotate = 0] [color={rgb, 255:red, 0; green, 0; blue, 0 }  ][fill={rgb, 255:red, 0; green, 0; blue, 0 }  ][line width=0.75]      (0, 0) circle [x radius= 1.34, y radius= 1.34]   ;
\draw [shift={(162.97,190.68)}, rotate = 0] [color={rgb, 255:red, 0; green, 0; blue, 0 }  ][fill={rgb, 255:red, 0; green, 0; blue, 0 }  ][line width=0.75]      (0, 0) circle [x radius= 1.34, y radius= 1.34]   ;
 
\draw    (183.77,240.68) -- (193.87,240.68) ;
\draw [shift={(193.87,240.68)}, rotate = 0] [color={rgb, 255:red, 0; green, 0; blue, 0 }  ][fill={rgb, 255:red, 0; green, 0; blue, 0 }  ][line width=0.75]      (0, 0) circle [x radius= 1.34, y radius= 1.34]   ;
\draw [shift={(183.77,240.68)}, rotate = 0] [color={rgb, 255:red, 0; green, 0; blue, 0 }  ][fill={rgb, 255:red, 0; green, 0; blue, 0 }  ][line width=0.75]      (0, 0) circle [x radius= 1.34, y radius= 1.34]   ;
 
\draw  [dash pattern={on 0.84pt off 2.51pt}]  (147.97,190.68) -- (127.87,240.68) ;
\draw [shift={(127.87,240.68)}, rotate = 111.9] [color={rgb, 255:red, 0; green, 0; blue, 0 }  ][fill={rgb, 255:red, 0; green, 0; blue, 0 }  ][line width=0.75]      (0, 0) circle [x radius= 1.34, y radius= 1.34]   ;
\draw [shift={(137.92,215.68)}, rotate = 111.9] [color={rgb, 255:red, 0; green, 0; blue, 0 }  ][fill={rgb, 255:red, 0; green, 0; blue, 0 }  ][line width=0.75]      (0, 0) circle [x radius= 1.34, y radius= 1.34]   ;
\draw [shift={(147.97,190.68)}, rotate = 111.9] [color={rgb, 255:red, 0; green, 0; blue, 0 }  ][fill={rgb, 255:red, 0; green, 0; blue, 0 }  ][line width=0.75]      (0, 0) circle [x radius= 1.34, y radius= 1.34]   ;
 
\draw    (173.07,190.68) -- (183.17,190.68) ;
\draw [shift={(183.17,190.68)}, rotate = 0] [color={rgb, 255:red, 0; green, 0; blue, 0 }  ][fill={rgb, 255:red, 0; green, 0; blue, 0 }  ][line width=0.75]      (0, 0) circle [x radius= 1.34, y radius= 1.34]   ;
\draw [shift={(173.07,190.68)}, rotate = 0] [color={rgb, 255:red, 0; green, 0; blue, 0 }  ][fill={rgb, 255:red, 0; green, 0; blue, 0 }  ][line width=0.75]      (0, 0) circle [x radius= 1.34, y radius= 1.34]   ;
 
\draw    (173.67,240.68) -- (183.77,240.68) ;
\draw [shift={(183.77,240.68)}, rotate = 0] [color={rgb, 255:red, 0; green, 0; blue, 0 }  ][fill={rgb, 255:red, 0; green, 0; blue, 0 }  ][line width=0.75]      (0, 0) circle [x radius= 1.34, y radius= 1.34]   ;
\draw [shift={(173.67,240.68)}, rotate = 0] [color={rgb, 255:red, 0; green, 0; blue, 0 }  ][fill={rgb, 255:red, 0; green, 0; blue, 0 }  ][line width=0.75]      (0, 0) circle [x radius= 1.34, y radius= 1.34]   ;
 
\draw  [dash pattern={on 0.84pt off 2.51pt}]  (228.97,190.68) -- (178.42,215.68) ;
\draw [shift={(178.42,215.68)}, rotate = 153.68] [color={rgb, 255:red, 0; green, 0; blue, 0 }  ][fill={rgb, 255:red, 0; green, 0; blue, 0 }  ][line width=0.75]      (0, 0) circle [x radius= 1.34, y radius= 1.34]   ;
\draw [shift={(228.97,190.68)}, rotate = 153.68] [color={rgb, 255:red, 0; green, 0; blue, 0 }  ][fill={rgb, 255:red, 0; green, 0; blue, 0 }  ][line width=0.75]      (0, 0) circle [x radius= 1.34, y radius= 1.34]   ;
 
\draw  [dash pattern={on 0.84pt off 2.51pt}]  (147.97,190.68) -- (178.42,215.68) ;
 
\draw  [dash pattern={on 0.84pt off 2.51pt}]  (190.01,193.89) -- (202.75,193.69) ;
 
\draw    (127.87,240.68) ;
\draw [shift={(127.87,240.68)}, rotate = 0] [color={rgb, 255:red, 0; green, 0; blue, 0 }  ][fill={rgb, 255:red, 0; green, 0; blue, 0 }  ][line width=0.75]      (0, 0) circle [x radius= 1.34, y radius= 1.34]   ;
 
\draw  [dash pattern={on 0.84pt off 2.51pt}]  (228.97,190.68) -- (208.87,240.68) ;
\draw [shift={(208.87,240.68)}, rotate = 111.9] [color={rgb, 255:red, 0; green, 0; blue, 0 }  ][fill={rgb, 255:red, 0; green, 0; blue, 0 }  ][line width=0.75]      (0, 0) circle [x radius= 1.34, y radius= 1.34]   ;
\draw [shift={(218.92,215.68)}, rotate = 111.9] [color={rgb, 255:red, 0; green, 0; blue, 0 }  ][fill={rgb, 255:red, 0; green, 0; blue, 0 }  ][line width=0.75]      (0, 0) circle [x radius= 1.34, y radius= 1.34]   ;
\draw [shift={(228.97,190.68)}, rotate = 111.9] [color={rgb, 255:red, 0; green, 0; blue, 0 }  ][fill={rgb, 255:red, 0; green, 0; blue, 0 }  ][line width=0.75]      (0, 0) circle [x radius= 1.34, y radius= 1.34]   ;
 
\draw [color={rgb, 255:red, 245; green, 166; blue, 35 }  ,draw opacity=1 ]   (345.97,190.68) -- (376.42,215.68) ;
 
\draw [color={rgb, 255:red, 245; green, 166; blue, 35 }  ,draw opacity=1 ]   (426.97,190.68) -- (376.42,215.68) ;
 
\draw    (345.97,190.68) -- (325.87,240.68) ;
 
\draw [line width=1.5]    (312.87,240.68) -- (436.97,240.68) ;
 
\draw  [dash pattern={on 0.84pt off 2.51pt}]  (365,237.44) -- (352.98,237.72) ;
 
\draw  [dash pattern={on 0.84pt off 2.51pt}]  (392.26,237.64) -- (399,237.44) ;
 
\draw [line width=1.5]    (313.87,190.68) -- (436.97,190.68) ;
 
\draw  [dash pattern={on 0.84pt off 2.51pt}]  (353.26,193.64) -- (360,193.44) ;
 
\draw    (335.87,240.68) -- (345.97,240.68) ;
\draw [shift={(345.97,240.68)}, rotate = 0] [color={rgb, 255:red, 0; green, 0; blue, 0 }  ][fill={rgb, 255:red, 0; green, 0; blue, 0 }  ][line width=0.75]      (0, 0) circle [x radius= 1.34, y radius= 1.34]   ;
\draw [shift={(335.87,240.68)}, rotate = 0] [color={rgb, 255:red, 0; green, 0; blue, 0 }  ][fill={rgb, 255:red, 0; green, 0; blue, 0 }  ][line width=0.75]      (0, 0) circle [x radius= 1.34, y radius= 1.34]   ;
 
\draw    (406.87,190.68) -- (416.97,190.68) ;
\draw [shift={(416.97,190.68)}, rotate = 0] [color={rgb, 255:red, 0; green, 0; blue, 0 }  ][fill={rgb, 255:red, 0; green, 0; blue, 0 }  ][line width=0.75]      (0, 0) circle [x radius= 1.34, y radius= 1.34]   ;
\draw [shift={(406.87,190.68)}, rotate = 0] [color={rgb, 255:red, 0; green, 0; blue, 0 }  ][fill={rgb, 255:red, 0; green, 0; blue, 0 }  ][line width=0.75]      (0, 0) circle [x radius= 1.34, y radius= 1.34]   ;
 
\draw    (406.87,240.68) -- (416.97,240.68) ;
\draw [shift={(416.97,240.68)}, rotate = 0] [color={rgb, 255:red, 0; green, 0; blue, 0 }  ][fill={rgb, 255:red, 0; green, 0; blue, 0 }  ][line width=0.75]      (0, 0) circle [x radius= 1.34, y radius= 1.34]   ;
\draw [shift={(406.87,240.68)}, rotate = 0] [color={rgb, 255:red, 0; green, 0; blue, 0 }  ][fill={rgb, 255:red, 0; green, 0; blue, 0 }  ][line width=0.75]      (0, 0) circle [x radius= 1.34, y radius= 1.34]   ;
 
\draw    (360.97,190.68) -- (371.07,190.68) ;
\draw [shift={(371.07,190.68)}, rotate = 0] [color={rgb, 255:red, 0; green, 0; blue, 0 }  ][fill={rgb, 255:red, 0; green, 0; blue, 0 }  ][line width=0.75]      (0, 0) circle [x radius= 1.34, y radius= 1.34]   ;
\draw [shift={(360.97,190.68)}, rotate = 0] [color={rgb, 255:red, 0; green, 0; blue, 0 }  ][fill={rgb, 255:red, 0; green, 0; blue, 0 }  ][line width=0.75]      (0, 0) circle [x radius= 1.34, y radius= 1.34]   ;
 
\draw    (381.77,240.68) -- (391.87,240.68) ;
\draw [shift={(391.87,240.68)}, rotate = 0] [color={rgb, 255:red, 0; green, 0; blue, 0 }  ][fill={rgb, 255:red, 0; green, 0; blue, 0 }  ][line width=0.75]      (0, 0) circle [x radius= 1.34, y radius= 1.34]   ;
\draw [shift={(381.77,240.68)}, rotate = 0] [color={rgb, 255:red, 0; green, 0; blue, 0 }  ][fill={rgb, 255:red, 0; green, 0; blue, 0 }  ][line width=0.75]      (0, 0) circle [x radius= 1.34, y radius= 1.34]   ;
 
\draw  [dash pattern={on 0.84pt off 2.51pt}]  (345.97,190.68) -- (325.87,240.68) ;
\draw [shift={(325.87,240.68)}, rotate = 111.9] [color={rgb, 255:red, 0; green, 0; blue, 0 }  ][fill={rgb, 255:red, 0; green, 0; blue, 0 }  ][line width=0.75]      (0, 0) circle [x radius= 1.34, y radius= 1.34]   ;
\draw [shift={(335.92,215.68)}, rotate = 111.9] [color={rgb, 255:red, 0; green, 0; blue, 0 }  ][fill={rgb, 255:red, 0; green, 0; blue, 0 }  ][line width=0.75]      (0, 0) circle [x radius= 1.34, y radius= 1.34]   ;
\draw [shift={(345.97,190.68)}, rotate = 111.9] [color={rgb, 255:red, 0; green, 0; blue, 0 }  ][fill={rgb, 255:red, 0; green, 0; blue, 0 }  ][line width=0.75]      (0, 0) circle [x radius= 1.34, y radius= 1.34]   ;
 
\draw    (371.07,190.68) -- (381.17,190.68) ;
\draw [shift={(381.17,190.68)}, rotate = 0] [color={rgb, 255:red, 0; green, 0; blue, 0 }  ][fill={rgb, 255:red, 0; green, 0; blue, 0 }  ][line width=0.75]      (0, 0) circle [x radius= 1.34, y radius= 1.34]   ;
\draw [shift={(371.07,190.68)}, rotate = 0] [color={rgb, 255:red, 0; green, 0; blue, 0 }  ][fill={rgb, 255:red, 0; green, 0; blue, 0 }  ][line width=0.75]      (0, 0) circle [x radius= 1.34, y radius= 1.34]   ;
 
\draw    (371.67,240.68) -- (381.77,240.68) ;
\draw [shift={(381.77,240.68)}, rotate = 0] [color={rgb, 255:red, 0; green, 0; blue, 0 }  ][fill={rgb, 255:red, 0; green, 0; blue, 0 }  ][line width=0.75]      (0, 0) circle [x radius= 1.34, y radius= 1.34]   ;
\draw [shift={(371.67,240.68)}, rotate = 0] [color={rgb, 255:red, 0; green, 0; blue, 0 }  ][fill={rgb, 255:red, 0; green, 0; blue, 0 }  ][line width=0.75]      (0, 0) circle [x radius= 1.34, y radius= 1.34]   ;
 
\draw  [dash pattern={on 0.84pt off 2.51pt}]  (376.42,215.68) -- (325.87,240.68) ;
\draw [shift={(325.87,240.68)}, rotate = 153.68] [color={rgb, 255:red, 0; green, 0; blue, 0 }  ][fill={rgb, 255:red, 0; green, 0; blue, 0 }  ][line width=0.75]      (0, 0) circle [x radius= 1.34, y radius= 1.34]   ;
\draw [shift={(376.42,215.68)}, rotate = 153.68] [color={rgb, 255:red, 0; green, 0; blue, 0 }  ][fill={rgb, 255:red, 0; green, 0; blue, 0 }  ][line width=0.75]      (0, 0) circle [x radius= 1.34, y radius= 1.34]   ;
 
\draw  [dash pattern={on 0.84pt off 2.51pt}]  (376.42,215.68) -- (406.87,240.68) ;
 
\draw  [dash pattern={on 0.84pt off 2.51pt}]  (388.01,193.89) -- (400.75,193.69) ;
 
\draw    (325.87,240.68) ;
\draw [shift={(325.87,240.68)}, rotate = 0] [color={rgb, 255:red, 0; green, 0; blue, 0 }  ][fill={rgb, 255:red, 0; green, 0; blue, 0 }  ][line width=0.75]      (0, 0) circle [x radius= 1.34, y radius= 1.34]   ;
 
\draw  [dash pattern={on 0.84pt off 2.51pt}]  (426.97,190.68) -- (406.87,240.68) ;
\draw [shift={(406.87,240.68)}, rotate = 111.9] [color={rgb, 255:red, 0; green, 0; blue, 0 }  ][fill={rgb, 255:red, 0; green, 0; blue, 0 }  ][line width=0.75]      (0, 0) circle [x radius= 1.34, y radius= 1.34]   ;
\draw [shift={(416.92,215.68)}, rotate = 111.9] [color={rgb, 255:red, 0; green, 0; blue, 0 }  ][fill={rgb, 255:red, 0; green, 0; blue, 0 }  ][line width=0.75]      (0, 0) circle [x radius= 1.34, y radius= 1.34]   ;
\draw [shift={(426.97,190.68)}, rotate = 111.9] [color={rgb, 255:red, 0; green, 0; blue, 0 }  ][fill={rgb, 255:red, 0; green, 0; blue, 0 }  ][line width=0.75]      (0, 0) circle [x radius= 1.34, y radius= 1.34]   ;
 
\draw   (138.25,244.65) .. controls (138.22,249.32) and (140.53,251.67) .. (145.2,251.7) -- (163.44,251.83) .. controls (170.11,251.88) and (173.42,254.23) .. (173.39,258.9) .. controls (173.42,254.23) and (176.77,251.92) .. (183.44,251.97)(180.44,251.95) -- (201.7,252.1) .. controls (206.37,252.13) and (208.72,249.82) .. (208.75,245.15) ;
 
\draw   (336.25,244.9) .. controls (336.22,249.57) and (338.53,251.92) .. (343.2,251.95) -- (361.44,252.08) .. controls (368.11,252.13) and (371.42,254.48) .. (371.39,259.15) .. controls (371.42,254.48) and (374.77,252.17) .. (381.44,252.22)(378.44,252.2) -- (399.7,252.35) .. controls (404.37,252.38) and (406.72,250.07) .. (406.75,245.4) ;

\draw (180.42,219.08) node [anchor=north west][inner sep=0.75pt]  [font=\tiny]  {$\times _{1}$};
 
\draw (139.87,219.08) node [anchor=north west][inner sep=0.75pt]  [font=\tiny]  {$\times _{2}$};
 
\draw (101,186.8) node [anchor=north west][inner sep=0.75pt]  [font=\tiny]  {$\partial $};
 
\draw (99.5,236.8) node [anchor=north west][inner sep=0.75pt]  [font=\tiny]  {$\partial ^{\prime }$};
 
\draw (161,161.5) node [anchor=north west][inner sep=0.75pt]    {$\epsilon =+$};
 
\draw (378.42,219.08) node [anchor=north west][inner sep=0.75pt]  [font=\tiny]  {$\times _{1}$};
 
\draw (337.87,219.08) node [anchor=north west][inner sep=0.75pt]  [font=\tiny]  {$\times _{2}$};
 
\draw (299,186.8) node [anchor=north west][inner sep=0.75pt]  [font=\tiny]  {$\partial $};
 
\draw (297.5,236.8) node [anchor=north west][inner sep=0.75pt]  [font=\tiny]  {$\partial ^{\prime }$};
 
\draw (359,161.5) node [anchor=north west][inner sep=0.75pt]    {$\epsilon =-$};
 
\draw (147.93,263.15) node [anchor=north west][inner sep=0.75pt]  [font=\tiny]  {$n\ marked\ points\ $};
 
\draw (345.93,263.4) node [anchor=north west][inner sep=0.75pt]  [font=\tiny]  {$n\ marked\ points\ $};
 
\draw (171,226.5) node [anchor=north west][inner sep=0.75pt]  [font=\tiny]  {$?$};
 
\draw (175.5,199.5) node [anchor=north west][inner sep=0.75pt]  [font=\tiny]  {$?$};
 
\draw (376.5,199) node [anchor=north west][inner sep=0.75pt]  [font=\tiny]  {$?$};
 
\draw (370,227.5) node [anchor=north west][inner sep=0.75pt]  [font=\tiny]  {$?$};
\end{tikzpicture}
 \end{figure}
\noindent It follows that these two skew-arcs are in $\Lambda_{\times_1}^{\neg\epsilon}$, and thus $|\Lambda_{\times_1}^{\neg\epsilon}| \geq 2$. 
\end{proof}

For simplicity, if $\zeta(\Lambda)_i = \pm$ and satisfies condition $A_0$ (resp. $A_1$, $A_2$), we abbreviate it as $\zeta(\Lambda)_i = \pm(0)$ (resp. $\pm(1)$, $\pm(2)$).

 Next, we come to determine the number $|\Lambda|$ of skew-arcs in $\Lambda$. Let $P_{\Lambda} = \{ \times_i \mid \zeta(\Lambda)_i \neq \pm \}\subset \mathcal{X}$  and let $L_{\times_1}$ (resp. $L_{\times_2}$) be the loop with endpoints at $\times_1$ (resp. $\times_2$). Consider the marked surface $(\mathcal{S}, M \cup P_{\Lambda})$, where the points in $P_{\Lambda}$ serve as punctures on $\mathcal{S}$. Note that skew-curves in  $\mathbf{C}_{\mathtt{b},\times}$ ($\times\in P_\Lambda$) or in $\Lambda_{\mathtt{b}}^\sigma \cup \Lambda_{\mathtt{p}}^\sigma$ can naturally be viewed as curves on  $(\mathcal{S}, M \cup P_{\Lambda})$. We define a set of curves $\Gamma_{\Lambda}$ on $(\mathcal{S}, M \cup P_{\Lambda})$, associated with $\Lambda$, as follows:

For each skew-arc $\widehat{\gamma} \in \Lambda_{\mathtt{b}}^\sigma \cup \Lambda_{\mathtt{p}}^\sigma$, add all curves in $\mathbf{C}(\widehat{\gamma})$ to $\Gamma_{\Lambda}$, and the proceed according to the values of $\zeta(\Lambda)_i$ for $i = 1, 2$:
\begin{itemize}
    \item[(i)] If $\zeta(\Lambda)_i = \pm(i)$ for $i = 1$ or $2$, add $L_{\times} (\times\neq \times_i)$ to $\Gamma_{\Lambda}$.
    \item[(ii)] If $\zeta(\Lambda)_i = \pm(0)$ for $i = 1$ or $2$, assume that $\Lambda_{\times_i} = \{\widehat{\gamma}, \sigma(\widehat{\gamma})\}$. Add the unique curve in $\mathbf{C}(\widehat{\gamma})$ to $\Gamma_{\Lambda}$.
    \item[(iii)] If $\zeta(\Lambda)_i \in \{+, -\}$ for $i = 1$ or $2$, for each $\widehat{\gamma} \in \Lambda_{\mathtt{b}, \times_i}^{\zeta(\Lambda)_i}$, add  $\widehat{\gamma}$ and $\sigma(\widehat{\gamma})$ to $\Gamma_{\Lambda}$.
    \item[(iv)] If $\zeta(\Lambda)_1, \zeta(\Lambda)_2 \in \{+, -\}$ and $\Lambda_* \neq \emptyset$, add $L_0$ and $L_1$ to $\Gamma_{\Lambda}$.
\end{itemize}
 
By Proposition \ref{prop:6.4} and Lemma \ref{lem:6.4}, one can check that $\Gamma_{\Lambda}$ is a triangulation on $(\mathcal{S}, M \cup P_{\Lambda})$. Moreover, $\Gamma_{\Lambda}$ is $\sigma$-fixed. 

\begin{exm}
Let $n=4$. We give some pseudo-triangulations $\Lambda$ on $(\mathcal{S}, M,\sigma)$ and the corresponding triangulations $\Gamma_{\Lambda}$  on $(\mathcal{S}, M \cup P_{\Lambda})$.
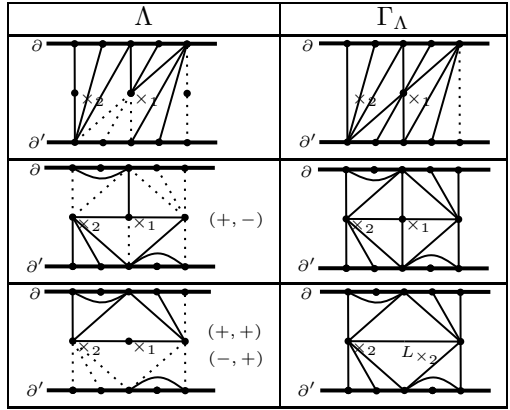
\[
\begin{array}{|c|c|}
\hline
\Lambda & \Gamma_{\Lambda} \\
\hline
\end{array}
\]
\end{exm}

\begin{prop}
	The number of skew-arcs  in $\Lambda$ is $|\Lambda| = n + 3.$
\end{prop}
\begin{proof}
The main idea of this proof is to use the process of constructing $\Gamma_{\Lambda}$ from $\Lambda$ and leverage $|\Gamma_{\Lambda}|$ to determine $|\Lambda|$. From the equation \ref{formula}, we have  $|\Gamma_{\Lambda}|=2n+3|P_{\Lambda}|$.  We analyze the cases based on the types of $\Lambda$: 
   
    (1) $\Lambda_* \neq \emptyset$ and $|\Lambda_*|=2$. Assume that $\Lambda_*=\{(\epsilon_1, \epsilon_2), (\epsilon_1^\prime, \epsilon_2^\prime)\}$, then $\epsilon_1=\epsilon_1^\prime$ or $\epsilon_2=\epsilon_2^\prime$, and $|P_{\Lambda}|=1$. We may assume, without loss of generality, that $\epsilon_1=\epsilon_1^\prime$. Note that 
    $$|\Gamma_{\Lambda}|=2|\Lambda_{\mathtt{b},  \times_1}^{\epsilon_1}\cup \Lambda_{\mathtt{p}}^\sigma|+\frac{1}{2}|\Lambda_*|.$$ We have
    $$|\Lambda| =|\Lambda_{\mathtt{b},  \times_1}^{\epsilon_1}\cup \Lambda_{\mathtt{p}}^\sigma|+|\Lambda_*|=\frac{|\Gamma_{\Lambda}|-\frac{1}{2}|\Lambda_*|}{2}+|\Lambda_*| = n+3.$$
    
    (2) $\Lambda_* \neq \emptyset$ and $|\Lambda_*|=1$. Assume that $\Lambda_*=\{(\epsilon_1, \epsilon_2)\}$. In this case,  $|P_{\Lambda}|=2$ and
    $$|\Gamma_{\Lambda}|=2|\Lambda_{\mathtt{b},  \times_1}^{\epsilon_1} \cup \Lambda_{\mathtt{b},  \times_2}^{\epsilon_2} \cup \Lambda_{\mathtt{p}}^\sigma\cup\Lambda_*|.$$ Thus,
    $$|\Lambda| =|\Lambda_{\mathtt{b},  \times_1}^{\epsilon_1} \cup \Lambda_{\mathtt{b},  \times_2}^{\epsilon_2} \cup \Lambda_{\mathtt{p}}^\sigma\cup\Lambda_*|=\frac{1}{2}|\Gamma_{\Lambda}| = n+3.$$
   
    (3) $\Lambda_* = \emptyset$ and $\zeta(\Lambda) = (\pm, \pm)$. Then we have $|P_{\Lambda}|=0$, $|\Lambda_{\mathtt{b},\mathcal{X}}| = 4$ and $$|\Gamma_{\Lambda}| = \frac{1}{2}|\Lambda_{\mathtt{b},\mathcal{X}}| + 2|\Lambda_{\mathtt{b}}^\sigma \cup \Lambda_{\mathtt{p}}^\sigma|.$$ Hence, 
    $$|\Lambda| = |\Lambda_{\mathtt{b},\mathcal{X}}\cup\Lambda_{\mathtt{b}}^\sigma \cup \Lambda_{\mathtt{p}}^\sigma|=\frac{|\Gamma_{\Lambda}| - \frac{1}{2}|\Lambda_{\mathtt{b},\mathcal{X}}|}{2} + |\Lambda_{\mathtt{b},\mathcal{X}}| = n + 3.$$
    
    (4) $\Lambda_* = \emptyset$ and $\zeta(\Lambda) = (\pm, +)$ or $(\pm, -)$. Then $|P_{\Lambda}|=1$,  $|\Lambda_{\times_1}|=2$ and  $$|\Gamma_{\Lambda}| = \frac{1}{2}|\Lambda_{\times_1}| + 2|\Lambda_{\mathtt{b}}^\sigma \cup \Lambda_{\mathtt{p}}^\sigma\cup \Lambda_{\times_2}|.$$ Thus, 
    $$|\Lambda| =|\Lambda_{\mathtt{b}}^\sigma \cup \Lambda_{\mathtt{p}}^\sigma\cup\Lambda_{\times_1}\cup \Lambda_{\times_2}|= \frac{|\Gamma_{\Lambda}| - \frac{1}{2}|\Lambda_{\times_1}|}{2} + |\Lambda_{\times_1}| = n + 3.$$
    
    (5) $\Lambda_* = \emptyset$ and $\zeta(\Lambda) = (+, \pm)$ or $(-,\pm)$. This case is similar to (4).
    
    (6) $\Lambda_* = \emptyset$ and $\zeta(\Lambda)=(\epsilon_1,\epsilon_2)$ with $\epsilon_1, \epsilon_2\in \{+,-\}$. Then $|P_{\Lambda}|=2$ and $$|\Gamma_{\Lambda}| = 2|\Lambda_{\mathtt{b},\mathcal{X}}\cup\Lambda_{\mathtt{b}}^\sigma \cup \Lambda_{\mathtt{p}}^\sigma|.$$ Thus, 
    $$|\Lambda| =|\Lambda_{\mathtt{b},\mathcal{X}}\cup\Lambda_{\mathtt{b}}^\sigma \cup \Lambda_{\mathtt{p}}^\sigma|=  \frac{1}{2}|\Gamma_{\Lambda}| = n + 3.$$
\end{proof}

\begin{prop}\label{tilting}
	Tilting sheaves in ${\rm coh}\mbox{-}\mathbb{X}$ correspond one-to-one with pseudo-triangulations on $(\mathcal{S}, M,\sigma)$.
\end{prop}
\begin{proof}
	Let $\Lambda=\{\widehat{\gamma}_1, \widehat{\gamma}_2, \cdots, \widehat{\gamma}_{n+3}\}$ be a pseudo-triangulation on $(\mathcal{S}, M,\sigma)$. Let $$T=\bigoplus\limits_{i=1}^{n+3}\widehat{\phi}(\widehat{\gamma}_i).$$ According to the definition of pseudo-triangulation, we have $${\rm dim_{\mathbf{k}}Ext_{\mathbb{X}}^{1}}(\widehat{\phi}(\widehat{\gamma}_i),\, \widehat{\phi}(\widehat{\gamma}_j))=0$$ for $1\leq i,\, j\leq n+3$.
 This implies ${\rm dim_{\mathbf{k}}Ext_{\mathbb{X}}^{1}}(T,T)=0$. Combining with Proposition \ref{tilting2}, we get that $T$ is a tilting sheaf in ${\rm coh}\mbox{-}\mathbb{X}$. 
 
 Conversely, if $T=\oplus_{i=1}^{n+3}T_{i}$ is a tilting sheaf in ${\rm coh}\mbox{-}\mathbb{X} $, then $\operatorname{Ext}_{\mathbb{X}}^1(T,T)=0$. By the bijection $\widehat{\phi}$, $\{\widehat{\phi}^{-1}(T_1), \widehat{\phi}^{-1}(T_2), \cdots, \widehat{\phi}^{-1}(T_{n+3})\}$ forms a pseudo-triangulation on $(\mathcal{S}, M,\sigma)$.
\end{proof}

In the later, we write $\widehat{\phi}(\Lambda) := \oplus_{\widehat{\gamma} \in \Lambda} \widehat{\phi}(\widehat{\gamma})$ the tilting sheaf corresponding to $\Lambda$.

\subsection{Flips and mutations}
For a detailed discussion on the flip of a skew-arc in a pseudo-triangulation, see Appendix \ref{foh}. In this subsection, we examine the compatibility between the flip of a skew-arc in a pseudo-triangulation and the mutation of the indecomposable direct summand of the corresponding tilting sheaf.

 Recall that each tilting sheaf in ${\rm coh}\mbox{-}\mathbb{X}$  has $n+3$ indecomposable direct  summands. A rigid sheaf in ${\rm coh}\mbox{-}\mathbb{X}$ is called  \emph{almost complete tilting} if it has $n+2$-many (pairwise non-isomorphic) indecomposable direct summands. Let $\overline{T}$ be an almost complete tilting sheaf in ${\rm coh}\mbox{-}\mathbb{X}$. If there exists an indecomposable sheaf $E$, such that $\overline{T}\oplus E$ is a tilting sheaf, then $E$ is called a \emph{complement} of $\overline{T}$. As shown in \cite{Hubner1996Dissertation}, there are exactly two complements of $\overline{T}$. Let $T=\overline{T}\oplus X$ and $T^{\prime}=\overline{T}\oplus X^{\prime}$ be two tilting sheaves in ${\rm coh}\mbox{-}\mathbb{X}$, where $X$ and $X^{\prime}$ are two non-isomorphic indecomposable objects in ${\rm coh}\mbox{-}\mathbb{X}$. Following \cite{Geng2020MutationTilting}, the sheaf $T^{\prime}$ is called the \emph{mutation} of $T$ at $X$ and denoted by $T^{\prime}=\mu_{X}(T)$. The summand $X^{\prime}$  the \emph{mutation} of $X$ with respect to $T$ and denote by $\mu_{T}(X)=X^{\prime}.$

\begin{cor}\label{mutation and flip}
Let $\Lambda$ be a pseudo-triangulation on $(\mathcal{S}, M,\sigma)$ and $\widehat{\gamma}$ be a skew-arc in $\Lambda$. Then \[\widehat{\phi}(\widehat{\mu}_{\Lambda}(\widehat{\gamma}))=\mu_{T}(\widehat{\phi}(\widehat{\gamma})).\]
\end{cor}
\begin{proof}
    By Proposition \ref{tilting}, we know that $\Lambda \setminus \{\widehat{\gamma}\}$ corresponds to an almost complete tilting sheaf $\overline{T}$ in the category ${\rm coh}\mbox{-}\mathbb{X}$. Clearly, the two complements of $\overline{T}$ are $\widehat{\phi}(\widehat{\gamma})$ and $\widehat{\mu}_{\Lambda}(\widehat{\gamma})$.  Thus, our claim follows.
\end{proof}

\section{Connectedness of the tilting graph $\mathcal{G}(\mathcal{T}_{\mathbb{X}})$}\label{connectedness}
 In this section, we employ geometric-combinatorial approaches to prove the connectivity of the tilting graph of the category ${\rm coh}\mbox{-}\mathbb{X}$. Recall that the \emph{tilting graph} $\mathcal{G}(\mathcal{T}_{\mathbb{X}})$ of the category ${\rm coh}\mbox{-}\mathbb{X}$ has as vertices the isoclasses of basic tilting sheaves of ${\rm coh}\mbox{-}\mathbb{X}$, while two vertices $T$ and $T^\prime$ are connected by an edge if and only if they differ by precisely one indecomposable direct summand.

Let $T$ and $T^{\prime}$ be tilting sheaves in ${\rm coh}\mbox{-}\mathbb{X}$.  $T$ can be \emph{mutated to}  $T^{\prime}$ if there exists a sequence of tilting sheaves $T = T^{0}, T^{1}, \dots, T^{n-1}, T^{n} = T^{\prime}$ such that $T^{i}$ is a mutation of $T^{i-1}$ for every $1 \leq i \leq n$. Furthermore, under the bijection in Proposition \ref{tilting}, we say a pseudo-triangulation $\Lambda$ can be \emph{flipped into} $\Lambda^{\prime}$ if the associated tilting sheaves $T$ and $T^{\prime}$ satisfy that $T$ can be \emph{mutated to} $T^{\prime}$. In other word, $\Lambda$ can be transformed into $\Lambda^{\prime}$ by a sequence of flips of skew-arcs. 

\begin{prop}\label{FV-pseudo-triangulations}
Let $\Lambda$ be a pseudo-triangulation of the form
$$
\{[D^{a_0}_{b_0}]^+, [D^{a_0}_{b_0}]^-, \widetilde{[D^{a_i}_{b_i}]}, [D^{a_n}_{b_n}]^+, [D^{a_n}_{b_n}]^- \mid 1 \leq i \leq n-1 \},
$$
where $a_0 + b_0 = 0$, $a_n + b_n = n$, and
\begin{equation}\label{assumption on chains}
a_0 \leq a_1 \leq \cdots \leq a_n, \quad b_0 \leq b_1 \leq \cdots \leq b_n.
\end{equation}
Then, $\widehat{\phi}(\Lambda)$ is a $\mathbb{Z}(\vec{x}_1 - \vec{x}_2)$-stable tilting bundle in ${\rm coh}\mbox{-}\mathbb{X}$. The converse also holds.
\end{prop}
\begin{proof}
Assume that $\Lambda$ has the form described in the proposition. Since $\Lambda \subset \mathbf{C}_{\mathtt{b},\mathcal{X}} \cup \mathbf{C}_{\mathtt{b}}^\sigma$, Proposition \ref{tilting} implies that $\widehat{\phi}(\Lambda)$ is a tilting bundle in $\operatorname{coh}\mbox{-}\mathbb{X}$. Furthermore, by Proposition \ref{L-action}, we know that $\widehat{\phi}(\Lambda)(\vec{x}_1 - \vec{x}_2) = \widehat{\phi}(\Lambda)$, meaning that $\widehat{\phi}(\Lambda)$ is $\mathbb{Z}(\vec{x}_1 - \vec{x}_2)$-stable.

Conversely, suppose that $T$ is a $\mathbb{Z}(\vec{x}_1 - \vec{x}_2)$-stable tilting bundle in $\operatorname{coh}\mbox{-}\mathbb{X}$ and let $\Lambda = \widehat{\phi}^{-1}(T)$. By Proposition \ref{tilting}, $\Lambda \subset \mathbf{C}_{\mathtt{b},\mathcal{X}} \cup \mathbf{C}_{\mathtt{b}}^\sigma$, and for each indecomposable line bundle summand $L$ of $T$, there is an indecomposable summand $L' \cong L(\vec{x}_1 - \vec{x}_2)$ of $T$. By Proposition \ref{L-action}, both $\widehat{\phi}^{-1}(L)$ and $\sigma(\widehat{\phi}^{-1}(L) )$ belong to $\Lambda$. Therefore,  $\Lambda$ has the required form.
\end{proof}

We refer to the pseudo-triangulations $\Lambda$ stated in the above proposition as \emph{FV-pseudo-triangulations}. For a $\mathbb{Z}(\vec{x}_1 - \vec{x}_2)$-stable tilting object $T$, Proposition \ref{FV-pseudo-triangulations} allows us to order its indecomposable direct summands as follows:
\[
T = T_0^+ \oplus T_0^- \oplus T_1 \oplus \cdots \oplus T_{n-1} \oplus T_n^+ \oplus T_n^-,
\]
\noindent where:  
\[
T_0^+ = \widehat{\phi}([D^{a_0}_{b_0}]^+), \quad T_0^- = \widehat{\phi}([D^{a_0}_{b_0}]^-),
\]
\[
T_i = \widehat{\phi}(\widetilde{[D^{a_i}_{b_i}]}) \quad \text{for } 1 \leq i \leq n-1,
\]
\noindent and
\[
T_n^+ = \widehat{\phi}([D^{a_n}_{b_n}]^+), \quad T_n^- = \widehat{\phi}([D^{a_n}_{b_n}]^-).
\]

 Let $\widehat{\mu}_0(T)$ denote the mutation of $\mu_{T_0^+}(T)$ at $T_0^-$ (or equivalently, the mutation of $\mu_{T_0^-}(T)$ at $T_0^+$). Similarly, define $\widehat{\mu}_n(T)$, and for $1 \leq i \leq n-1$, let $\widehat{\mu}_i(T)$ represent the mutation of $T$ at $T_i$. We define the map
\[
\iota: \{\mathbb{Z}(\vec{x}_1 - \vec{x}_2) \text{-stable tilting bundles}\} \to \mathbb{Z}_2^{n+1}
\]
\[
T \mapsto (\iota(T)_0, \iota(T)_1, \ldots, \iota(T)_n),
\]
\noindent where 
\[
\iota(T)_i = \begin{cases} 
1 & \text{if } \widehat{\mu}_i(T) \text{ remains a } \mathbb{Z}(\vec{x}_1 - \vec{x}_2)\text{-stable tilting bundle}, \\
0 & \text{otherwise}.
\end{cases}
\]

For an integer $k$, let $\left\lfloor k \right\rfloor_{\text{o}}$ and $\left\lfloor k \right\rfloor_{\text{e}}$ denote the largest odd and even numbers less than or equal to $k$, respectively. We continue to consider FV-pseudo-triangulations of the following forms: 
\[
\overrightarrow{\Lambda}^a_b =\{ [D^{a}_{b}]^+, [D^{a}_{b}]^- , \widetilde{[D^{a+1}_{b}]} , \widetilde{[D^{a+1}_{b+1}]} , \cdots , [D^{a+\frac{\left\lfloor n+1 \right\rfloor_{\text{e}}}{2}}_{b+\frac{\left\lfloor n \right\rfloor_{\text{e}}}{2}}]^+ , [D^{a+\frac{\left\lfloor n+1 \right\rfloor_{\text{e}}}{2}}_{b+\frac{\left\lfloor n \right\rfloor_{\text{e}}}{2}}]^-\}
\]
\[
\underrightarrow{\Lambda}^a_b = \{[D^{a}_{b}]^+ , [D^{a}_{b}]^- , \widetilde{[D^{a}_{b+1}]} , \widetilde{[D^{a+1}_{b+1}]} , \cdots , [D^{a+\frac{\left\lfloor n \right\rfloor_{\text{e}}}{2}}_{b+\frac{\left\lfloor n+1 \right\rfloor_{\text{e}}}{2}}]^+ , [D^{a+\frac{\left\lfloor n \right\rfloor_{\text{e}}}{2}}_{b+\frac{\left\lfloor n+1 \right\rfloor_{\text{e}}}{2}}]^-\}
\]
where $a,b \in \mathbb{Z}$ such that $a+b=0$. The following figures are shown as illustrated.
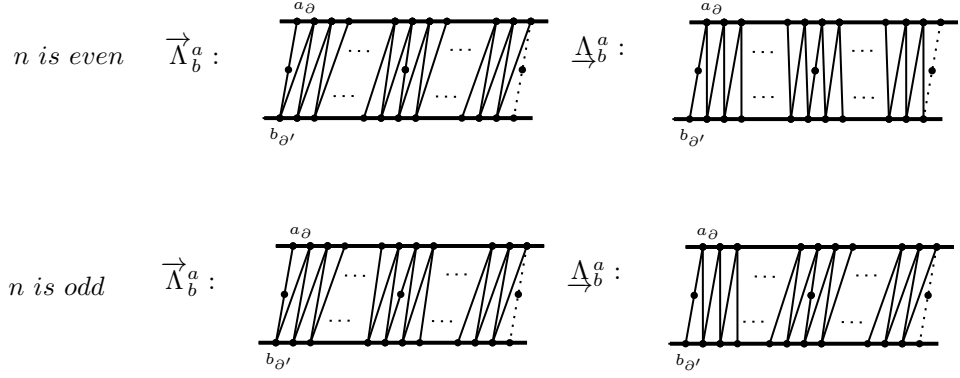
\begin{figure}[H]
    \centering

\tikzset{every picture/.style={line width=0.75pt}}  



\end{figure}

To describe the relationship between these FV-pseudo-triangulations, we present the following statement:
\begin{prop}
Let $a, b \in \mathbb{Z}$ such that $a + b = 0$. The following relations hold: 
\begin{itemize}
    \item[(a)]  $\widehat{\mu}_{\left\lfloor n \right\rfloor_{\text{o}}} \cdot \widehat{\mu}_{\left\lfloor n \right\rfloor_{\text{o}}-2} \cdots \widehat{\mu}_{1}\cdot\widehat{\phi}(\overrightarrow{\Lambda}^a_b) =\widehat{\phi}(\underrightarrow{\Lambda}^a_b)$;
    \item[(b)] $\widehat{\mu}_{\left\lfloor n \right\rfloor_{\text{o}}} \cdot \widehat{\mu}_{\left\lfloor n \right\rfloor_{\text{o}}-2} \cdots \widehat{\mu}_{1}\cdot\widehat{\phi}(\underrightarrow{\Lambda}^a_b) = \widehat{\phi}(\overrightarrow{\Lambda}^a_b)$;
    \item[(c)] $\widehat{\mu}_{\left\lfloor n \right\rfloor_{\text{e}}} \cdot \widehat{\mu}_{\left\lfloor n \right\rfloor_{\text{e}}-2} \cdots \widehat{\mu}_{0}\cdot\widehat{\phi}(\overrightarrow{\Lambda}^a_b) = \widehat{\phi}(\underrightarrow{\Lambda}^{a+1}_{b- 1});$
    \item[(d)]  $\widehat{\mu}_{\left\lfloor n \right\rfloor_{\text{e}}} \cdot \widehat{\mu}_{\left\lfloor n \right\rfloor_{\text{e}}-2} \cdots \widehat{\mu}_{0}\cdot\widehat{\phi}(\underrightarrow{\Lambda}^a_b) = \widehat{\phi}(\overrightarrow{\Lambda}^{a-1}_{b+1})$.  
\end{itemize}
\end{prop}
\begin{proof}
For even $n$, we have
\[
\begin{aligned}
\widehat{\mu}_{n} \cdot \widehat{\mu}_{n-2} \cdots \widehat{\mu}_{0} \cdot \widehat{\phi}(\overrightarrow{\Lambda}^a_b) 
&= \widehat{\phi}([D^{a+1}_{b-1}]^+) \oplus \widehat{\phi}([D^{a+1}_{b-1}]^-) \oplus \widehat{\phi}(\widetilde{[D^{a+1}_{b}]}) \\
&\oplus \cdots \oplus \widehat{\phi}([D^{a+1+\frac{n}{2}}_{b-1+\frac{n}{2}}]^+) \oplus \widehat{\phi}([D^{a+1+\frac{n}{2}}_{b-1+\frac{n}{2}}]^-) \\
&= \widehat{\phi}(\underrightarrow{\Lambda}^{a+1}_{b-1}).
\end{aligned}
\]
as shown in the following figure, the transition occurs from $\overrightarrow{\Lambda}^a_b$ (red and black skew-arcs) to $\underrightarrow{\Lambda}^{a+1}_{b-1}$ (green and black skew-arcs).
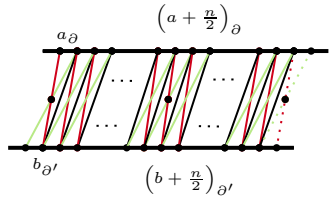
\begin{figure}[H]
    \centering

\tikzset{every picture/.style={line width=0.75pt}}  


\end{figure}
\noindent The remaining cases follow similarly. Hence, the proof is complete.
\end{proof}

As a result, by combining the above proposition with Corollary \ref{mutation and flip}, we have the following conclusion:
\begin{cor}\label{Tab}
The $\mathbb{Z}(\vec{x}_1 - \vec{x}_2)$-stable tilting bundles in 
\[\{ \widehat{\phi}(\overrightarrow{\Lambda}^a_b), \widehat{\phi}(\underrightarrow{\Lambda}^a_b) \mid a, b \in \mathbb{Z} \text{ such that } a + b = 0 \}\] can be mutated to each other in  $\mathcal{G}(\mathcal{T}_{\mathbb{X}})$.
\end{cor}

We present the following criterion for the $\mathbb{Z}(\vec{x}_1 - \vec{x}_2)$-stable tilting bundles given in Corollary \ref{Tab}:
\begin{prop}\label{111}
   A $\mathbb{Z}(\vec{x}_1 - \vec{x}_2)$-stable tilting bundle $T$  belongs to 
   \[\{\widehat{\phi}(\overrightarrow{\Lambda}^a_b), \widehat{\phi}(\underrightarrow{\Lambda}^a_b) \mid a, b \in \mathbb{Z} \text{ such that } a + b = 0\}\] if and only if
    $$\iota(T) = (\underbrace{1,1,\cdots,1}_{n\ \text{times}}).$$
\end{prop}
\begin{proof}
Compare with \cite[Propsition 5.4]{Chen2023GeometricModel}, it not difficult to check.
\end{proof}

The following observations can be verified directly:
\begin{lem}\label{flipsss}
   Assume that $T$ is a $\mathbb{Z}(\vec{x}_1 - \vec{x}_2)$-stable tilting bundle. 
   \begin{itemize}
       \item[(a)] If $\iota(T)=(1, 0, \iota(T)_{2}, \dots, \iota(T)_{n-1}, 1)$, then 
       $$\iota(\widehat{\mu}_0(T)) = (1, 1, \iota(T)_{2}, \dots, \iota(T)_{n-1}, 1);$$
       \item[(b)] If $\iota(T)=(1, 1, 0, \iota(T)_{3}, \dots, \iota(T)_{n-1}, 1)$, then $$\iota(\widehat{\mu}_1(T)) = (1, 1, 1, \iota(T)_{3}, \dots, \iota(T)_{n-1}, 1);$$
       \item[(c)] If $\iota(T)=(1, \iota(T)_{1}, \dots, \iota(T)_{i-2}, 1, 1, 0, \iota(T)_{i+2}, \dots, \iota(T)_{n-1}, 1)$ for $i \geq 2$, then $\iota(\widehat{\mu}_i(T)) = (1, \iota(T)_{1}, \dots, \iota(T)_{i-2}, 0, 1, 1, \iota(T)_{i+2}, \dots, \iota(T)_{n-1}, 1).$
   \end{itemize}
\end{lem}

\begin{prop}{}{}
   Any $\mathbb{Z}(\vec{x}_1 - \vec{x}_2)$-stable tilting bundle $T$ can be  mutated to some $\mathbb{Z}(\vec{x}_1 - \vec{x}_2)$-stable tilting bundle in $\{\widehat{\phi}(\overrightarrow{\Lambda}^a_b), \widehat{\phi}(\underrightarrow{\Lambda}^a_b)\mid a, b \in \mathbb{Z} \text{ such that } a + b = 0\}$. 
\end{prop}
\begin{proof}
    By Lemma \ref{flipsss} (a)(b), we may assume that  
    \[\iota(T) = (1, 1, 1, \iota(T)_3, \dots, \iota(T)_{n-1}, 1). \]
    
Then, consider $\iota(T)_3, \iota(T)_4, \dots, \iota(T)_{n-1}$. For $3 \leq j \leq n$, if $\iota(T)_j = 1$, then we consider $\iota(T)_{j+1}$. If $\iota(T)_j = 0$, by Lemma \ref{flipsss} (3), we can find a sequence of flips, denoted as $\nu_j$, such that either
$$\iota(\nu_j(T)) = (1, 0, 1, 1, \dots, 1, \iota(T)_{j+1}, \dots, \iota(T)_{n-1}, 1),$$
or
$$\iota(\nu_j(T)) = (1, 1, 0, 1, \dots, 1, \iota(T)_{j+1}, \dots, \iota(T)_{n-1}, 1).$$
In either case, by Lemma \ref{flipsss} (a)(b), we can find a sequence of flips, denoted as $\nu_j^\prime$, such that
$$\iota(\nu_j^\prime \nu_j(T)) = (1, 1, 1, 1, \dots, 1, \iota(T)_{j+1}, \dots, \iota(T)_{n-1}, 1).$$

We then proceed similarly for $\iota(T)_{j+1}$. By recursion, we eventually get a sequence of flips, denoted as $\nu$, such that
$$\iota(\nu(T)) = (1, 1, \dots, 1).$$
Thus, by Proposition \ref{111}, the proof is complete.
\end{proof}

It is well known that any two triangulations of a polygon can be flipped into each other.  Thus, for a pseudo-triangulation $\Lambda$ with $\Lambda_{\mathtt{b}}^\sigma \neq \emptyset$, and adopting Convention \ref{conven2}, we can flip $\Lambda$ into some pseudo-triangulation $\Lambda^{\prime}$ such that  $\Lambda_{org}^{\prime}\subset \mathbf{C}_{\mathtt{b}}^\sigma, \Lambda^{\prime}_{grn}=\Lambda_{grn}$ and $\Lambda^{\prime}_{blu}=\Lambda_{blu}$.

\begin{prop}\label{filp to FV}
    Any pseudo-triangulation $\Lambda$ with $\Lambda_{\mathtt{b}}^\sigma \neq \emptyset$  can be flipped into an FV-pseudo-triangulation.
\end{prop}
\begin{proof}
Without loss of generality, assume $\Lambda_{org} = \Lambda_{\mathtt{b}}^\sigma$ and $\zeta(\Lambda)_i \neq \pm(0)$ for all $i$. 

First, we consider $\Lambda_{grn}$. Note that the skew-arcs in $\Lambda_{grn}$ are locally confined within the yellow region shown below:
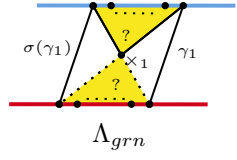
\begin{figure}[H]
    \centering

\tikzset{every picture/.style={line width=0.75pt}}  

\begin{tikzpicture}[x=0.75pt,y=0.75pt,yscale=-1,xscale=1]

\draw  [fill={rgb, 255:red, 248; green, 231; blue, 28 }  ,fill opacity=1 ][dash pattern={on 0.84pt off 2.51pt}] (133.47,136.81) -- (147.5,161.5) -- (102.5,161.5) -- cycle ;
 
\draw  [fill={rgb, 255:red, 248; green, 231; blue, 28 }  ,fill opacity=1 ] (164.37,111.93) -- (133.47,136.81) -- (119.41,112.02) -- cycle ;
 
\draw [color={rgb, 255:red, 208; green, 2; blue, 27 }  ,draw opacity=1 ][line width=1.5]    (77.5,161.5) -- (176.54,161.59) ;
 
\draw [color={rgb, 255:red, 104; green, 161; blue, 226 }  ,draw opacity=1 ][line width=1.5]    (91.34,111.83) -- (190.37,111.93) ;
 
\draw    (102.5,161.5) ;
\draw [shift={(102.5,161.5)}, rotate = 0] [color={rgb, 255:red, 0; green, 0; blue, 0 }  ][fill={rgb, 255:red, 0; green, 0; blue, 0 }  ][line width=0.75]      (0, 0) circle [x radius= 1.34, y radius= 1.34]   ;
\draw [shift={(102.5,161.5)}, rotate = 0] [color={rgb, 255:red, 0; green, 0; blue, 0 }  ][fill={rgb, 255:red, 0; green, 0; blue, 0 }  ][line width=0.75]      (0, 0) circle [x radius= 1.34, y radius= 1.34]   ;
 
\draw    (147.5,161.5) ;
\draw [shift={(147.5,161.5)}, rotate = 0] [color={rgb, 255:red, 0; green, 0; blue, 0 }  ][fill={rgb, 255:red, 0; green, 0; blue, 0 }  ][line width=0.75]      (0, 0) circle [x radius= 1.34, y radius= 1.34]   ;
\draw [shift={(147.5,161.5)}, rotate = 0] [color={rgb, 255:red, 0; green, 0; blue, 0 }  ][fill={rgb, 255:red, 0; green, 0; blue, 0 }  ][line width=0.75]      (0, 0) circle [x radius= 1.34, y radius= 1.34]   ;
 
\draw    (119.41,112.02) ;
\draw [shift={(119.41,112.02)}, rotate = 0] [color={rgb, 255:red, 0; green, 0; blue, 0 }  ][fill={rgb, 255:red, 0; green, 0; blue, 0 }  ][line width=0.75]      (0, 0) circle [x radius= 1.34, y radius= 1.34]   ;
\draw [shift={(119.41,112.02)}, rotate = 0] [color={rgb, 255:red, 0; green, 0; blue, 0 }  ][fill={rgb, 255:red, 0; green, 0; blue, 0 }  ][line width=0.75]      (0, 0) circle [x radius= 1.34, y radius= 1.34]   ;
 
\draw  [dash pattern={on 0.84pt off 2.51pt}]  (149.35,115.84) -- (130.02,115.68) ;
 
\draw  [dash pattern={on 0.84pt off 2.51pt}]  (116.16,158.83) -- (135.49,158.83) ;
 
\draw    (119.41,112.02) -- (133.47,136.81) ;
 
\draw    (119.34,111.83) -- (102.5,161.5) ;
 
\draw    (164.37,111.93) -- (147.5,161.5) ;
 
\draw    (164.37,111.93) -- (133.47,136.81) ;
\draw [shift={(133.47,136.81)}, rotate = 141.16] [color={rgb, 255:red, 0; green, 0; blue, 0 }  ][fill={rgb, 255:red, 0; green, 0; blue, 0 }  ][line width=0.75]      (0, 0) circle [x radius= 1.34, y radius= 1.34]   ;
\draw [shift={(164.37,111.93)}, rotate = 141.16] [color={rgb, 255:red, 0; green, 0; blue, 0 }  ][fill={rgb, 255:red, 0; green, 0; blue, 0 }  ][line width=0.75]      (0, 0) circle [x radius= 1.34, y radius= 1.34]   ;
 
\draw    (128.37,111.93) ;
\draw [shift={(128.37,111.93)}, rotate = 0] [color={rgb, 255:red, 0; green, 0; blue, 0 }  ][fill={rgb, 255:red, 0; green, 0; blue, 0 }  ][line width=0.75]      (0, 0) circle [x radius= 1.34, y radius= 1.34]   ;
 
\draw    (155.37,111.93) ;
\draw [shift={(155.37,111.93)}, rotate = 0] [color={rgb, 255:red, 0; green, 0; blue, 0 }  ][fill={rgb, 255:red, 0; green, 0; blue, 0 }  ][line width=0.75]      (0, 0) circle [x radius= 1.34, y radius= 1.34]   ;
 
\draw    (138.5,161.5) ;
\draw [shift={(138.5,161.5)}, rotate = 0] [color={rgb, 255:red, 0; green, 0; blue, 0 }  ][fill={rgb, 255:red, 0; green, 0; blue, 0 }  ][line width=0.75]      (0, 0) circle [x radius= 1.34, y radius= 1.34]   ;
 
\draw    (111.5,161.5) ;
\draw [shift={(111.5,161.5)}, rotate = 0] [color={rgb, 255:red, 0; green, 0; blue, 0 }  ][fill={rgb, 255:red, 0; green, 0; blue, 0 }  ][line width=0.75]      (0, 0) circle [x radius= 1.34, y radius= 1.34]   ;
 
\draw  [dash pattern={on 0.84pt off 2.51pt}]  (133.47,136.81) -- (147.5,161.5) ;
 
\draw  [dash pattern={on 0.84pt off 2.51pt}]  (133.47,136.81) -- (102.5,161.5) ;

\draw (133.44,135.11) node [anchor=north west][inner sep=0.75pt]  [font=\tiny]  {$\times _{1}$};
 
\draw (160.17,131.57) node [anchor=north west][inner sep=0.75pt]  [font=\tiny]  {$\gamma _{1}$};
 
\draw (83,127.73) node [anchor=north west][inner sep=0.75pt]  [font=\tiny]  {$\sigma ( \gamma _{1})$};
 
\draw (132,122.9) node [anchor=north west][inner sep=0.75pt]  [font=\tiny]  {$?$};
 
\draw (127,147.4) node [anchor=north west][inner sep=0.75pt]  [font=\tiny]  {$?$};
 
\draw (117.5,169.9) node [anchor=north west][inner sep=0.75pt]  [font=\normalsize]  {$\Lambda _{grn}$};

\end{tikzpicture}

\end{figure}
\noindent Bold blue and red curves represent different boundaries of $(\mathcal{S}, M)$.

Second, fix the skew-arcs in $\Lambda_{org}$ and $\Lambda_{blu}$. Since any two triangulations of a polygon can be flipped into each other, we can iteratively flip (Type II(1)) the skew-arcs in $\Lambda_{grn}$ into those in $\mathbf{C}_{\mathtt{b}, \times_1}^{\zeta(\Lambda)_1}$, as illustrated below:
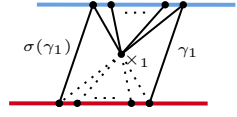
\begin{figure}[H]
    \centering

\tikzset{every picture/.style={line width=0.75pt}}  

\begin{tikzpicture}[x=0.75pt,y=0.75pt,yscale=-1,xscale=1]

\draw [color={rgb, 255:red, 208; green, 2; blue, 27 }  ,draw opacity=1 ][line width=1.5]    (77.5,161.5) -- (176.54,161.59) ;
 
\draw [color={rgb, 255:red, 104; green, 161; blue, 226 }  ,draw opacity=1 ][line width=1.5]    (91.34,111.83) -- (190.37,111.93) ;
 
\draw    (102.5,161.5) ;
\draw [shift={(102.5,161.5)}, rotate = 0] [color={rgb, 255:red, 0; green, 0; blue, 0 }  ][fill={rgb, 255:red, 0; green, 0; blue, 0 }  ][line width=0.75]      (0, 0) circle [x radius= 1.34, y radius= 1.34]   ;
\draw [shift={(102.5,161.5)}, rotate = 0] [color={rgb, 255:red, 0; green, 0; blue, 0 }  ][fill={rgb, 255:red, 0; green, 0; blue, 0 }  ][line width=0.75]      (0, 0) circle [x radius= 1.34, y radius= 1.34]   ;
 
\draw    (147.5,161.5) ;
\draw [shift={(147.5,161.5)}, rotate = 0] [color={rgb, 255:red, 0; green, 0; blue, 0 }  ][fill={rgb, 255:red, 0; green, 0; blue, 0 }  ][line width=0.75]      (0, 0) circle [x radius= 1.34, y radius= 1.34]   ;
\draw [shift={(147.5,161.5)}, rotate = 0] [color={rgb, 255:red, 0; green, 0; blue, 0 }  ][fill={rgb, 255:red, 0; green, 0; blue, 0 }  ][line width=0.75]      (0, 0) circle [x radius= 1.34, y radius= 1.34]   ;
 
\draw    (119.41,112.02) ;
\draw [shift={(119.41,112.02)}, rotate = 0] [color={rgb, 255:red, 0; green, 0; blue, 0 }  ][fill={rgb, 255:red, 0; green, 0; blue, 0 }  ][line width=0.75]      (0, 0) circle [x radius= 1.34, y radius= 1.34]   ;
\draw [shift={(119.41,112.02)}, rotate = 0] [color={rgb, 255:red, 0; green, 0; blue, 0 }  ][fill={rgb, 255:red, 0; green, 0; blue, 0 }  ][line width=0.75]      (0, 0) circle [x radius= 1.34, y radius= 1.34]   ;
 
\draw  [dash pattern={on 0.84pt off 2.51pt}]  (145.35,115.84) -- (133.52,115.96) ;
 
\draw  [dash pattern={on 0.84pt off 2.51pt}]  (120.16,158.83) -- (131.49,158.83) ;
 
\draw    (119.41,112.02) -- (133.47,136.81) ;
 
\draw    (119.34,111.83) -- (102.5,161.5) ;
 
\draw    (164.37,111.93) -- (147.5,161.5) ;
 
\draw    (164.37,111.93) -- (133.47,136.81) ;
\draw [shift={(133.47,136.81)}, rotate = 141.16] [color={rgb, 255:red, 0; green, 0; blue, 0 }  ][fill={rgb, 255:red, 0; green, 0; blue, 0 }  ][line width=0.75]      (0, 0) circle [x radius= 1.34, y radius= 1.34]   ;
\draw [shift={(164.37,111.93)}, rotate = 141.16] [color={rgb, 255:red, 0; green, 0; blue, 0 }  ][fill={rgb, 255:red, 0; green, 0; blue, 0 }  ][line width=0.75]      (0, 0) circle [x radius= 1.34, y radius= 1.34]   ;
 
\draw    (128.37,111.93) ;
\draw [shift={(128.37,111.93)}, rotate = 0] [color={rgb, 255:red, 0; green, 0; blue, 0 }  ][fill={rgb, 255:red, 0; green, 0; blue, 0 }  ][line width=0.75]      (0, 0) circle [x radius= 1.34, y radius= 1.34]   ;
 
\draw    (155.37,111.93) ;
\draw [shift={(155.37,111.93)}, rotate = 0] [color={rgb, 255:red, 0; green, 0; blue, 0 }  ][fill={rgb, 255:red, 0; green, 0; blue, 0 }  ][line width=0.75]      (0, 0) circle [x radius= 1.34, y radius= 1.34]   ;
 
\draw    (138.5,161.5) ;
\draw [shift={(138.5,161.5)}, rotate = 0] [color={rgb, 255:red, 0; green, 0; blue, 0 }  ][fill={rgb, 255:red, 0; green, 0; blue, 0 }  ][line width=0.75]      (0, 0) circle [x radius= 1.34, y radius= 1.34]   ;
 
\draw    (111.5,161.5) ;
\draw [shift={(111.5,161.5)}, rotate = 0] [color={rgb, 255:red, 0; green, 0; blue, 0 }  ][fill={rgb, 255:red, 0; green, 0; blue, 0 }  ][line width=0.75]      (0, 0) circle [x radius= 1.34, y radius= 1.34]   ;
 
\draw  [dash pattern={on 0.84pt off 2.51pt}]  (133.47,136.81) -- (147.5,161.5) ;
 
\draw  [dash pattern={on 0.84pt off 2.51pt}]  (133.47,136.81) -- (102.5,161.5) ;
 
\draw    (128.37,111.93) -- (133.47,136.81) ;
 
\draw    (155.37,111.93) -- (133.47,136.81) ;
 
\draw  [dash pattern={on 0.84pt off 2.51pt}]  (133.47,136.81) -- (138.57,161.68) ;
 
\draw  [dash pattern={on 0.84pt off 2.51pt}]  (133.4,136.62) -- (111.5,161.5) ;

\draw (133.44,135.11) node [anchor=north west][inner sep=0.75pt]  [font=\tiny]  {$\times _{1}$};
 
\draw (160.17,131.57) node [anchor=north west][inner sep=0.75pt]  [font=\tiny]  {$\gamma _{1}$};
 
\draw (83,127.73) node [anchor=north west][inner sep=0.75pt]  [font=\tiny]  {$\sigma ( \gamma _{1})$};
\end{tikzpicture}

\end{figure}
Denote the resulting pseudo-triangulation by $\Lambda^\prime$. Next, alternate flips of the leftmost and rightmost skew-arcs in $\Lambda_{grn}^\prime$, yielding a pseudo-triangulation $\Lambda^{\prime\prime}$ such that $\zeta(\Lambda^{\prime\prime})_1 = \pm(0)$. For example, when $|\Lambda_{grn}^\prime| = 4$, the sequence of flips proceeds as follows:
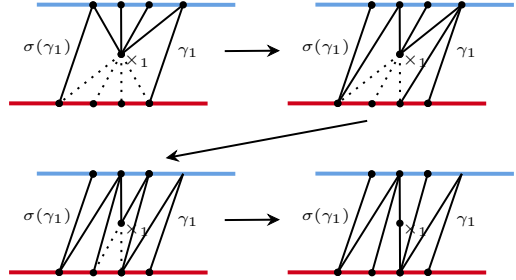
\begin{figure}[H]
    \centering

\tikzset{every picture/.style={line width=0.75pt}}  

\begin{tikzpicture}[x=0.75pt,y=0.75pt,yscale=-1,xscale=1]

\draw [color={rgb, 255:red, 208; green, 2; blue, 27 }  ,draw opacity=1 ][line width=1.5]    (77.5,161.5) -- (176.54,161.59) ;
 
\draw [color={rgb, 255:red, 104; green, 161; blue, 226 }  ,draw opacity=1 ][line width=1.5]    (91.34,111.83) -- (190.37,111.93) ;
 
\draw    (102.5,161.5) ;
\draw [shift={(102.5,161.5)}, rotate = 0] [color={rgb, 255:red, 0; green, 0; blue, 0 }  ][fill={rgb, 255:red, 0; green, 0; blue, 0 }  ][line width=0.75]      (0, 0) circle [x radius= 1.34, y radius= 1.34]   ;
\draw [shift={(102.5,161.5)}, rotate = 0] [color={rgb, 255:red, 0; green, 0; blue, 0 }  ][fill={rgb, 255:red, 0; green, 0; blue, 0 }  ][line width=0.75]      (0, 0) circle [x radius= 1.34, y radius= 1.34]   ;
 
\draw    (147.5,161.5) ;
\draw [shift={(147.5,161.5)}, rotate = 0] [color={rgb, 255:red, 0; green, 0; blue, 0 }  ][fill={rgb, 255:red, 0; green, 0; blue, 0 }  ][line width=0.75]      (0, 0) circle [x radius= 1.34, y radius= 1.34]   ;
\draw [shift={(147.5,161.5)}, rotate = 0] [color={rgb, 255:red, 0; green, 0; blue, 0 }  ][fill={rgb, 255:red, 0; green, 0; blue, 0 }  ][line width=0.75]      (0, 0) circle [x radius= 1.34, y radius= 1.34]   ;
 
\draw    (119.41,112.02) ;
\draw [shift={(119.41,112.02)}, rotate = 0] [color={rgb, 255:red, 0; green, 0; blue, 0 }  ][fill={rgb, 255:red, 0; green, 0; blue, 0 }  ][line width=0.75]      (0, 0) circle [x radius= 1.34, y radius= 1.34]   ;
\draw [shift={(119.41,112.02)}, rotate = 0] [color={rgb, 255:red, 0; green, 0; blue, 0 }  ][fill={rgb, 255:red, 0; green, 0; blue, 0 }  ][line width=0.75]      (0, 0) circle [x radius= 1.34, y radius= 1.34]   ;
 
\draw    (119.41,112.02) -- (133.47,136.81) ;
 
\draw    (119.34,111.83) -- (102.5,161.5) ;
 
\draw    (164.37,111.93) -- (147.5,161.5) ;
 
\draw    (164.37,111.93) -- (133.47,136.81) ;
\draw [shift={(133.47,136.81)}, rotate = 141.16] [color={rgb, 255:red, 0; green, 0; blue, 0 }  ][fill={rgb, 255:red, 0; green, 0; blue, 0 }  ][line width=0.75]      (0, 0) circle [x radius= 1.34, y radius= 1.34]   ;
\draw [shift={(164.37,111.93)}, rotate = 141.16] [color={rgb, 255:red, 0; green, 0; blue, 0 }  ][fill={rgb, 255:red, 0; green, 0; blue, 0 }  ][line width=0.75]      (0, 0) circle [x radius= 1.34, y radius= 1.34]   ;
 
\draw    (133.37,111.93) ;
\draw [shift={(133.37,111.93)}, rotate = 0] [color={rgb, 255:red, 0; green, 0; blue, 0 }  ][fill={rgb, 255:red, 0; green, 0; blue, 0 }  ][line width=0.75]      (0, 0) circle [x radius= 1.34, y radius= 1.34]   ;
 
\draw    (147.37,111.93) ;
\draw [shift={(147.37,111.93)}, rotate = 0] [color={rgb, 255:red, 0; green, 0; blue, 0 }  ][fill={rgb, 255:red, 0; green, 0; blue, 0 }  ][line width=0.75]      (0, 0) circle [x radius= 1.34, y radius= 1.34]   ;
 
\draw  [dash pattern={on 0.84pt off 2.51pt}]  (133.47,136.81) -- (147.5,161.5) ;
 
\draw  [dash pattern={on 0.84pt off 2.51pt}]  (133.47,136.81) -- (102.5,161.5) ;
 
\draw    (133.37,111.93) -- (133.47,136.81) ;
 
\draw    (147.37,111.93) -- (133.47,136.81) ;
 
\draw  [dash pattern={on 0.84pt off 2.51pt}]  (133.47,136.81) -- (133.57,161.68) ;
\draw [shift={(133.57,161.68)}, rotate = 89.77] [color={rgb, 255:red, 0; green, 0; blue, 0 }  ][fill={rgb, 255:red, 0; green, 0; blue, 0 }  ][line width=0.75]      (0, 0) circle [x radius= 1.34, y radius= 1.34]   ;
\draw [shift={(133.47,136.81)}, rotate = 89.77] [color={rgb, 255:red, 0; green, 0; blue, 0 }  ][fill={rgb, 255:red, 0; green, 0; blue, 0 }  ][line width=0.75]      (0, 0) circle [x radius= 1.34, y radius= 1.34]   ;
 
\draw  [dash pattern={on 0.84pt off 2.51pt}]  (133.47,136.81) -- (119.57,161.68) ;
\draw [shift={(119.57,161.68)}, rotate = 119.19] [color={rgb, 255:red, 0; green, 0; blue, 0 }  ][fill={rgb, 255:red, 0; green, 0; blue, 0 }  ][line width=0.75]      (0, 0) circle [x radius= 1.34, y radius= 1.34]   ;
 
\draw [color={rgb, 255:red, 208; green, 2; blue, 27 }  ,draw opacity=1 ][line width=1.5]    (216.5,161.5) -- (315.54,161.59) ;
 
\draw [color={rgb, 255:red, 104; green, 161; blue, 226 }  ,draw opacity=1 ][line width=1.5]    (230.34,111.83) -- (329.37,111.93) ;
 
\draw    (241.5,161.5) ;
\draw [shift={(241.5,161.5)}, rotate = 0] [color={rgb, 255:red, 0; green, 0; blue, 0 }  ][fill={rgb, 255:red, 0; green, 0; blue, 0 }  ][line width=0.75]      (0, 0) circle [x radius= 1.34, y radius= 1.34]   ;
\draw [shift={(241.5,161.5)}, rotate = 0] [color={rgb, 255:red, 0; green, 0; blue, 0 }  ][fill={rgb, 255:red, 0; green, 0; blue, 0 }  ][line width=0.75]      (0, 0) circle [x radius= 1.34, y radius= 1.34]   ;
 
\draw    (286.5,161.5) ;
\draw [shift={(286.5,161.5)}, rotate = 0] [color={rgb, 255:red, 0; green, 0; blue, 0 }  ][fill={rgb, 255:red, 0; green, 0; blue, 0 }  ][line width=0.75]      (0, 0) circle [x radius= 1.34, y radius= 1.34]   ;
\draw [shift={(286.5,161.5)}, rotate = 0] [color={rgb, 255:red, 0; green, 0; blue, 0 }  ][fill={rgb, 255:red, 0; green, 0; blue, 0 }  ][line width=0.75]      (0, 0) circle [x radius= 1.34, y radius= 1.34]   ;
 
\draw    (258.41,112.02) ;
\draw [shift={(258.41,112.02)}, rotate = 0] [color={rgb, 255:red, 0; green, 0; blue, 0 }  ][fill={rgb, 255:red, 0; green, 0; blue, 0 }  ][line width=0.75]      (0, 0) circle [x radius= 1.34, y radius= 1.34]   ;
\draw [shift={(258.41,112.02)}, rotate = 0] [color={rgb, 255:red, 0; green, 0; blue, 0 }  ][fill={rgb, 255:red, 0; green, 0; blue, 0 }  ][line width=0.75]      (0, 0) circle [x radius= 1.34, y radius= 1.34]   ;
 
\draw    (272.37,111.93) -- (241.5,161.5) ;
 
\draw    (258.34,111.83) -- (241.5,161.5) ;
 
\draw    (303.37,111.93) -- (286.5,161.5) ;
 
\draw    (303.37,111.93) -- (272.47,136.81) ;
\draw [shift={(272.47,136.81)}, rotate = 141.16] [color={rgb, 255:red, 0; green, 0; blue, 0 }  ][fill={rgb, 255:red, 0; green, 0; blue, 0 }  ][line width=0.75]      (0, 0) circle [x radius= 1.34, y radius= 1.34]   ;
\draw [shift={(303.37,111.93)}, rotate = 141.16] [color={rgb, 255:red, 0; green, 0; blue, 0 }  ][fill={rgb, 255:red, 0; green, 0; blue, 0 }  ][line width=0.75]      (0, 0) circle [x radius= 1.34, y radius= 1.34]   ;
 
\draw    (272.37,111.93) ;
\draw [shift={(272.37,111.93)}, rotate = 0] [color={rgb, 255:red, 0; green, 0; blue, 0 }  ][fill={rgb, 255:red, 0; green, 0; blue, 0 }  ][line width=0.75]      (0, 0) circle [x radius= 1.34, y radius= 1.34]   ;
 
\draw    (286.37,111.93) ;
\draw [shift={(286.37,111.93)}, rotate = 0] [color={rgb, 255:red, 0; green, 0; blue, 0 }  ][fill={rgb, 255:red, 0; green, 0; blue, 0 }  ][line width=0.75]      (0, 0) circle [x radius= 1.34, y radius= 1.34]   ;
 
\draw    (303.37,111.93) -- (272.57,161.68) ;
 
\draw  [dash pattern={on 0.84pt off 2.51pt}]  (272.47,136.81) -- (241.5,161.5) ;
 
\draw    (272.37,111.93) -- (272.47,136.81) ;
 
\draw    (286.37,111.93) -- (272.47,136.81) ;
 
\draw  [dash pattern={on 0.84pt off 2.51pt}]  (272.47,136.81) -- (272.57,161.68) ;
\draw [shift={(272.57,161.68)}, rotate = 89.77] [color={rgb, 255:red, 0; green, 0; blue, 0 }  ][fill={rgb, 255:red, 0; green, 0; blue, 0 }  ][line width=0.75]      (0, 0) circle [x radius= 1.34, y radius= 1.34]   ;
\draw [shift={(272.47,136.81)}, rotate = 89.77] [color={rgb, 255:red, 0; green, 0; blue, 0 }  ][fill={rgb, 255:red, 0; green, 0; blue, 0 }  ][line width=0.75]      (0, 0) circle [x radius= 1.34, y radius= 1.34]   ;
 
\draw  [dash pattern={on 0.84pt off 2.51pt}]  (272.47,136.81) -- (258.57,161.68) ;
\draw [shift={(258.57,161.68)}, rotate = 119.19] [color={rgb, 255:red, 0; green, 0; blue, 0 }  ][fill={rgb, 255:red, 0; green, 0; blue, 0 }  ][line width=0.75]      (0, 0) circle [x radius= 1.34, y radius= 1.34]   ;
 
\draw    (185.1,135.2) -- (209.1,135.02) ;
\draw [shift={(212.1,135)}, rotate = 179.58] [fill={rgb, 255:red, 0; green, 0; blue, 0 }  ][line width=0.08]  [draw opacity=0] (5.36,-2.57) -- (0,0) -- (5.36,2.57) -- (3.56,0) -- cycle    ;
 
\draw [color={rgb, 255:red, 208; green, 2; blue, 27 }  ,draw opacity=1 ][line width=1.5]    (77.5,246.5) -- (176.54,246.59) ;
 
\draw [color={rgb, 255:red, 104; green, 161; blue, 226 }  ,draw opacity=1 ][line width=1.5]    (91.34,196.83) -- (190.37,196.93) ;
 
\draw    (102.5,246.5) ;
\draw [shift={(102.5,246.5)}, rotate = 0] [color={rgb, 255:red, 0; green, 0; blue, 0 }  ][fill={rgb, 255:red, 0; green, 0; blue, 0 }  ][line width=0.75]      (0, 0) circle [x radius= 1.34, y radius= 1.34]   ;
\draw [shift={(102.5,246.5)}, rotate = 0] [color={rgb, 255:red, 0; green, 0; blue, 0 }  ][fill={rgb, 255:red, 0; green, 0; blue, 0 }  ][line width=0.75]      (0, 0) circle [x radius= 1.34, y radius= 1.34]   ;
 
\draw    (147.5,246.5) ;
\draw [shift={(147.5,246.5)}, rotate = 0] [color={rgb, 255:red, 0; green, 0; blue, 0 }  ][fill={rgb, 255:red, 0; green, 0; blue, 0 }  ][line width=0.75]      (0, 0) circle [x radius= 1.34, y radius= 1.34]   ;
\draw [shift={(147.5,246.5)}, rotate = 0] [color={rgb, 255:red, 0; green, 0; blue, 0 }  ][fill={rgb, 255:red, 0; green, 0; blue, 0 }  ][line width=0.75]      (0, 0) circle [x radius= 1.34, y radius= 1.34]   ;
 
\draw    (119.41,197.02) ;
\draw [shift={(119.41,197.02)}, rotate = 0] [color={rgb, 255:red, 0; green, 0; blue, 0 }  ][fill={rgb, 255:red, 0; green, 0; blue, 0 }  ][line width=0.75]      (0, 0) circle [x radius= 1.34, y radius= 1.34]   ;
\draw [shift={(119.41,197.02)}, rotate = 0] [color={rgb, 255:red, 0; green, 0; blue, 0 }  ][fill={rgb, 255:red, 0; green, 0; blue, 0 }  ][line width=0.75]      (0, 0) circle [x radius= 1.34, y radius= 1.34]   ;
 
\draw    (133.39,196.97) -- (102.5,246.5) ;
 
\draw    (119.34,196.83) -- (102.5,246.5) ;
 
\draw    (164.37,196.93) -- (147.5,246.5) ;
 
\draw    (147.37,196.93) -- (133.57,246.68) ;
\draw [shift={(133.57,246.68)}, rotate = 105.5] [color={rgb, 255:red, 0; green, 0; blue, 0 }  ][fill={rgb, 255:red, 0; green, 0; blue, 0 }  ][line width=0.75]      (0, 0) circle [x radius= 1.34, y radius= 1.34]   ;
\draw [shift={(147.37,196.93)}, rotate = 105.5] [color={rgb, 255:red, 0; green, 0; blue, 0 }  ][fill={rgb, 255:red, 0; green, 0; blue, 0 }  ][line width=0.75]      (0, 0) circle [x radius= 1.34, y radius= 1.34]   ;
 
\draw    (133.37,196.93) ;
\draw [shift={(133.37,196.93)}, rotate = 0] [color={rgb, 255:red, 0; green, 0; blue, 0 }  ][fill={rgb, 255:red, 0; green, 0; blue, 0 }  ][line width=0.75]      (0, 0) circle [x radius= 1.34, y radius= 1.34]   ;
 
\draw    (147.37,196.93) ;
\draw [shift={(147.37,196.93)}, rotate = 0] [color={rgb, 255:red, 0; green, 0; blue, 0 }  ][fill={rgb, 255:red, 0; green, 0; blue, 0 }  ][line width=0.75]      (0, 0) circle [x radius= 1.34, y radius= 1.34]   ;
 
\draw    (164.37,196.93) -- (133.57,246.68) ;
 
\draw    (133.39,196.97) -- (119.57,246.68) ;
 
\draw    (133.37,196.93) -- (133.47,221.81) ;
 
\draw    (147.37,196.93) -- (133.47,221.81) ;
 
\draw  [dash pattern={on 0.84pt off 2.51pt}]  (133.47,221.81) -- (133.57,246.68) ;
\draw [shift={(133.57,246.68)}, rotate = 89.77] [color={rgb, 255:red, 0; green, 0; blue, 0 }  ][fill={rgb, 255:red, 0; green, 0; blue, 0 }  ][line width=0.75]      (0, 0) circle [x radius= 1.34, y radius= 1.34]   ;
\draw [shift={(133.47,221.81)}, rotate = 89.77] [color={rgb, 255:red, 0; green, 0; blue, 0 }  ][fill={rgb, 255:red, 0; green, 0; blue, 0 }  ][line width=0.75]      (0, 0) circle [x radius= 1.34, y radius= 1.34]   ;
 
\draw  [dash pattern={on 0.84pt off 2.51pt}]  (133.47,221.81) -- (119.57,246.68) ;
\draw [shift={(119.57,246.68)}, rotate = 119.19] [color={rgb, 255:red, 0; green, 0; blue, 0 }  ][fill={rgb, 255:red, 0; green, 0; blue, 0 }  ][line width=0.75]      (0, 0) circle [x radius= 1.34, y radius= 1.34]   ;
 
\draw    (185.1,220.2) -- (209.1,220.02) ;
\draw [shift={(212.1,220)}, rotate = 179.58] [fill={rgb, 255:red, 0; green, 0; blue, 0 }  ][line width=0.08]  [draw opacity=0] (5.36,-2.57) -- (0,0) -- (5.36,2.57) -- (3.56,0) -- cycle    ;
 
\draw    (256.16,170.28) -- (159.11,188.37) ;
\draw [shift={(156.16,188.92)}, rotate = 349.44] [fill={rgb, 255:red, 0; green, 0; blue, 0 }  ][line width=0.08]  [draw opacity=0] (5.36,-2.57) -- (0,0) -- (5.36,2.57) -- (3.56,0) -- cycle    ;
 
\draw [color={rgb, 255:red, 208; green, 2; blue, 27 }  ,draw opacity=1 ][line width=1.5]    (216.5,246.5) -- (315.54,246.59) ;
 
\draw [color={rgb, 255:red, 104; green, 161; blue, 226 }  ,draw opacity=1 ][line width=1.5]    (230.34,196.83) -- (329.37,196.93) ;
 
\draw    (241.5,246.5) ;
\draw [shift={(241.5,246.5)}, rotate = 0] [color={rgb, 255:red, 0; green, 0; blue, 0 }  ][fill={rgb, 255:red, 0; green, 0; blue, 0 }  ][line width=0.75]      (0, 0) circle [x radius= 1.34, y radius= 1.34]   ;
\draw [shift={(241.5,246.5)}, rotate = 0] [color={rgb, 255:red, 0; green, 0; blue, 0 }  ][fill={rgb, 255:red, 0; green, 0; blue, 0 }  ][line width=0.75]      (0, 0) circle [x radius= 1.34, y radius= 1.34]   ;
 
\draw    (286.5,246.5) ;
\draw [shift={(286.5,246.5)}, rotate = 0] [color={rgb, 255:red, 0; green, 0; blue, 0 }  ][fill={rgb, 255:red, 0; green, 0; blue, 0 }  ][line width=0.75]      (0, 0) circle [x radius= 1.34, y radius= 1.34]   ;
\draw [shift={(286.5,246.5)}, rotate = 0] [color={rgb, 255:red, 0; green, 0; blue, 0 }  ][fill={rgb, 255:red, 0; green, 0; blue, 0 }  ][line width=0.75]      (0, 0) circle [x radius= 1.34, y radius= 1.34]   ;
 
\draw    (258.41,197.02) ;
\draw [shift={(258.41,197.02)}, rotate = 0] [color={rgb, 255:red, 0; green, 0; blue, 0 }  ][fill={rgb, 255:red, 0; green, 0; blue, 0 }  ][line width=0.75]      (0, 0) circle [x radius= 1.34, y radius= 1.34]   ;
\draw [shift={(258.41,197.02)}, rotate = 0] [color={rgb, 255:red, 0; green, 0; blue, 0 }  ][fill={rgb, 255:red, 0; green, 0; blue, 0 }  ][line width=0.75]      (0, 0) circle [x radius= 1.34, y radius= 1.34]   ;
 
\draw    (272.39,196.97) -- (241.5,246.5) ;
 
\draw    (258.34,196.83) -- (241.5,246.5) ;
 
\draw    (303.37,196.93) -- (286.5,246.5) ;
 
\draw    (286.37,196.93) -- (272.57,246.68) ;
\draw [shift={(272.57,246.68)}, rotate = 105.5] [color={rgb, 255:red, 0; green, 0; blue, 0 }  ][fill={rgb, 255:red, 0; green, 0; blue, 0 }  ][line width=0.75]      (0, 0) circle [x radius= 1.34, y radius= 1.34]   ;
\draw [shift={(286.37,196.93)}, rotate = 105.5] [color={rgb, 255:red, 0; green, 0; blue, 0 }  ][fill={rgb, 255:red, 0; green, 0; blue, 0 }  ][line width=0.75]      (0, 0) circle [x radius= 1.34, y radius= 1.34]   ;
 
\draw    (272.37,196.93) ;
\draw [shift={(272.37,196.93)}, rotate = 0] [color={rgb, 255:red, 0; green, 0; blue, 0 }  ][fill={rgb, 255:red, 0; green, 0; blue, 0 }  ][line width=0.75]      (0, 0) circle [x radius= 1.34, y radius= 1.34]   ;
 
\draw    (286.37,196.93) ;
\draw [shift={(286.37,196.93)}, rotate = 0] [color={rgb, 255:red, 0; green, 0; blue, 0 }  ][fill={rgb, 255:red, 0; green, 0; blue, 0 }  ][line width=0.75]      (0, 0) circle [x radius= 1.34, y radius= 1.34]   ;
 
\draw    (303.37,196.93) -- (272.57,246.68) ;
 
\draw    (272.39,196.97) -- (258.57,246.68) ;
\draw [shift={(258.57,246.68)}, rotate = 105.53] [color={rgb, 255:red, 0; green, 0; blue, 0 }  ][fill={rgb, 255:red, 0; green, 0; blue, 0 }  ][line width=0.75]      (0, 0) circle [x radius= 1.34, y radius= 1.34]   ;
 
\draw    (272.37,196.93) -- (272.47,221.81) ;
 
\draw    (272.47,221.81) -- (272.57,246.68) ;
\draw [shift={(272.57,246.68)}, rotate = 89.77] [color={rgb, 255:red, 0; green, 0; blue, 0 }  ][fill={rgb, 255:red, 0; green, 0; blue, 0 }  ][line width=0.75]      (0, 0) circle [x radius= 1.34, y radius= 1.34]   ;
\draw [shift={(272.47,221.81)}, rotate = 89.77] [color={rgb, 255:red, 0; green, 0; blue, 0 }  ][fill={rgb, 255:red, 0; green, 0; blue, 0 }  ][line width=0.75]      (0, 0) circle [x radius= 1.34, y radius= 1.34]   ;

\draw (133.44,135.11) node [anchor=north west][inner sep=0.75pt]  [font=\tiny]  {$\times _{1}$};
 
\draw (160.17,131.57) node [anchor=north west][inner sep=0.75pt]  [font=\tiny]  {$\gamma _{1}$};
 
\draw (83,127.73) node [anchor=north west][inner sep=0.75pt]  [font=\tiny]  {$\sigma ( \gamma _{1})$};
 
\draw (272.44,135.11) node [anchor=north west][inner sep=0.75pt]  [font=\tiny]  {$\times _{1}$};
 
\draw (299.17,131.57) node [anchor=north west][inner sep=0.75pt]  [font=\tiny]  {$\gamma _{1}$};
 
\draw (222,127.73) node [anchor=north west][inner sep=0.75pt]  [font=\tiny]  {$\sigma ( \gamma _{1})$};
 
\draw (133.44,220.11) node [anchor=north west][inner sep=0.75pt]  [font=\tiny]  {$\times _{1}$};
 
\draw (160.17,216.57) node [anchor=north west][inner sep=0.75pt]  [font=\tiny]  {$\gamma _{1}$};
 
\draw (83,212.73) node [anchor=north west][inner sep=0.75pt]  [font=\tiny]  {$\sigma ( \gamma _{1})$};
 
\draw (272.44,220.11) node [anchor=north west][inner sep=0.75pt]  [font=\tiny]  {$\times _{1}$};
 
\draw (299.17,216.57) node [anchor=north west][inner sep=0.75pt]  [font=\tiny]  {$\gamma _{1}$};
 
\draw (222,212.73) node [anchor=north west][inner sep=0.75pt]  [font=\tiny]  {$\sigma ( \gamma _{1})$};

\end{tikzpicture}

\end{figure}

Finally, apply a similar flipping process to $\Lambda_{blu}^{\prime\prime}$ to obtain a pseudo-triangulation $\Lambda^{\prime\prime\prime}$ where $\zeta(\Lambda^{\prime\prime\prime}) = (\pm(0), \pm(0))$ and $\Lambda^{\prime\prime\prime}_{org} \subset \mathbf{C}_{\mathtt{b}}^\sigma$. Since $\Lambda^{\prime\prime\prime}$ is an FV-pseudo-triangulation, the proof is complete.
\end{proof}
Consequently, to prove $\mathcal{G}(\mathcal{T}_{\mathbb{X}})$ is connected, it suffices to show that any pseudo-triangulation $\Lambda$ can be flipped into a pseudo-triangulation that includes skew-arcs in $\mathbf{C}_{\mathtt{b}}^\sigma$. To demonstrate this, we analyze several cases of $\Lambda$, which we address through the following lemmas:

\begin{lem}\label{lem:1}
    If $|\Lambda_*| = 2$, then $\Lambda$ can be transformed into a pseudo-triangulation that includes skew-arcs in $\mathbf{C}_{\mathtt{b}}^\sigma$.
\end{lem}
\begin{proof}
We only prove the case $\Lambda_* = \{(+, +), (-, +)\}$; other cases follow similarly. Since any two triangulations of a polygon can be flipped into each other, we may assume that $ \Lambda_{\mathtt{p}}^\sigma=\emptyset$. 

Sequentially flipping $(+, +)$ and $(-, +)$, followed by flipping the orange (or blue) skew-arc in $\Lambda^\prime$, yields a new pseudo-triangulation that includes skew-arcs in $\mathbf{C}_{\mathtt{b}}^\sigma$. See the diagram below for illustration.

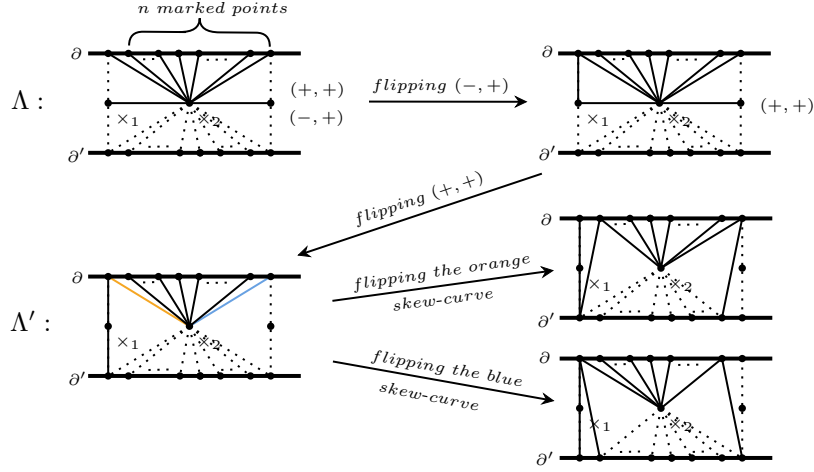
\begin{figure}[H]
    \centering

\tikzset{every picture/.style={line width=0.75pt}}  


\end{figure}
\end{proof}

\begin{lem}\label{lem:2}
    If $|\Lambda_*|=1$, then $\Lambda$ can be flipped into a pseudo-triangulation that includes skew-arcs in $\mathbf{C}_{\mathtt{b}}^\sigma$.
\end{lem}

\begin{proof}
 We  proof the case $\Lambda = \{(+, +)\}$; other cases follow similarly. Since any two triangulations of a polygon can be flipped into each other, we may assume that $ \Lambda_{\mathtt{p}}^\sigma=\emptyset$. We have two cases:
        
         (1) $|\Lambda_{\mathtt{b},\times_i}^+| = 1$ for some $i$. Assume without loss of generality that $i=1$.  By flipping the only skew-arc in $\Lambda_{\mathtt{b},\times_1}^+$, we obtain a pseudo-triangulation $\Lambda'$ with $|\Lambda'_*| = 2$, as shown in the figure below:
  
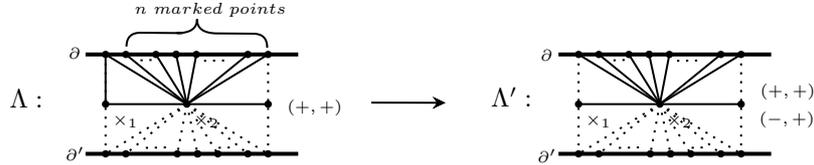
\begin{figure}[H]
           \centering
\tikzset{every picture/.style={line width=0.75pt}}  

       \end{figure}
\noindent Then, by Lemma \ref{lem:1}, the statement holds.
    
    (2) $|\Lambda_{\mathtt{b},\times_i}^+| \neq 1$ for all $i$. We can flip $(+, +)$ to obtain a skew-arc in $\mathbf{C}_{\mathtt{b}}^\sigma$, as shown in the figure below:
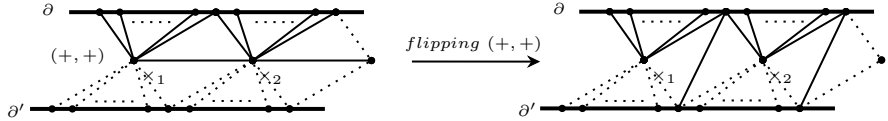
\begin{figure}[H]
    \centering

\tikzset{every picture/.style={line width=0.75pt}}  

%
\end{figure}
\end{proof}

\begin{lem}\label{lem:3}
    If $\Lambda_{\mathtt{b}}^\sigma = \Lambda_* = \emptyset$, then $\Lambda$ can be flipped into a pseudo-triangulation that includes skew-arcs in $\mathbf{C}_{\mathtt{b}}^\sigma$.
\end{lem}

\begin{proof}
    We claim that $\zeta(\Lambda)_i = \pm(0)$ holds for at least one $i$. Suppose otherwise; then $\Lambda \cup \{(\zeta(\Lambda)_1, \zeta(\Lambda)_2)\}$ forms a collection of distinct, pairwise compatible skew-arcs, contradicting the maximality of $\Lambda$. Thus, we have three cases: 
    
        (1)   $\zeta(\Lambda)_1 = \zeta(\Lambda)_2 = \pm(0)$. In this case, regardless of the skew-arcs in the yellow region, we can always flip the orange skew-arc in $\Lambda_{\mathtt{p}}^\sigma$, as shown in the figure below:
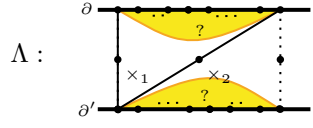
\begin{figure}[H]
    \centering

\tikzset{every picture/.style={line width=0.75pt}}  


\end{figure}
\noindent to obtain a skew-arc in $\mathbf{C}_{\mathtt{b}}^\sigma$.
        
        (2) $\zeta(\Lambda)_1 = \pm(0)$ and $\zeta(\Lambda)_2 \neq \pm(0)$. Assume without loss of generality that $\zeta(\Lambda)_2 =+$. In this case, regardless of the skew-arcs in the yellow region,  we can always flip one of the skew-arcs in $\Lambda_{\mathtt{b},\times_1}$, as shown in the figure below:
  
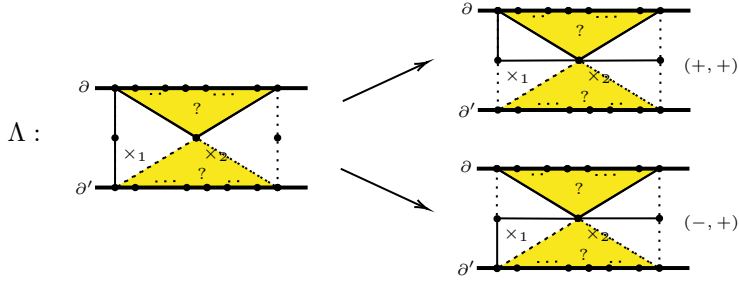
\begin{figure}[H]
      \centering

\tikzset{every picture/.style={line width=0.75pt}}  

%

  \end{figure}
  \noindent to obtain a skew-arc in $\mathbf{C}_*$. Then, by Lemma \ref{lem:2}, the statement holds. 
  
   (3) $\zeta(\Lambda)_2 = \pm(0)$ and $\zeta(\Lambda)_1 \neq \pm(0)$. This case is similar to (2).
\end{proof}

\begin{thm}{}{}\label{connected}
	The tilting graph $\mathcal{G}(\mathcal{T}_{\mathbb{X}})$ is connected.
\end{thm}
\begin{proof}
   By Proposition \ref{filp to FV} and Lemmas \ref{lem:1}-\ref{lem:3}, any two pseudo-triangulations on $(\mathcal{S},M)$ can be flipped into each other. Combined with Corollary \ref{mutation and flip} and the bijection $\widehat{\phi}$, the result follows.
\end{proof}

\appendix
\section{The equivariant relationship between ${\rm coh}\mbox{-}\mathbb{X}$ and ${\rm coh}\mbox{-}\mathbb{Y}$}\label{equivariant}
Let $\mathbb{X}$ and $\mathbb{Y}$ be the weighted projective lines of types $(2,2,n)$ and $(n,n)$, respectively.  We refer to \cite{Chen2020DualActions, Chen2017Equivariantization} for background on $\mathbf{k}$-linear finite group actions on a category and equivariantization.  Let $G$ be a finite group and $\mathcal{C}$ a category with a $\mathbf{k}$-linear $G$-action. We denote by $\mathcal{C}^G$ the category of $G$-equivariant objects in $\mathcal{C}$. Two natural functors play an important role in equivariant theory: the induction functor $\operatorname{Ind} : \mathcal{C} \rightarrow \mathcal{C}^G$ and the forgetful functor $F : \mathcal{C}^G \rightarrow \mathcal{C}$.

In this appendix, we use the equivariantization to relate weighted projective lines $\mathbb{X}$ and $\mathbb{Y}$. Our discussion follows a similar structure to that in \cite[Section 8]{MR4245628}, where the equivariant relationships between different weighted projective lines of tubular type are examined. For the reader's convenience, we present the details here as well.

According to \cite{MR3644095}, there exists finite subgroup $G$ of the automorphism group $\operatorname{Aut}(\mathbb{Y})$ such that $\mathbb{Y}/G \cong \mathbb{X}$. Specifically, $G$ is the cyclic group of order $2$, generated by $\sigma_{1,2}$, which exchanges the two weighted points of $\mathbb{Y}$. Moreover, the group $G$ can be expressed as the subgroup of $\operatorname{Aut}({\rm coh}\mbox{-}\mathbb{Y})$, the
group of isomorphism classes of self-equivalences of ${\rm coh}\mbox{-}\mathbb{Y}$, fixing the structure sheaf $\mathcal{O}_{\mathbb{Y}}$, compare \cite{Lenzing2000AutomorphismGroup}. 
Thus, there is natural $G$-action on ${\rm coh}\mbox{-}\mathbb{Y}$. By equivariantization, we get  the category $({\rm coh}\mbox{-}\mathbb{Y})^G$ of $G$-equivariant objects. 

\begin{rem}\label{isomor}
The isomorphism $\mathbb{Y}/G \cong \mathbb{X}$ is described explicitly as follows, though it is implicit in \cite{MR3644095}:   Consider the orbifold quotient $\mathbb{Y}/G = (\mathbb{Y}/G, \overline{\omega})$ of $\mathbb{Y}$ by $G$. Denote the orbits in $\mathbb{Y}/G$ by $G\lambda$, with stabilizer $G_\lambda$ for each $\lambda \in \mathbb{Y}$. Given that $\overline{\omega}(G\lambda)$ equals the product of the weight of $\lambda$ and $|G_\lambda|$, we have:
\[
\overline{\omega}(G\lambda) =
\begin{cases}
1, & \text{if } \lambda \notin \{\infty, 0, \pm1\}, \\
2, & \text{if } \lambda = \pm1, \\
n, & \text{if } \lambda \in \{\infty, 0\}.
\end{cases}
\]

Consequently, there exists an isomorphism $\varphi: \mathbb{Y}/G \to \mathbb{X}$ such that:

\[
\varphi(G\lambda) =
\begin{cases}
\infty, & \text{if } \lambda = 1, \\
0, & \text{if } \lambda = -1, \\
1, & \text{if } \lambda \in \{\infty, 0\}.
\end{cases}
\]
\end{rem}

Note that the dualizing element $\vec{x}_1-\vec{x}_2$ in $\mathbb{L}(2,2,n)$ has order $2$. We consider the cyclic group $\mathbb{Z}(\vec{x}_1-\vec{x}_2)$ of order $2$, which has a strict  $\mathbf{k}$-linear action on ${\rm coh}\mbox{-}\mathbb{X}$ by the degree-shift. We will consider the category $({\rm coh}\mbox{-}\mathbb{X})^{\mathbb{Z}(\vec{x}_1-\vec{x}_2)}$ of $\mathbb{Z}(\vec{x}_1-\vec{x}_2)$-equivariant sheaves.  We identify $G$ with the character group $\widehat{\mathbb{Z}(\vec{x}_1-\vec{x}_2)}$. Thus we obtain the dual $G$-action on $({\rm coh}\mbox{-}\mathbb{X})^{\mathbb{Z}(\vec{x}_1-\vec{x}_2)}$.

\begin{prop}\label{(2,2,n) and (n,n)}
Keep the notations as above. Then the following statements hold.
\begin{enumerate}
\item[(a)] There is an equivalence of categories
$$({\rm coh}\mbox{-}\mathbb{X})^{\mathbb{Z}(\vec{x}_1-\vec{x}_2)}\stackrel{\sim}\longrightarrow {\rm coh}\mbox{-}\mathbb{Y},$$
which is equivariant with respect to the above two $G$-actions.
    \item[(b)] There is an equivalence of categories
    $$ ({\rm coh}\mbox{-}\mathbb{Y})^{G} \stackrel{\sim} \longrightarrow {\rm coh}\mbox{-}\mathbb{X}.$$
\end{enumerate}
\end{prop}
\begin{proof}
The equivalence in (a) is obtained explicitly in \cite{MR4444429}, which will be denoted by  $H_1$. We observe that $H_1 \cdot{\rm Ind}_1 (\mathcal{O}_\mathbb{X})=\mathcal{O}_\mathbb{Y}$, where ${\rm Ind}_1\colon {\rm coh}\mbox{-}\mathbb{X} \rightarrow ({\rm coh}\mbox{-}\mathbb{X})^{\mathbb{Z}(\vec{x}_1-\vec{x}_2)}$ is the induction functor. In particular, the dual $G$-action fixes ${\rm Ind}_1(\mathcal{O}_\mathbb{X})$; see \cite[Subsection 4.3]{Chen2020DualActions}.

Transporting the dual $G$-action on $({\rm coh}\mbox{-}\mathbb{X})^{\mathbb{Z}(\vec{x}_1-\vec{x}_2)}$ via a quasi-inverse of $H_1$, we obtain the $G$-action on ${\rm coh}\mbox{-}\mathbb{Y}$. It follows that the transported action fixes $\mathcal{O}_\mathbb{Y}$, that is, given by an automorphism of $\mathbb{Y}$; see \cite[Proposition 3.1]{Lenzing2000AutomorphismGroup}.
We infer that the transported action swaps the simple sheaves associate to the weighted point $\infty$ and $0$ on $\mathbb{Y}$. Thus, the transported action coincides with the one determined  by $\sigma_{1,2}$; compare \cite[Proposition 3.1]{Lenzing2000AutomorphismGroup}.
 This proves (a).

The statement (b) follows from (a) and \cite[Theorem 4.6]{Chen2020DualActions}. 
\end{proof}

Denote the equivalence in Proposition \ref{(2,2,n) and (n,n)} as $H_2$. This equivalence $H_2$ induces a bijection $\overline{H_2}$ from  ${\rm ind}({\rm coh}\mbox{-}\mathbb{Y})^G$ to ${\rm ind}({\rm coh}\mbox{-}\mathbb{X})$. By \cite[Example 2.5]{Chen2017Equivariantization}, we can list the types of all indecomposable $G$-equivariant objects in $({\rm coh}\mbox{-}\mathbb{Y})^{G}$ in the following table: 
\begin{table}[H]\label{form}
\centering
\caption{Types of indecomposable $G$-equivariant objects in $({\rm coh}\mbox{-}\mathbb{Y})^{G}$}
\[
\begin{array}{|c|c|c|}
\hline
\text{Form} & \text{Properties of $Y$} & \text{Associated isomorphisms} \\
\hline
(Y, \alpha^-) & \text{indec.}, G\text{-stable} & \alpha_e^- = \alpha_{\sigma_{1,2}}^- = \text{id}_Y \\
\hline
(Y, \alpha^+) & \text{indec.}, G\text{-stable} & \alpha_e^+ = -\alpha_{\sigma_{1,2}}^+ = \text{id}_Y \\
\hline
(Y \oplus \sigma_{1,2}(Y), \alpha) & \text{indec., not } G\text{-stable} & 
\begin{aligned}
    \alpha_e & = \text{id}_{Y \oplus \sigma_{1,2}(Y)}, \\
    \alpha_{\sigma_{1,2}} & = {\scriptsize
    \renewcommand{\arraystretch}{0.5} 
\setlength{\arraycolsep}{1pt} 
\begin{pmatrix} 0 & \text{id}_{\sigma_{1,2}(Y)} \\ \text{id}_Y & 0 \end{pmatrix}}
\end{aligned} \\
\hline
\end{array}
\]
\end{table} 

To explicitly described the bijection $\overline{H_2}$, we provide a detailed description of the indecomposable sheaves in ${\rm coh}\mbox{-}\mathbb{X}$ and ${\rm coh}\mbox{-}\mathbb{Y}$:

On one hand, according to \cite[Corollary 3.8]{Lenzing2011WeightedProjective}, all indecomposable sheaves in ${\rm vect}\mbox{-}\mathbb{Y}$ are line bundles of the form $\mathcal{O}_{\mathbb{Y}}(\vec{y})$ for some $\vec{y} \in \mathbb{L}(n,n)$. We observe that the automorphism $\sigma_{1,2}$ exchanges $S_{\infty,\overline{i}}^{(j)}$ with $S_{0,\overline{i}}^{(j)}$ for any $i \in \mathbb{Z}$ and $j \in \mathbb{Z}_{\geq 1}$, and sends line bundles $\mathcal{O}_\mathbb{Y}(l_1\vec{y}_1 + l_2\vec{y}_2 + l\vec{c})$ to $\mathcal{O}_\mathbb{Y}(l_1\vec{y}_2 + l_2\vec{y}_1 + l\vec{c})$ for any $0 \leq l_i \leq n-1$, with $i=1,2$ and $l \in \mathbb{Z}$. Thus, indecomposable sheaves in ${\rm coh}\mbox{-}\mathbb{Y}$ are classified as follows:

\begin{itemize}
    \item Indecomposable $G$-stable sheaves:
    \begin{itemize}
        \item  $\mathcal{O}_\mathbb{Y}(l_1 \vec{y}_1 + l_2 \vec{y}_2 + l  \vec{c})$, where $l_1 = l_2$, $0\leq l_1, l_2\leq n-1$ and $l \in \mathbb{Z}$;
        \item  $S_{\lambda}^{(j)}$, where $\lambda = \pm 1$ and $j \in \mathbb{Z}_+$.
    \end{itemize}
    
    \item Indecomposable not $G$-stable sheaves:
    \begin{itemize}
        \item $\mathcal{O}_\mathbb{Y}(l_1 \vec{y}_1 + l_2 \vec{y}_2 + l  \vec{c})$, where $l_1 \neq l_2$, $0\leq l_1, l_2\leq n-1$ and $l \in \mathbb{Z}$;
        \item $S_{0,\overline{i}}^{(j)}$,  $S_{\infty,\overline{i}}^{(j)}$, $S_{\lambda}^{(j)}$, where $\lambda \neq \pm 1$, $i\in \mathbb{Z}$ and $j \in \mathbb{Z}_+$.
    \end{itemize}
\end{itemize}

On the other hand, for any line bundle $L\in {\rm coh}\mbox{-}\mathbb{X}$ and $\vec{x} \in \mathbb{L}(2,2,n)$ with $0 \leq \vec{x} \leq(n-2) \vec{x}_3$, we have 
\begin{equation}\label{dim of ext eqe 1}
{\rm Ext}^1_{\mathbb{X}}(L(\vec{x}),L(\vec{\omega}))\cong D{\rm Hom}_{\mathbb{X}}(L,L(\vec{x}))\cong \mathbf{k}. 
\end{equation}
Consequently, the  middle term $E$ in the non-split exact sequence
\begin{equation}\label{extension bundle}
    \begin{tikzcd}[ampersand replacement=\&,cramped,sep=small]
	0 \& {L(\vec{\omega})} \& E \& {L(\vec{x})} \& 0
	\arrow[from=1-2, to=1-3]
	\arrow[from=1-3, to=1-4]
	\arrow[from=1-4, to=1-5]
	\arrow[from=1-1, to=1-2]
\end{tikzcd}
\end{equation}
	is uniquely determined up to isomorphism. We refer to $E_L\langle\vec{x}\rangle:=E$ as the \emph{extension bundle} for $L$ and $\vec{x}$.  According to \cite{MR3028577}, each indecomposable bundle in ${\rm vect}\mbox{-}\mathbb{X}$ is either a line bundle or an extension bundle.

Then, for $i\in\mathbb{Z}$, $j\in\mathbb{Z}_+$, $1 \leq k \leq n-1$ , $\lambda \notin \{\pm 1, 0, \infty\} \subset \mathbb{Y}$ and  the isomorphism $\varphi$ in Remark \ref{isomor}, the bijection $\overline{H_2}: {\rm ind}({\rm coh}\mbox{-}\mathbb{Y})^G\xrightarrow[]{}{\rm ind}({\rm coh}\mbox{-}\mathbb{X})$ is described as follows:

\begin{table}[H]\label{H2}
\centering
\small
\caption{The bijection $\overline{H_2}$ induced by $H_2$}
\[\begin{tabular}{|c|c|}
\hline
\text{Isoclasses in ${\rm ind}({\rm coh}\mbox{-}\mathbb{Y})^G$} & \text{Corresponding elements in ${\rm ind}({\rm coh}\mbox{-}\mathbb{X})$} \\
\hline
$(\mathcal{O}_\mathbb{Y}(2i\vec{\omega} ), \alpha^-)$ & $\mathcal{O}_\mathbb{X}(2i\vec{\omega})$ \\
$(\mathcal{O}_\mathbb{Y}((2i+1)\vec{\omega} ), \alpha^+)$ & $\mathcal{O}_\mathbb{X}((2i+1)\vec{\omega})$ \\
$(\mathcal{O}_\mathbb{Y}(2i\vec{\omega} ), \alpha^+)$ & $\mathcal{O}_\mathbb{X}(\vec{x}_1 - \vec{x}_2 + 2i\vec{\omega})$ \\
$(\mathcal{O}_\mathbb{Y}((2i+1)\vec{\omega} ), \alpha^-)$ & $\mathcal{O}_\mathbb{X}(\vec{x}_1 - \vec{x}_2 + (2i+1)\vec{\omega})$ \\
$(\mathcal{O}_\mathbb{Y}( \vec{c} + 2i\vec{\omega} ), \alpha^-)$ & $\mathcal{O}_\mathbb{X}(\vec{x}_1 + 2i\vec{\omega})$ \\
$(\mathcal{O}_\mathbb{Y}( \vec{c} + (2i+1)\vec{\omega} ), \alpha^-)$ & $\mathcal{O}_\mathbb{X}(\vec{x}_1 + (2i+1)\vec{\omega})$ \\
$(\mathcal{O}_\mathbb{Y}( \vec{c} + 2i\vec{\omega} ), \alpha^+)$ & $\mathcal{O}_\mathbb{X}(\vec{x}_2 + 2i\vec{\omega})$ \\
$(\mathcal{O}_\mathbb{Y}( \vec{c} + (2i+1)\vec{\omega} ), \alpha^+)$ & $\mathcal{O}_\mathbb{X}(\vec{x}_2 + (2i+1)\vec{\omega})$ \\
$(\mathcal{O}_\mathbb{Y}(k\vec{y}_{1} + i\vec{\omega} ) \oplus \mathcal{O}_\mathbb{Y}(k\vec{y}_{2} + i\vec{\omega} ), \alpha)$ & $E_{\mathcal{O}_{\mathbb{X}}} \langle (k-1)\vec{x}_3 \rangle ((i-1)\vec{\omega})$ \\
$(S_{1}^{(j)}, \alpha^-)$ & $S_{\infty, \overline{i}}^{(j)}$ \\
$(S_{1}^{(j)}, \alpha^+)$ & $S_{\infty, \overline{i+1}}^{(j)}$ \\
$(S_{-1}^{(j)}, \alpha^-)$ & $S_{0, \overline{i}}^{(j)}$ \\
$(S_{-1}^{(j)}, \alpha^+)$ & $S_{0, \overline{i+1}}^{(j)}$ \\
$(S_{0, \overline{i}}^{(j)} \oplus S_{\infty, \overline{i}}^{(j)}, \alpha)$ & $S_{1, \overline{i}}^{(j)}$ \\
$(S_{\lambda}^{(j)} \oplus S_{\lambda^{-1}}^{(j)}, \alpha)$ & $S_{\varphi(G\lambda)}^{(j)}$ \\
\hline
\end{tabular}\]
\end{table}

\section{Flips of skew-arcs}\label{foh}
Recall from \cite{MR2448067} that the \emph{flip} of an arc $\gamma$ in a triangulation $\Gamma$ replaces $\gamma$ with the other diagonal $\gamma^{\prime}$ of the quadrilateral formed by the two triangles sharing $\gamma$, resulting in a new triangulation. We refer to $\gamma^{\prime}$ as the \emph{flip} of $\gamma$ with respect to $\Gamma$, and denote this operation as $\gamma^{\prime}= \mu_{\Gamma}(\gamma)$.

In this appendix, we define a similar operation for a skew-arc $\widehat{\gamma} $ in a pseudo-triangulation $\Lambda$. More precisely, the operation transforms $\widehat{\gamma}$ into a uniquely defined, different skew-arc $\widehat{\mu}_{\Lambda}(\widehat{\gamma})$, such that the pseudo-triangulation $\Lambda$ changes to another pseudo-triangulation  $\widehat{\mu}_{\widehat{\gamma}}(\Lambda) := \Lambda \setminus \{\widehat{\gamma}\} \cup \{\widehat{\mu}_{\Lambda}(\widehat{\gamma})\}$. For convenience, we refer to $\widehat{\mu}_{\Lambda}(\widehat{\gamma})$ as the \emph{flip} of $\widehat{\gamma}$ with respect to $\Lambda$. 

 For $\widehat{\gamma} \in \Lambda$, Lemma \ref{well-defined} allows us to define 
\[
\zeta(\Lambda \setminus \{\widehat{\gamma}\}) := \big(\zeta(\Lambda \setminus \{\widehat{\gamma}\})_1, \zeta(\Lambda \setminus \{\widehat{\gamma}\})_2\big),
\]
where
\begin{align*}
\zeta(\Lambda \setminus \{\widehat{\gamma}\})_i &= 
\begin{cases} 
\epsilon & \text{if } (\Lambda_{\times_i} \setminus \{\widehat{\gamma}\}) \subset \mathbf{C}_{\times_i}^\epsilon, \text{ where } \epsilon \in \{+, -\}; \\
\pm & \text{if } (\Lambda_{\times_i} \setminus \{\widehat{\gamma}\}) \cap \mathbf{C}_{\times_i}^+ \neq \emptyset \text{ and } (\Lambda_{\times_i} \setminus \{\widehat{\gamma}\}) \cap \mathbf{C}_{\times_i}^- \neq \emptyset.
\end{cases}
\end{align*}
\noindent  

 Whenever this can
be done without ambiguity, we shall consider curves in $\Gamma_\Lambda$ that have endpoints not in $P_\Lambda$ as curves on $(\mathcal{S}, M)$ by ignoring the points in $P_\Lambda$. Moreover,  skew-arcs in  $\Lambda_{\mathtt{b},\times}$ ($\times\in P_\Lambda$) can naturally be  viewed as curves  in  $\Gamma_\Lambda$, with one endpoint at $\times\in P_\Lambda$.  For $i\in \{1,2\}$, let $j\in \{1,2\}$ be the index different from $i$. Based on the type of $\widehat{\gamma}$ and the values of  $\zeta(\Lambda)$, and $\zeta(\Lambda \setminus \{\widehat{\gamma}\})$, all types of flips of skew-arcs are summarized as follows:

\begin{table}[H]
\centering
\footnotesize
\begin{tabular}{|c|c|p{5cm}|}
\hline
\multirow{2}{*}{\emph{Type}} & \multirow{2}{*}{Condition} & \multirow{2}{*}{The flip of $\widehat{\gamma}$ with respect to $\Lambda$} \\
                             &                            &                             \\ \hline
\emph{Type I} & $\widehat{\gamma} \in \Lambda_{\mathtt{b}, \times_i}$, $\mathbf{C}(\widehat{\gamma}) = \{\gamma\}$ &  \\ \hline
(1) & If $\zeta(\Lambda)_i = \pm(0)$, and $\mu_{\Gamma_\Lambda}(\gamma)$ is $\sigma$-fixed on $\times_i$ & 
$\widehat{\mu}_{\Lambda}(\widehat{\gamma}) = \mu_{\Gamma_\Lambda}(\gamma)^{\zeta(\Lambda \setminus \{\widehat{\gamma}\})_i} \in \Lambda_{\mathtt{b}, \times_i}$ \\ \hline
(2) & If $\zeta(\Lambda)_i = \pm(0)$, and $\mu_{\Gamma_\Lambda}(\gamma) = L_{\times_{j}}$  & 
$\widehat{\mu}_{\Lambda}(\widehat{\gamma}) = \zeta(\Lambda \setminus \{\widehat{\gamma}\}) \in \mathbf{C}_*$ \\ \hline
(3) & If $\zeta(\Lambda)_i \neq \pm(0)$ and $\mu_{\Gamma_\Lambda}(\widehat{\gamma}) = L_{\times_{j}}$ & 
$\widehat{\mu}_{\Lambda}(\widehat{\gamma}) = \begin{cases} 
(\neg\zeta(\Lambda)_1, \zeta(\Lambda)_2) & \text{if } i=1 \\
(\zeta(\Lambda)_1, \neg\zeta(\Lambda)_2) & \text{if } i=2 
\end{cases}$ \\ \hline
(4) & If $\zeta(\Lambda)_i \neq \pm(0)$ and $\mu_{\Gamma_\Lambda}(\widehat{\gamma})$ is not $\sigma$-fixed & 
$\widehat{\mu}_{\Lambda}(\widehat{\gamma})=\{\mu_{\Gamma_\Lambda}(\widehat{\gamma}), \mu_{\Gamma_\Lambda}(\sigma(\widehat{\gamma}))\}$ \\ \hline
(5) & If $\zeta(\Lambda)_i \neq \pm(0)$, and $\mu_{\Gamma_\Lambda}(\widehat{\gamma})\in\mathbf{C}_{\mathtt{b}, \times_{j}}$ & 
$\widehat{\mu}_{\Lambda}(\widehat{\gamma}) = \mu_{\Gamma_\Lambda}(\gamma)^{\zeta(\Lambda)_{j}}$ \\ \hline
(6) & If $\zeta(\Lambda)_i \neq \pm(0)$ and $\mu_{\Gamma_\Lambda}(\widehat{\gamma})$ is $\sigma$-fixed on $\times_i$ & 
$\widehat{\mu}_{\Lambda}(\widehat{\gamma}) = \mu_{\Gamma_\Lambda}(\widehat{\gamma})^{\neg \zeta(\Lambda)_i}$ \\ \hline
\emph{Type II} & $\widehat{\gamma} \in \Lambda_{\mathtt{b}}^\sigma \cup \Lambda_{\mathtt{p}}^\sigma$, $\mathbf{C}(\widehat{\gamma}) = \{\gamma, \sigma(\gamma)\}$ &  \\ \hline
(1) & If $\mu_{\Gamma_\Lambda}(\gamma) \in \mathbf{C}_{\mathtt{b}, \times_i}$ & 
$\widehat{\mu}_{\Lambda}(\widehat{\gamma})\in\{\mu_{\Gamma_\Lambda}(\gamma), \mu_{\Gamma_\Lambda}(\sigma(\gamma))\} \cap \mathbf{C}_{\mathtt{b}, \times_i}^{\zeta(\Lambda)_i}$ \\ \hline
(2) & If $\mu_{\Gamma_\Lambda}(\widehat{\gamma})$ is not $\sigma$-fixed & 
$\widehat{\mu}_{\Lambda}(\widehat{\gamma})=\{\mu_{\Gamma_\Lambda}(\widehat{\gamma}), \mu_{\Gamma_\Lambda}(\sigma(\widehat{\gamma}))\}$ \\ \hline
(3) & If $\{\mu_{\Gamma_\Lambda}(\gamma), \mu_{\Gamma_\Lambda}(\sigma(\gamma))\} = \{L_0, L_1\}$ & 
$\widehat{\mu}_{\Lambda}(\widehat{\gamma}) = (\zeta(\Lambda)_1, \zeta(\Lambda)_2) \in \mathbf{C}_*$ \\ \hline
\emph{Type III} & $\widehat{\gamma} \in \Lambda_*$, $\mathbf{C}(\widehat{\gamma}) = \{L_0, L_1\}$ &  \\ \hline
(1) & If $\mu_{\Gamma_\Lambda}(L_0)$ is not $\sigma$-fixed & 
$\widehat{\mu}_{\Lambda}(\widehat{\gamma})=\{\mu_{\Gamma_\Lambda}(L_0), \mu_{\Gamma_\Lambda}(L_1)\}$ \\ \hline
(2) & If $\mu_{\Gamma_\Lambda}(L_0) = \mu_{\Gamma_\Lambda}(\sigma(L_1))$ is $\sigma$-fixed on $\times_i$ & 
$\widehat{\mu}_{\Lambda}(\widehat{\gamma}) = \mu_{\Gamma_\Lambda}(L_0)^{\neg\zeta(\Lambda)_i}$ \\ \hline
(3) & If $\zeta(\Lambda)_i = \pm(i)$ for some $i \in \{1, 2\}$ & 
$\widehat{\mu}_{\Lambda}(\widehat{\gamma}) = \mu_{\Gamma_\Lambda}(L_{\times_{j}})^{\zeta(\Lambda \setminus \{\widehat{\gamma}\})_i}$ \\ \hline
\end{tabular}
\end{table}

\begin{rem}
One can observe that the uniqueness of the skew-arc $\widehat{\mu}_{\Lambda}(\widehat{\gamma})$ is determined by  the uniqueness of the flips of specific curves with respect to $\Gamma_\Lambda$, along with the values of $\zeta(\Lambda)$ and $\zeta(\Lambda \setminus \{\widehat{\gamma}\})$.
\end{rem}

To illustrate, we define a graph $\mathcal{G}(\mathcal{S}, M, \sigma)$ whose vertices represent all pseudo-triangulations on $(\mathcal{S}, M,\sigma)$. Two pseudo-triangulations are connected by an edge if and only if they differ by exactly one skew-arc. In the following example, we present the flips of some skew-arcs with respect to given pseudo-triangulations.

\begin{exm}
In this example, bold blue and red curves represent different boundaries of $(\mathcal{S}, M, \sigma)$, while orange and green points denote various $\times_i$. Rather than presenting entire pseudo-triangulations, we focus locally on the flips of skew-arcs.  We annotate certain edges of the full subgraph of $\mathcal{G}(\mathcal{S}, M,\sigma)$ with arrows to indicate the types of the two flips associated with each edge. We have the following full subgraphs of $\mathcal{G}(\mathcal{S}, M,\sigma)$: 
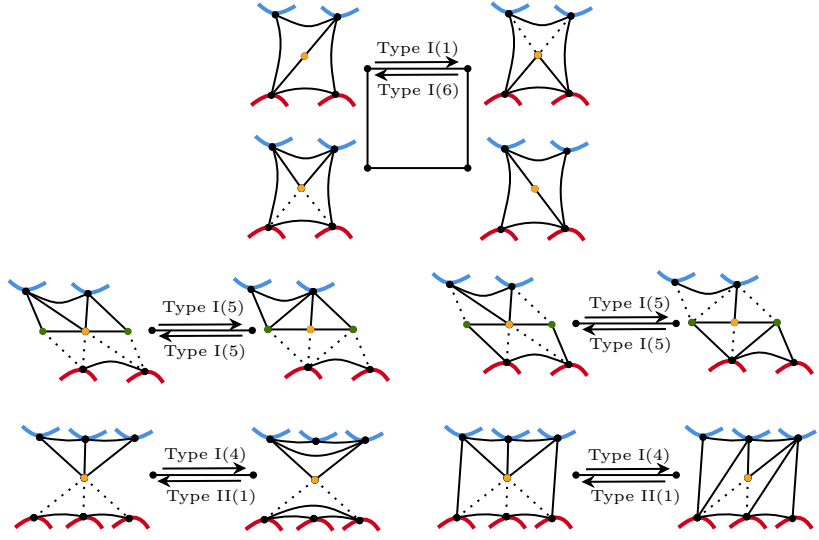
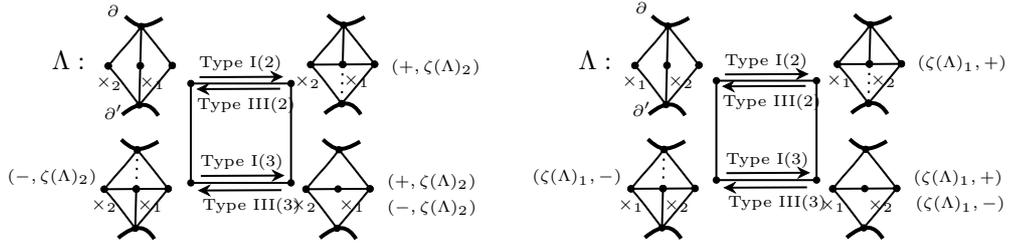
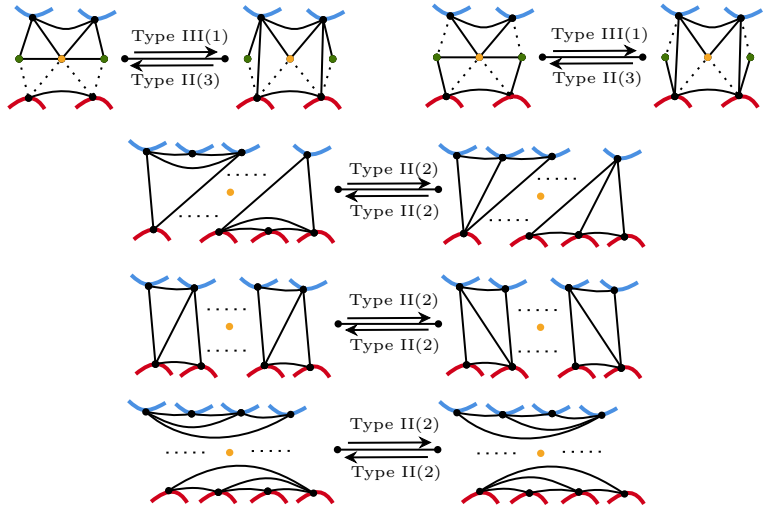
\begin{figure}[H]
    \centering

\tikzset{every picture/.style={line width=0.75pt}}  


\end{figure}
Moreover, we have
\begin{figure}[H]
    \centering

\tikzset{every picture/.style={line width=0.75pt}}  

%
\end{figure}

\end{exm}

 \noindent {\bf Acknowledgements.}\quad
 The authors would like to thank Shiquan Ruan for helpful discussions.  Jianmin Chen and Jinfeng Zhang were partially supported by the National Natural Science Foundation of China (Nos. 12371040, 12131018 and 12161141001).

\end{document}